\documentclass[a4paper,11pt]{amsart}
\usepackage{amsfonts,amssymb,amsmath,amsthm,abstract,color}
\usepackage[mathscr]{euscript}
\usepackage[ps,all,arc,rotate]{xy}
\usepackage[lmargin=1in,rmargin=1in,tmargin=1in,bmargin=1in]{geometry}
\usepackage{fancyhdr}
\usepackage{pb-diagram}
\usepackage{enumitem}
\usepackage{hyperref}
\usepackage[usenames,dvipsnames]{xcolor}
\usepackage{tikz-cd}
\usepackage{extpfeil}
\usepackage{float}
\hypersetup{colorlinks=true,citecolor=NavyBlue,linkcolor=Brown,urlcolor=Orange}

\newtheorem{prop}{Proposition}[section]
\newtheorem{lemma}[prop]{Lemma}
\newtheorem{thm}[prop]{Theorem}
\newtheorem{cor}[prop]{Corollary}

\theoremstyle{definition}
\newtheorem{defn}[prop]{Definition}

\newtheorem{rmk}[prop]{Remark}
\newtheorem{ex}[prop]{Example}

\renewcommand{\tilde}{\widetilde}     
\newcommand{\ra}{\rightarrow}

\DeclareMathOperator{\Bun}{Bun}

\DeclareMathOperator{\Jac}{Jac}
 
\DeclareMathOperator{\rk}{rk}     
\DeclareMathOperator{\Pic}{Pic} 
 
\DeclareMathOperator{\Spec}{Spec} 
\DeclareMathOperator{\Sym}{Sym}

\DeclareMathOperator{\Hom}{Hom}
\DeclareMathOperator{\ParHom}{ParHom}

\DeclareMathOperator{\Aut}{Aut}
\DeclareMathOperator{\id}{id}

\def\DM{\mathrm{DM}}

\def\CH{\mathrm{CH}} 
\def\CHM{\mathrm{CHM}}

\DeclareMathOperator{\eff}{eff}

\DeclareMathOperator{\SmProj}{SmProj}
\DeclareMathOperator{\codim}{codim}

\DeclareMathOperator{\ch}{ch}
\DeclareMathOperator{\td}{td}
\DeclareMathOperator{\pr}{pr}
\DeclareMathOperator{\loc}{loc}
\DeclareMathOperator{\im}{Im}
\def\ab{{\mathrm{ab}}}
\def\kim{{\mathrm{Kim}}}
\def\num{{\mathrm{num}}}
\def\even{{\mathrm{even}}}
\def\odd{{\mathrm{odd}}}

\def\cA{\mathcal A}\def\cC{\mathcal C}
\def\cE{\mathcal E}\def\cF{\mathcal F}
\def\cH{\mathcal H}\def\cI{\mathcal I}
\def\cJ{\mathcal J}\def\cK{\mathcal K}\def\cL{\mathcal L}
\def\cM{\mathcal M}\def\cN{\mathcal N}\def\cO{\mathcal O}
\def\cP{\mathcal P}

\def\CC{\mathbb C}

\def\GG{\mathbb G}
\def\LL{\mathbb L}
\def\NN{\mathbb N}
\def\PP{\mathbb P}\def\QQ{\mathbb Q}\def\RR{\mathbb R}

\def\ZZ{\mathbb Z}

\def\fh{\mathfrak h}

\makeatletter
\newcommand{\flip}[2][]{\ext@arrow 3359\leftrightarrowfill@@{#1}{#2}}
\def\rightarrowfill@@{\arrowfill@@\relax\relbar\rightarrow}
\def\leftarrowfill@@{\arrowfill@@\leftarrow\relbar\relax}
\def\leftrightarrowfill@@{\arrowfill@@\leftarrow\relbar\rightarrow}
\def\arrowfill@@#1#2#3#4{%
  $\m@th\thickmuskip0mu\medmuskip\thickmuskip\thinmuskip\thickmuskip
   \relax#4#1
   \xleaders\hbox{$#4#2$}\hfill
   #3$%
}
\makeatother

\title[Motives of moduli spaces of bundles]{Motives of moduli spaces of bundles on curves via variation of stability and flips}

\author{Lie Fu, Victoria Hoskins and Simon Pepin Lehalleur}

\thanks{L. F. is supported by the Radboud Excellence Initiative programme. S. P. L. is supported by The Netherlands Organisation for Scientific Research (NWO), under project number 613.001.752.}

\begin{document}

\maketitle

\begin{abstract}
We study the rational Chow motives of certain moduli spaces of vector bundles on a smooth projective curve with additional structure (such as a parabolic structure or Higgs field). In the parabolic case, these moduli spaces depend on a choice of stability condition given by weights; our approach is to use explicit descriptions of variation of this stability condition in terms of simple birational transformations (standard flips/flops and Mukai flops) for which we understand the variation of the Chow motives. For moduli spaces of parabolic vector bundles, we describe the change in motive under wall-crossings, and for moduli spaces of parabolic Higgs bundles, we show the motive does not change under wall-crossings. Furthermore, we prove a motivic analogue of a classical theorem of Harder and Narasimhan relating the rational cohomology of moduli spaces of vector bundles with and without fixed determinant. For rank $2$ vector bundles of odd degree, we obtain formulas for the rational Chow motives of moduli spaces of semistable vector bundles, moduli spaces of Higgs bundles and moduli spaces of parabolic (Higgs) bundles that are semistable with respect to a generic weight (all with and without fixed determinant).
\end{abstract}

\tableofcontents

\section{Introduction}

Let $C$ be a smooth projective geometrically connected curve of genus $g$ over a field $k$. Let $\cN = \cN_C(n,d)$ denote the moduli space of semistable vector bundles of rank $n$ and degree $d$ on $C$.  When $n$ and $d$ are coprime, $\cN$ is a smooth projective variety of dimension $n^2(g-1) +1$.

There has been a long history of work on the cohomological invariants of $\cN$. Inductive formulas for the Betti numbers of $\cN$ were obtained by Harder and Narasimhan \cite{HN} using the Weil conjectures and point counting over finite fields, by Atiyah and Bott \cite{AB} over $k = \CC$ using gauge theory, and by Bifet, Ghione and Letizia \cite{BGL} using more algebro-geometric methods. All three approaches essentially involve first describing the cohomology of the stack $\Bun_{n,d}$ of all vector bundles on $C$ and then inductively computing the cohomology of $\cN$ by performing a Harder--Narasimhan recursion. The approach of \cite{BGL} lead to a closed formula for the class of $\Bun_{n,d}$ in a dimensional completion of the Grothendieck ring of varieties \cite{BD} and in Voevodsky's triangulated category of motives over $k$ with $\QQ$-coefficients \cite{HPL_formula}. Furthermore, the ideas in \cite{BGL} were used by del Ba\~no to show that the Chow motive of $\cN$ lies in the tensor subcategory generated by the motive of the curve \cite{dB_motive_moduli_vb}.

More generally, for a smooth projective variety $X$ and a moduli space of sheaves of some kind (possibly with some additional structure) on $X$, there are several examples in which the motive of this moduli space lies in the tensor subcategory generated by the motive of $X$. This holds for the moduli space of stable Higgs bundles on $C$ of coprime rank and degree \cite{HPL_Higgs} and for certain moduli spaces of semistable sheaves on K3 and abelian surfaces (as well as for closely related spaces such as crepant resolutions, twisted and non-commutative analogues) \cite{Bulles,Salvator-Lie-Ziyu}.

\subsection{Qualitative results in arbitrary rank}

Our motivation for this paper was to provide more explicit descriptions of the Chow motives of certain moduli spaces of vector bundles on a smooth projective curve with additional structure. Our approach is to use concrete geometric descriptions of the birational transformations between moduli spaces under variation of stability \cite{BH, Thaddeus_VGIT} together with recent descriptions giving the change in Chow motives of smooth projective varieties under standard flips and flops and Mukai flops \cite{Jiang19, LLW-annals}. In fact, since we also want to work with moduli spaces of Higgs bundles on $C$, which are non-proper, but smooth if $n$ and $d$ are coprime, we first need to extend the last of these results to smooth quasi-projective varieties (\textit{cf.}\ Theorem \ref{thm:MotiveMukaiFlop}) using a local-to-global trick explained in Appendix \ref{appendix}. 

One can simplify the geometry by fixing a degree $d$ line bundle $\cL$ and studying the moduli space $\cN_{\cL} = \cN_{\cL,C}(n,d)$ of semistable vector bundles with determinant isomorphic to $\cL$. When $n$ and $d$ are coprime, $\cN_{\cL}$ is also smooth and projective and, moreover, is Fano \cite{DN} and rational \cite{KS}. Quite remarkably, Harder and Narasimhan \cite{HN} showed that, for any $\ell$ coprime to the characteristic of $k$, the $\ell$-adic cohomology of $\cN$ is the tensor product of that of $\cN_{\cL}$ and $\Jac(C)$. Our first result gives a motivic refinement of this classical theorem (see Theorem \ref{thm motive N and N fixed det}).

\begin{thm}\label{main thm N with and without fixed det}
Let $n$ and $d$ be coprime; then there is an isomorphism of rational Chow motives 
\[\fh(\cN) \simeq \fh(\cN_{\cL}) \otimes \fh(\Jac(C)).\]
\end{thm}

The proof involves first checking that $\fh(\cN)$ and $\fh(\cN_{\cL})$ lie in the tensor subcategory generated by $\fh(C)$, and thus are abelian motives; for $\cN$ this is a theorem of del Ba\~no \cite{dB_motive_moduli_vb} and for $\cN_{\cL}$, see Proposition \ref{prop:N-abelian}, which uses an argument of B\"ulles \cite{Bulles} and Beauville \cite{Beauville_diag}. We reduce to the case of a field of characteristic zero, then use the theorem of Harder and Narasimhan and the fact that the $\ell$-adic realisation is conservative on abelian geometric motives in characteristic $0$ \cite{wildeshaus}.

In the case of parabolic vector bundles, which are vector bundles with flags of a specified type at a finite number of parabolic points $D = \{p_1,\dots, p_N\}$ on $C$, there are various notions of stability, encoded by a set of parabolic weights $\alpha$, which give rise to different moduli spaces $\cN^\alpha =\cN^\alpha_{C,D}(n,d)$. For generic weights (i.e. where semistability and stability coincide), these moduli spaces are smooth and projective. The space of weights admits a wall and chamber decomposition by considering how the notion of (semi)stability changes and the birational transformation between moduli spaces separated by a wall can be explicitly described \cite{BH,BY_rationality,Thaddeus_VGIT}. In the nicest cases, the wall-crossing transformation is a standard flip or flop whose centres can be explicitly described as certain projective bundles over a product of smaller moduli spaces and the dimensions of the projective space fibres can be computed in terms of the dimension of a certain Ext group. Consequently, for such walls, one obtains an explicit description of how the Chow motive varies with $\alpha$ (Corollary \ref{cor motivic par WC}). We also provide a motivic description of flag degeneration (Corollary \ref{cor motives flag degen}). Furthermore, for coprime rank and degree and sufficiently small weights $\alpha$, there is a forgetful map $\cN^\alpha \ra \cN$ which is an iterated flag bundle \cite{BH,BY_rationality}. Hence, it suffices to know $\fh(\cN)$ to compute the Chow motives of $\fh(\cN^\alpha)$ for generic $\alpha$ via wall-crossing; for some explicit formulas in rank $n =2$, see $\S$\ref{sec intro rk 2}. In particular, $\fh(\cN^\alpha)$ depends on $\alpha$, as was already known for the Poincar\'{e} polynomial \cite{Bauer,Holla}.  As a corollary of these motivic formulas, we obtain descriptions of the Chow groups of 1-cycles (see Corollaries \ref{cor:CH1constant} and \ref{cor int Jac 1cycles par}) strengthening \cite{Chakraborty1,Chakraborty2} and for fixed $i$, a stabilisation result for $\CH^i(\cN^\alpha)$ (see Corollary \ref{cor stable chow gps parabolic}).  

The algebraic symplectic analogues of moduli spaces of (parabolic) vector bundles are the moduli spaces of Higgs bundles $\cM$ and parabolic Higgs bundles $\cM^{\alpha}$, which are no longer proper, but are smooth when semistability and stability coincide. There is a $\GG_m$-action on $\cM$ given by scaling the Higgs field \cite{Hitchin, Simpson} which gives an associated Bia{\l}ynicki-Birula decomposition \cite{BB} that enables one to see that the cohomology of $\cM$ is nonetheless pure. In \cite[Corollary 6.9]{HPL_Higgs}, the second and third authors show that the Voevodsky motive of $\cM$ is a Chow motive by this method. The same holds for moduli spaces of parabolic Higgs bundles which are semistable with respect to a generic weight (\textit{cf.}\ Lemma \ref{lemma:par Higgs pure}). Moreover, via the Bia{\l}ynicki-Birula decomposition, the Chow motive of $\cM$ (resp. $\cM^{\alpha}$) can be expressed in terms of $\cN$ (resp. $\cN^{\alpha}$) and moduli spaces of chains (resp. parabolic chains); see \cite{HPL_Higgs} for the non-parabolic case.

There are several surprising changes as we pass from (parabolic) vector bundles to (parabolic) Higgs bundles. First, the analogue of Theorem \ref{main thm N with and without fixed det} does not hold if we replace $\cN$ by $\cM$; this was already seen on cohomology in rank $n = 2$ by Hitchin \cite{Hitchin}. Moreover, for a general curve $C$ of genus at least 2, the motive of $\cM_{\cL}$ is not contained in the subcategory generated by $\fh(C)$ (see Proposition \ref{prop:EnlargeCategory}). Second, although the Chow motive of $\cN^\alpha$ depends on $\alpha$, this is no longer the case for $\cM^\alpha$. 

\begin{thm}\label{main thm par higgs} (Corollary \ref{cor motive par higgs indept of alpha})
For a generic weight $\alpha$, the integral Chow motive of the moduli space $\cM^\alpha$ of $\alpha$-semistable parabolic Higgs bundles of rank $n$ and degree $d$ is independent of $\alpha$.
\end{thm}

On the level of Poincar\'{e} polynomials, this result was known in rank $2$ with fixed determinant over $k = \CC$ \cite{BY_ParHiggs} (in this case, the moduli spaces are diffeomorphic by \cite{NakajimaPHB}). Our proof relies on Thaddeus' description \cite{Thaddeus_var_PHiggs} of wall-crossings as Mukai flops, which can be thought of as algebraic symplectic versions of standard flops and do not alter the motive by Theorem \ref{thm:MotiveMukaiFlop}.

\subsection{Explicit formulas in rank 2 and odd degree}\label{sec intro rk 2}

In order to compute Chow motives of these moduli spaces, the basic building block is $\fh(\cN)$. By Theorem \ref{main thm N with and without fixed det}, it suffices to know $\fh(\cN_{\cL})$. In rank $n = 2$ and odd degree $d$, we compute this by using variation of stability for pairs consisting of a vector bundle and a section studied by Thaddeus \cite{Thaddeus_pairs, Thaddeus_VGIT} and improving work of del Ba\~{n}o in \cite{dB_rk2} on the motive of $\cN_{\cL}$ in a semisimple category of motives by using the fact that $\fh(\cN_{\cL})$ is abelian (\textit{cf.}\ Proposition \ref{prop:N-abelian}) and hence Kimura finite-dimensional. Consequently, we obtain corresponding decompositions of the Chow groups of these moduli spaces (see $\S$\ref{sec cor Chow gps N} for corollaries and comparisons with previous results for $\cN_{\cL}$).

\begin{thm}\label{main thm rk 2 rational}
For $n = 2$, $d$ odd and $\cL \in \Pic^d(C)(k)$, the rational Chow motive of the moduli space $\cN_{\cL}$ of semistable bundles with determinant $\cL$ is
\[ \fh(\cN_{\cL}) \simeq \fh(\Sym^{g-1}(C))(g-1) \oplus \bigoplus_{i=0}^{g-2} \fh(\Sym^i(C))\otimes \left( \QQ(i) \oplus \QQ(3g-3-2i) \right). \]
\end{thm}

We obtain the following formulas for the integral Chow motives of the other bundle moduli spaces in terms of $\fh(\cN)$ or $\fh(\cN_{\cL})$. For parabolic vector bundles, this comes from studying the wall-crossing flips/flops. For (parabolic) Higgs bundles, this come from using Bia{\l}ynicki-Birula decompositions associated to $\GG_m$-actions scaling the Higgs field and Theorem \ref{main thm par higgs}.

\begin{thm}\label{main thm rk 2 integral}
For $n = 2$, $d$ odd and $\cL \in \Pic^d(C)(k)$, we have the following formulas for the integral Chow motives of associated bundle moduli spaces.
\begin{enumerate}[label=\emph{(\roman*)}, leftmargin=0.7cm] 
\item (Theorem \ref{thm par motive rk 2 odd deg}) For the moduli space $\cN^\alpha$ of parabolic bundles on $C$ with flags at $N$ parabolic points which are semistable with respect to a generic weight $\alpha$, we have
\[ \fh(\cN^{\alpha})  \simeq \fh(\cN) \otimes \fh(\PP^1)^{\otimes N} \oplus \bigoplus_{j=0}^{N-3} \fh(\Jac(C))^{\otimes 2}(g+j)^{\oplus b_j(\alpha)}.\]
where the exponents $b_j(\alpha)$ are given in Definition \ref{def exponents in terms of alpha}.
\item (Corollary \ref{cor par motive rk 2 odd det}) For the moduli space $\cN^\alpha_{\cL}$ of parabolic bundles on $C$ with determinant $\cL$ and flags at $N$ parabolic points which are semistable with respect to a generic weight $\alpha$, we have
\[ \fh(\cN^{\alpha}_{\cL})  \simeq \fh(\cN_{\cL}) \otimes \fh(\PP^1)^{\otimes N} \oplus \bigoplus_{j=0}^{N-3} \fh(\Jac(C))(g+j)^{\oplus b_j(\alpha)}\]
where the exponents $b_j(\alpha)$ are as above.
\item (Theorem \ref{thm motive rk 2 Higgs}) For the moduli space $\cM$ of semistable Higgs bundles, we have
\[ \fh(\cM) \simeq \fh(\cN(2,d)) \oplus \bigoplus_{j=1}^{g-1} \fh(\Pic^{a_{d,j}}(C)) \otimes \fh(\Sym^{2j-1}(C))(3g - 2j -2 ), \]
where $a_{d,j} = g-j+(d-1)/2$.
\item (Proposition \ref{prop motive rk 2 higgs fixed det}) For the moduli space $\cM_{\cL}$ of semistable Higgs bundles with determinant $\cL$, we have
\[ \fh(\cM_{\cL}) \simeq \fh(\cN_{\cL}) \oplus \bigoplus_{j=1}^{g-1} \fh(\widetilde{\Sym}^{2j-1}(C))(3g - 2j -2 )\]
where $\widetilde{\Sym}^{i}(C) \ra {\Sym}^{i}(C)$ is the pull back of the multiplication-by-2 map on $\Jac(C)$. 
\item (Theorem \ref{thm motive rk 2 para higgs}) For the moduli space $\cM^\alpha$ of parabolic Higgs bundles on $C$ with flags at $N$ parabolic points which are semistable with respect to a generic weight $\alpha$, we have
\[ \fh(\cM^\alpha) \simeq \fh(\cN) \otimes \fh(\PP^1)^{\otimes N} \oplus \bigoplus_{\begin{smallmatrix}  0 \leq l \leq N  \\ \frac{l+1-N}{2} \leq j \leq g -1 \end{smallmatrix}} \fh(\Pic^{a_{d,j}}(C)) \otimes \fh(\Sym^{2j+N - l-1}(C))(3g -2j+l-2)^{\oplus {N \choose l }} \]
where $a_{d,j} := g-j+(d-1)/2$. As in Theorem \ref{main thm par higgs}, this is independent of $\alpha$.
\item (Proposition \ref{prop par higgs rk 2 fixed det}) For the moduli space $\cM^\alpha_{\cL}$ of parabolic Higgs bundles on $C$ with determinant $\cL$ and flags at $N$ parabolic points which are semistable with respect to a generic weight $\alpha$, we have
\[ \fh(\cM^\alpha_{\cL}) \simeq \fh(\cN_{\cL}) \otimes \fh(\PP^1)^{\otimes N} \oplus \bigoplus_{\begin{smallmatrix}  0 \leq l \leq N  \\ \frac{l+1-N}{2} \leq j \leq g -1 \end{smallmatrix}}  \fh(\widetilde{\Sym}^{2j+N - l-1}(C))(3g -2j+l-2)^{\oplus {N \choose l }} \] 
where $\widetilde{\Sym}^{i}(C) \ra {\Sym}^{i}(C)$ is the pull back of the multiplication-by-2 map on $\Jac(C)$.
\end{enumerate}
\end{thm}

\section{Background on motives}
\label{sec motives}

Let $k$ be a field and $R$ a commutative ring.

\subsection{Chow motives}

We write $\CHM(k,R)$ for the category of Chow motives over $k$ with coefficients in $R$. This is of course a very classical object, but since the definition is short and there are two possible conventions in the literature, we provide some details. Objects in $\CHM(k,R)$ are triples $(X,p,n)$, where $X$ is a smooth projective $k$-variety, $p$ is a projector in the ring $\CH^{*}(X\times X, R)$ of correspondences up to rational equivalence (with product given by composition of correspondences) and $n\in\mathbb{Z}$. We adopt the standard notations: $\fh(X)=(X,\id_{X},0)$ and $R(n)=(\Spec(k),\id_{\Spec(k)},n)$. The category $\CHM(k,R)$ admits a symmetric monoidal structure coming from products in $\SmProj_{k}$. Morphisms in $\CHM(k,R)$ are given, following an homological convention, by
\[
\CHM(k,R)((X,p,m),(Y,q,n))=q\circ \CH^{\dim(Y)+n-m}(X\times Y,R) \circ p
\]
(where $\dim(Y)$ has to be interpreted as a locally constant function if $Y$ is reducible). With this convention, there is a covariant functor $\fh:\SmProj_{k}\to \CHM(k,R)$.

By definition, we get the following, which we state for later reference.

\begin{lemma}\label{lem chow vanish}
Let $X$ be a smooth projective variety. Then $\CH^{i}(X,R)\simeq \Hom(\fh(X),R(i))$ (and thus vanishes if $i<0$ or $i>\dim(X)$).
\end{lemma}  

A Chow motive $M\in\CHM(k,\QQ)$ is Kimura finite-dimensional, in the sense of \cite{Kimura-FiniteDimension} (see also \cite{AndreBourbaki}), if there exists a decomposition $M=M_{\even}\oplus M_{\odd}$ and  $n_{\even},n_{\odd}\in \NN$ such that $\Lambda^{n_{\even}}(M_{\even})=0$ and $\Sym^{n_{\odd}}(M_{\odd})=0$. Let $\CHM(k,\QQ)^\kim$ denote the full subcategory of Kimura finite-dimensional motives. Kimura finite-dimensionality has some very pleasant categorical consequences, one of which we now recall.

\begin{prop}[see {\cite{AndreBourbaki}}]\label{prop kimura canc} 
The category $\CHM(k,\QQ)^{\kim}$ is closed under taking direct sums, tensor products, direct summands and duals. The functor $\CHM(k,\QQ)^\kim \ra M_\num(k,\QQ)$ to the category of motives for numerical equivalence is full and conservative.   
\end{prop}    

In particular, because $M_{\num}(k,\QQ)$ is abelian semi-simple by \cite{Jannsen}, $\CHM(k,\QQ)^{\kim}$ satisfies cancellation: if there exists an isomorphism $M\oplus P\simeq N\oplus P$, then $M\simeq N$.

Kimura and O'Sullivan conjectured that every Chow motive is Kimura finite-dimensional. While this is wide open in general, we still have the following result for abelian motives which turns out to cover all the motives encountered in this paper. The category $\CHM(k,\QQ)^{\ab}$ of abelian motives is the thick subcategory generated by motives of abelian varieties, or equivalently, the thick tensor subcategory generated by motives of curves (see \cite[Proposition 4.5, Theorem 5.2]{SchollClassical}). 

\begin{prop}[{\cite[Th\'eor\`eme 2.8]{AndreBourbaki}}]\label{prop kimura ab}
Abelian motives are Kimura finite-dimensional.
\end{prop}  

\subsection{Voevodsky motives}

Some of the moduli spaces we are interested in, namely moduli spaces of (parabolic) Higgs bundles, are not smooth projective varieties, but only smooth and quasi-projective. The natural way to associate a motive to them is via the triangulated category $\DM(k,R)$ of Voevodsky motives over $k$ with coefficients in $R$. It turns out that, in all the cases considered in this paper, the resulting motive is actually a Chow motive, where we identify $\CHM(k,R)$ with a full subcategory of $\DM(k,R)$ via the fundamental embedding theorem of Voevodsky \cite{VoevodskyBookChapter}, which we recall now in the form we need.

\begin{thm}\label{thm:embedding}
Assume that $k$ is a perfect field or that the characteristic of $k$ is invertible in $R$. Then the ``motive'' functor
\[
M:\SmProj_{k}\to \DM(k,R)
\]
factors through the category of Chow motives and induces a fully faithful embedding
\[
\CHM(k,R)\hookrightarrow \DM(k,R).
\]
\end{thm}
\begin{proof}
  This result follows directly (by passing to idempotent completions) from the following formula for Hom groups in $\DM(k,R)$; for smooth projective varieties $X,Y$ of pure dimensions $d,e$, we have an isomorphism
  \[
\DM(k,R)(M(X),M(Y))\simeq \CH^{e}(X\times_{k}Y,R)
\]
which is natural in $X$ and $Y$. Let us assume first that $k$ is perfect. The formula then follows from the combination of the following results, which all only require $k$ perfect: duality for smooth projective varieties in the category $\DM^{\eff}(k,R)$ of effective Voevodsky motives (see \cite[Appendix B]{HuberKahn} for a reference which only requires $k$ perfect), the representability of Chow groups in $\DM^{\eff}(k,R)$ \cite[Proposition 14.16 and Corollary 19.2]{MVW}, and Voevodsky's cancellation theorem saying that $\DM^{\eff}(k,R)\to \DM(k,R)$ is fully faithful \cite{VoevodskyCancellation}.

Let us now assume that $p$ is invertible in $R$. Let $k^{\mathrm{perf}}/k$ be a perfect closure of $k$. Under the assumption on $R$, the base change functor $\DM(k,R)\to \DM(k^{\mathrm{perf}},R)$ is an equivalence of categories \cite[Proposition 8.1]{cisinski-deglise-integral}, while the Chow groups also do not change under purely inseparable extensions. This implies the general case.
\end{proof}

\section{Motives under flips and flops}

In the context of Minimal Model Program (see \cite{KollarMori}), flips and flops are among the basic building blocks for birational transforms between algebraic varieties of dimension greater than 2. We recall here the geometry of two particularly elementary instances, namely standard flips and flops, and Mukai flops\footnote{In the literature, standard flips and flops are also called elementary, ordinary or Atiyah flips and flops; Mukai flops also go under the name of elementary transforms, especially in the context of hyper-K\"ahler geometry.}, and present some results on the corresponding change of motives, which will be applied later to birational transformations between various moduli spaces of stable vector bundles on curves. 
\subsection{Standard flips and flops}

\begin{defn}[Standard flips and flops]
	\label{def:Flip}
	Let $S$ be a smooth $k$-variety. Let $m$ and $l$ be two positive integers.  Let $V$ be a vector bundle on $S$ of rank $m+1$ and $\varpi: Z=\PP(V)\to S$ the associated projective bundle. 
	Let $X$ be a smooth variety containing $Z$ as a closed subvariety such that the restriction of the normal bundle $N_{Z/X}$ to each fibre of $\varpi$ is isomorphic to $\mathcal{O}_{\PP^m}(-1)^{\oplus l+1}$, or equivalently (see \cite[Lemma 1.1]{LLW-annals}), $N_{Z/X}\simeq \mathcal{O}_\varpi(-1)\otimes \varpi^*(V')$ for some vector bundle $V'$ of rank $l+1$ on $S$. 
	
	Let $\tau\colon\tilde X\to X$ be the blow-up of~$X$ along the smooth center~$Z$; then the exceptional divisor~$E$ is isomorphic to~$\PP(V)\times_{S}\PP(V')$ with normal bundle~$N_{E/\tilde X}\simeq\mathcal{O}(-1,-1)$. Therefore we can contract~$E$ along the other direction, $E\to Z':=\mathbb{P}(V')$, which identifies~$E$ as the projectivisation of the vector bundle~$\varpi^{\prime*}(V')$, where $\varpi'\colon Z'\to S$ is the natural projection. The contraction of $E$ to $Z'$ amounts to blowing down $\tilde X$ to a new (smooth) variety $X'$, and the normal bundle becomes $N_{Z'/X'}\simeq\mathcal{O}_{\varpi'}(-1)\otimes\varpi^{\prime *}V$. We call the induced birational transform $\phi: X\dashrightarrow X'$ a \textit{standard flip} \textit{of type} $(m,l)$ \textit{with centre} $S$. When $m=l$, it is called a \textit{standard flop}. In terms of the $K$-partial order $\geq_K$ on the birational class of $X$, for a standard flip as above, we have $X \geq_K X'$ if and only if $m \geq l$. 
	
In this paper, we assume that $X$ is quasi-projective. Note that the flipped variety $X'$ could fail to be quasi-projective in general. We assume further that there exists a small extremal contraction (in the sense of Minimal Model Program \cite{KollarMori}) $X\to \bar{X}$ whose restriction to $Z$ is $\varpi$. Then $X'$ is quasi-projective (see \cite[Proposition 4.2]{Kanemitsu-K-equivalence} and  \cite[Proposition 1.3]{LLW-annals}). 
	We summarise the situation in the following commutative diagram.
	\begin{equation}
		\label{equation:flip}
	\begin{tikzcd}
	& & E=Z\times_S Z' \arrow[d, hook, "j"] \arrow[ddll, swap, "p"] \arrow[ddrr, "p'"] \\
	& & \tilde{X} \arrow[dl, swap, "\tau"] \arrow[dr, "\tau'"] \\
	Z \arrow[r, hook, "i"] \arrow[rrdd, swap, "\varpi"] & X \arrow[dashed, rr, "\phi"] \arrow[rd]& & X'  \arrow[ld] & Z' \arrow[l, hook', swap, "i'"] \arrow[ddll, "\varpi'" ] \\
	&& \bar{X}\\
	& & S \arrow[u, hook]
	\end{tikzcd}
	\end{equation}	
\end{defn}

If all the varieties are projective, the relation between the Chow motives of $X$ and $X'$ are established by Lee--Lin--Wang \cite[Theorem 2.1]{LLW-annals} in the case of flops ($m=l$) and by   Jiang \cite[Theorem 3.4.(3), Corollary 3.8]{Jiang19} in general. 

\begin{thm}\label{thm Jiang Chow flip}
	Let $\phi\colon X \dashrightarrow X'$ be a standard flip of type $(m,l)$ with centre $S$  between smooth projective varieties, as in Definition \ref{def:Flip}. Suppose $m \geq l$, then there is an isomorphism of integral Chow motives:
	\[(\tau_*{\tau'}^{*}, \sum_j i_*\circ (-\cdot c_1(\mathcal{O}_\varpi(1))^j)\circ\varpi^* )\colon \fh(X') \oplus \bigoplus_{j=l+1}^m \fh(S)(j) \xrightarrow{\simeq} \fh(X). \]
	In particular, when $\phi$ is a standard flop ($m=l$), we have an isomorphism $\tau'_*\tau^*\colon \fh(X)\xrightarrow{\simeq} \fh(X')$, with inverse $\tau_*{\tau'}^{*}$.
\end{thm}


\subsection{Mukai flops}
\label{subsec:MukaiFlop}
The so-called Mukai flop was discovered by Mukai \cite{Mukai84Invent} as a typical instance of birational transforms between hyper-K\"ahler manifolds (see also \cite[\S 3]{Huybrechts97JDG}). The following definition is a mild generalisation of the classical one\footnote{Our definition is called a twisted Mukai flop in \cite{LLW-annals}.}, which is enough for the purpose of this paper. 
\begin{defn}[Mukai flops]
		\label{def:MukaiFlop}
	Let $S$ be a smooth variety and let $m$ be a positive integer.  Let $V$ be a vector bundle on $S$ of rank $m+1$ and $\varpi: Z=\PP(V)\to S$ the associated projective bundle. 
	Let $X$ be a smooth variety containing $Z$ as a closed subvariety such that 
 	\[N_{Z/X}\simeq \Omega_{Z/S}\otimes \varpi^*(L),\]
 	for some line bundle $L$ on $S$, where $\Omega_{Z/X}$ is the relative cotangent bundle of $\varpi: Z\to S$.
	
	Let $\tau\colon\tilde X\to X$ be the blow-up of~$X$ along the smooth center~$Z$, then the exceptional divisor $E=\PP(\Omega_{Z/S})$, which is isomorphic to the incidence hypersurface inside $\PP(V)\times_{S}\PP(V^\vee)$, has normal bundle $N_{E/\tilde X}\simeq\mathcal{O}(-1,-1)$. 
	By  contracting $E$ along the other direction $E\to Z':=\mathbb{P}(V^\vee)$, we blow down $\tilde X$ to a new (smooth) variety $X'$. We call the induced birational transform $\phi: X\dashrightarrow X'$ a \textit{Mukai flop} of type $m$ with centre $S$. A significant difference compared to standard flops is that the fibre product $X\times_{\bar{X}}X'$ has two irreducible components of the same dimension, namely, $\tilde{X}$ and $Z\times_SZ'$, intersecting along $E$.
	
	As in Definition \ref{def:Flip}, in order to stay in the category of quasi-projective varieities, we assume in this paper that there exists a small extremal contraction in the sense of Minimal Model Program $X\to \bar{X}$ whose restriction to $Z$ is $\varpi$. The situation is summarised in the following diagram.
	
	\begin{equation}
	\label{diag:MukaiFlop}
	\begin{tikzcd}
	&& Z\times_S Z' \arrow[dddll, swap, "\pr_1"] \arrow[dddrr, "\pr_2"] \\
	& & E \arrow[d, hook, "j"] \arrow[ddll, "p"] \arrow[ddrr, swap, "p'"] \arrow[u, hook]\\
	& & \tilde{X} \arrow[dl, swap, "\tau"] \arrow[dr, "\tau'"] \\
	Z \arrow[r, hook, "i"] \arrow[rrdd, swap, "\varpi"] & X \arrow[dashed, rr, "\phi"] \arrow[rd]& & X'  \arrow[ld] & Z' \arrow[l, hook', swap, "i'"] \arrow[ddll, "\varpi'"] \\
	&& \bar{X}\\
	& & S \arrow[u, hook]
	\end{tikzcd}
	\end{equation}
\end{defn}

In the projective situation, Lee--Lin--Wang \cite[Theorem 6.3]{LLW-annals} proved the invariance of Chow motives under a Mukai flop; see also \cite{FuWang08} for a generalisation to \textit{stratified Mukai flops}.
\begin{thm}
	\label{thm:LLWMukai}
	Let $\phi\colon X\dashrightarrow X''$ be a Mukai flop between smooth projective varieties, as in Definition \ref{def:MukaiFlop}, then the cycle given by the fibre product $[X\times_{\bar{X}}X']=[\tilde{X}]+[Z\times_SZ']$ induces an isomorphism of integral Chow motives:
	\[\tau'_*\tau^*+i'_*\pr_{2,*}\pr_1^*i^*\colon \fh(X)\xrightarrow{\simeq} \fh(X'),\]
	whose inverse is given by the same cycle. 
\end{thm}

\begin{rmk}\label{rmk LLW holds for S nonproper}
In fact, if $X$ and $X'$ are both smooth varieties which are defined and proper over $S$, but $S$ itself is not necessarily proper, then the arguments in \cite{LLW-annals} give a corresponding isomorphism of Chow groups $\CH^*(X)\xrightarrow{\simeq} \CH^*(X')$.
\end{rmk}

Later when dealing with parabolic Higgs bundles in \S\ref{subsec:ParabolicHiggs}, we will need the following slight generalisation of Theorem \ref{thm:LLWMukai} to treat certain smooth quasi-projective varieties. 
\begin{thm} 
	\label{thm:MotiveMukaiFlop}
	Let $\phi\colon X\dashrightarrow X'$ be a Mukai flop between smooth varieties. Using the notation in Diagram \eqref{diag:MukaiFlop}, there is an isomorphism of Chow groups 
	\[\tau'_*\tau^*+i'_*\pr_{2,*}\pr_1^*i^*\colon \CH^*(X)\xrightarrow{\simeq} \CH^*(X'),\]
	whose inverse is given by $\tau_*{\tau'}^*+i_*\pr_{1,*}\pr_2^*{i'}^*$. Here all the pull-backs are Gysin homomorphisms as in \cite{FultonBook}.
	Assuming that the Voevodsky motives of $X$ and $X'$ can be identified with Chow motives (via Theorem \ref{thm:embedding}; in particular $k$ is perfect or we invert the characteristic), then these isomorphisms give isomorphisms between $\fh(X)$ and $\fh(X')$.
\end{thm}
\begin{proof}
	Let $X_{\loc}=\PP_Z(N_{Z/X}\oplus \mathcal{O}_Z) \stackrel{\pi}{\ra} Z$ and $X'_{\loc}=\PP_{Z'}(N_{Z'/X'}\oplus \mathcal{O}_{Z'})\stackrel{\pi'}{\ra} Z'$ which contain $Z$ and $Z'$ respectively via the zero sections.
	Then there is a Mukai flop $\phi_{\loc}\colon X_{\loc}\dashrightarrow X'_{\loc}$ as in Definition \ref{def:MukaiFlop}, which should be considered as the local model of $\phi$.
	 Let $\iota$ and $\iota'$ be the inclusions of $E$ into $X_{\loc}$ and $X'_{\loc}$ as the infinite parts respectively (see Appendix \ref{appendix}). Consider the following diagram
\begin{equation}
\label{diag:ChowFlop}
	\begin{tikzcd}
	\CH^{l-1}(E) \arrow[d, equal, "\id"] \arrow[r, hook, "{(j_*, -\iota_*)}"]& \CH^l(\tilde{X})\oplus \CH^l(X_{\loc}) \arrow[d, "{(\id,\mathcal{F}_{\loc})}"]\arrow[r, two heads, "{(j^*, -\iota^*)}"]& \CH^k(E)\arrow[d, equal, "\id"]\\
	\CH^{l-1}(E)  \arrow[r, hook, "{(j_*, -\iota'_*)}"]& \CH^l(\tilde{X})\oplus \CH^l(X'_{\loc}) \arrow[r, two heads, "{(j^*, -{\iota'}^*)}"]& \CH^l(E),
	\end{tikzcd}
\end{equation}
where $\mathcal{F}_{\loc}$ is the isomorphism obtained by  applying Theorem \ref{thm:LLWMukai} to the Mukai flop $\phi_{\loc}$ (in fact, we apply Remark \ref{rmk LLW holds for S nonproper} as $S$ is not necessarily proper, but $X_{\loc}$ and $X'_{\loc}$ are both defined and proper over $S$). Note that Diagram \eqref{diag:ChowFlop} is indeed commutative, as $\mathcal{F}_{\loc}$ restricts to identity outside of $Z$ and $Z'$. By the local-to-global trick recalled in  Proposition \ref{prop:LocalGlobal}, $\CH^l(X)$ (resp.~$\CH^l(X')$) is isomorphic, via an explicit correspondence, to the middle cohomology of the top (resp.~bottom) row of  Diagram \eqref{diag:ChowFlop}. Hence, we conclude that the composition 
\begin{equation*}
	\mathcal{F}:=(\tau'_*, i'_*\pi'_*)\circ (\id, \mathcal{F}_{\loc})\circ (\tau^*, \pi^*i^*)=\tau'_*\tau^*+i'_*\pr_{2,*}\pr_1^*i^*
\end{equation*}
 induces an isomorphism between $\CH^*(X)$ and $\CH^*(X')$, with inverse given by the composition 
 \begin{equation*}
 \mathcal{F}^{-1}=(\tau_*, i_*\pi_*)\circ (\id, \mathcal{F}_{\loc}^{-1})\circ ({\tau'}^*, {\pi'}^*{i}'^*)=\tau_*{\tau'}^*+i_*\pr_{1,*}\pr_2^*{i'}^*.
 \end{equation*}
 Finally, with the additional hypothesis that the motives of $X$ and $X'$ are Chow motives, we obtain the isomorphism between $\fh(X)$ and $\fh(X')$ by Manin's identity principle, since for any smooth projective variety $T$, the product $\phi\times \id\colon X\times T\dashrightarrow X'\times T$ is again a Mukai flop.
\end{proof}

\section{Motives of moduli spaces of stable vector bundles}

Let $\cN = \cN_C(n,d)$ denote the moduli space of semistable vector bundles of rank $n$ and degree $d$ on a smooth geometrically connected genus $g$ curve $C$ over $k$. We recall that a vector bundle $E$ is semistable if for all proper subbundles $F \subset E$, we have $\mu(F) \leq \mu(E)$, where $\mu(E):= \deg(E)/\rk(E)$. It is stable if this equality is strict, and geometrically stable\footnote{Over an algebraically closed field $k$, the notions of stability and geometric stability coincide.} if it is stable after all field extensions. The moduli space $\cN$ is a projective $k$-variety constructed by GIT \cite{Seshadri_VBAC} containing the moduli space $\cN^s$ of geometrically stable bundles as an open subset.

In this section, we assume that $n$ and $d$ are coprime, so that semistability, stability and geometric stability for rank $n$ degree $d$ vector bundles coincide and so $\cN$ is a smooth projective $k$-variety of dimension $n^2(g-1) + 1$. For $g=0$, this moduli space is empty unless $n = 1$ (in which case it is a point) and for $g =1$, we have $\cN \simeq \Jac(C)$ when $k=\bar{k}$ \cite[Theorem 7]{Atiyah_Elliptic}. The moduli space $\cN$ admits a universal family $\cE$, as $n$ and $d$ are coprime \cite{Ramanan}; that is, $\cE$ is a vector bundle on $ C \times \cN$ such that for every geometric point $p$ of $\cN$, the bundle $\cE_{|C \times \{p\}}$ is a stable vector bundle in the isomorphism class specified by $p$.

Let $\cL \in \Pic^d(C)(k)$ and consider the moduli space $\cN_{\cL}=\cN_{\cL}(n)$ of semistable vector bundles with determinant isomorphic to $\cL$. As $n$ and $d$ are coprime, $\cN_{\cL}$ is a smooth projective variety admitting a universal family. In this section, we study the Chow motives $\fh(\cN)$ and $\fh(\cN_{\cL})$. 

\subsection{Beauville's diagonal trick}

Given a moduli space $\cM$ of sheaves of some kind on a smooth projective variety $X$, there are several cases in which the motive of $\cM$ lies in the tensor subcategory generated by the motive of $X$. In the case of the Chow motive of $\cN$, this was shown by del Ba\~{n}o in \cite{dB_motive_moduli_vb}. Other cases where this general principle is known to hold include the moduli space of stable Higgs bundles on $C$ \cite{HPL_Higgs} and certain moduli spaces of (semistable) sheaves on K3 and abelian surfaces (and closely related spaces: crepant resolutions, twisted and non-commutative analogues) \cite{Bulles,Salvator-Lie-Ziyu}. However, in $\S$\ref{sec higgs rank 2 fixed det}, we will see that this is not the case for moduli spaces of stable Higgs bundles with fixed determinant on a general curve $C$, and that motives of certain finite covers of the curve are necessary to generate the motive of the moduli space in that case (Proposition \ref{prop:EnlargeCategory}).

Based on an idea of Beauville \cite{Beauville_diag} using Chern classes of the universal family to describe the diagonal of $\cN$, B\"ulles \cite{Bulles} observed that one can give a short proof of del Ba\~{n}o's theorem. We review the argument and check that it gives a similar result for $\cN_{\cL}$.

\begin{prop}\label{prop:N-abelian}
Assume that $n$ and $d$ are coprime and $\cL$ is a line bundle of degree $d$ on $C$. Then the rational Chow motives $\fh(\cN)$ and $\fh(\cN_{\cL})$ both lie in the tensor subcategory of $\CHM(k,\QQ)$ generated by $\fh(C)$ (i.e. appear as direct summands of the Chow motive of a large enough power of $C$).
\end{prop}
\begin{proof}
Since the argument is the same in both cases, we give the proof for $\cN_{\cL}$ as this case has not explicitly appeared in the literature. We follow closely the exposition of \cite[Theorem 0.1]{Bulles}. First, consider the natural projections 
\[
    \xymatrix{
      & \cN_{\cL}\times C \times \cN_{\cL} \ar[ld]_{\pi_{12}} \ar[d]_{\pi} \ar[rd]^{\pi_{23}} & \\
      \cN_{\cL} \times C & \cN_{\cL}\times \cN_{\cL} & C\times \cN_{\cL}.
    }
\]
Since $n$ and $d$ are coprime, there is a universal family on $\cN_{\cL}$ given by a vector bundle $\cE$ on $C \times \cN_{\cL}$. Following Beauville \cite{Beauville_diag}, let us define a K-theory class
\[
[\mathcal{E}\mathrm{xt}^{!}_{\pi}]:= \sum_{i}(-1)^{i}[R^{i}\pi_{*} \mathcal{H}\mathrm{om}(\pi_{12}^{*}\cE,\pi_{23}^{*}\cE)] \in K_{0}(\cN_{\cL}\times \cN_{\cL}).
\]
As $C$ is 1-dimensional, this sum has only two potentially non-zero terms, for $i=0$ or $1$, and it is possible to represent this class by a two-term complex $u:K^{0}\to K^{1}$ of locally free sheaves on $\cN_{\cL}\times \cN_{\cL}$. For stable vector bundles $E$ and $F$ on $C$, the homomorphism group $\mathrm{Hom}(E,F)$ vanishes unless $E\simeq F$, and is one-dimensional if $E\simeq F$. From this, Beauville deduces that the diagonal $\Delta_{\cN_{\cL}}$ is equal to the degeneracy locus of $u$ (at least set-theoretically, which is enough to carry out the rest of the proof) \cite[p.28]{Beauville_diag}. This degeneracy locus is a determinantal subvariety of the expected codimension and by applying Porteous's formula \cite[Theorem 14.4]{FultonBook} one obtains
\begin{equation}
  \label{Porteous}
c_{N}(-[\mathcal{E}\mathrm{xt}^{!}_{\pi}])=[\Delta_{\cN_{\cL}}]\in \CH^{N}(\cN_{\cL}\times \cN_{\cL})_\QQ,
\end{equation}
where $N=\dim \cN_{\cL} = (n^2-1)(g-1)$. 

The Chow group $\CH^{*}(\cN_{\cL}\times \cN_{\cL})_\QQ$ has a ring structure given by convolution of cycles, and the class $[\Delta_{\cN_{\cL}}]$ of the diagonal is a two-sided unit. Following B\"{u}lles, we introduce
\[
I:=\langle  \beta\circ \alpha|\alpha\in \CH^{*}(\cN_{\cL}\times C^{k})_\QQ, \beta \in\CH^{*}(C^{k}\times \cN_{\cL})_\QQ, k\geq 1\rangle
  \]
which is clearly a two-sided ideal of $\CH^{*}(\cN_{\cL}\times \cN_{\cL})_\QQ$. The same computation as in \cite[Proof of Theorem 0.1]{Bulles} shows that $I$ is also closed under intersection products.

Let us show that \eqref{Porteous} implies that $[\Delta_{\cN_{\cL}}]\in I$. As in \cite[Proof of Theorem 0.1]{Bulles}, the Grothendieck--Riemann--Roch Theorem implies that 
\[
\ch(-[\mathcal{E}\mathrm{xt}^{!}_{\pi}])= -\pi_{*}(\pi^{*}_{12} \ch(\cE^{\vee})\cdot \pi^{*}_{23}(\cE) \cdot \pi_{2}^{*}\td(C))
  \]
  with $\cE^{\vee}=\mathcal{H}\mathrm{om}(\cE,\cO)$. Let $\alpha:=\ch(\cE^{\vee})\cdot \pi^{*}_{2}\sqrt{\td{C}}$ and $\beta:=\ch(\cE)\cdot \pi^{*}_{2}\sqrt{\td{C}}$. Then, for $a \in \NN$, by considering the graded part of the previous equation in codimension $a$, we see that
  \[
\ch_{a}(-[\mathcal{E}\mathrm{xt}^{!}_{\pi}])=-\sum_{i+j=a+2}\alpha^{i}\circ \beta^{j} \in I.
\]
An induction on $a$ then implies that $c_{a}(-[\mathcal{E}\mathrm{xt}^{!}_{\pi}])$ also lies in $I$. By \eqref{Porteous}, we see that $[\Delta_{\cN_{\cL}}]\in I$. By the argument at the end of \cite[Proof of Theorem 0.1]{Bulles}, we then see that $\fh(\cN_{\cL})$ can be realised as a direct summand of a motive of the form $\bigoplus_{i}\fh(C^{k_{i}})(n_{i})$. This last motive is in the tensor subcategory of $\CHM(k,\QQ)$ generate by $\fh(C)$, and can also be embedded as a direct summand of the motive of a large enough power of $C$. This completes the proof.
\end{proof}

\subsection{Motives of moduli spaces of vector bundles with and without fixed determinant}

\begin{thm}\label{thm motive N and N fixed det}
Assume that $n$ and $d$ are coprime and $\cL \in \Pic^d(C)(k)$. In $\CHM(k,\QQ)$, we have
\[\fh(\cN) \simeq \fh(\cN_{\cL}) \otimes \fh(\Jac(C)).\]
\end{thm}
\begin{proof}
Since $n$ and $d$ are coprime, the moduli spaces $\cN$ and $\cN_{\cL}$ are fine moduli spaces and we fix universal families $\cE$ (resp. $\cE'$) on $\cN\times C$ (resp. $\cN_{\cL}\times C$). One can choose $\cE'$ such that the induced  morphism $\iota:\cN_{\cL}\to \cN$ satisfies $(\iota\times\id)^{*}(\cE)\simeq \cE'$. The morphism $\iota$ is easily seen to be a proper monomorphism (for instance by checking the valuative criterion of properness), hence a closed immersion. Combining this with the action of $\Jac(C)$ on $\cN$ by tensor product, we get a morphism
\[
\phi : \cN_{\cL}\times \Jac(C)\to \cN, \quad (E,L)\mapsto \iota(E)\otimes L.
\]
We will show that the map of Chow motives induced by this last morphism is an isomorphism.

First, let us show that we can reduce to the case of a field $k$ of characteristic zero. Let $B=\mathrm{Spec}(R)$ be a complete local trait with $R$ a discrete valuation ring with fraction field $K$ of characteristic $0$ and residue field $k$. Since $C$ is a smooth projective curve, there exists a smooth projective curve $\cC/B$ which lifts $C$ \cite[Exposé III Théorème 7.3]{SGA1}. Moreover, since $C$ has dimension $1$, we have $H^{2}(C,\cO_{C})=0$; hence there exists a line bundle $\Lambda$ on $\cC$ which lifts $\cL$ \cite[Corollary 5.6.(a)]{illusie-fga}. There is a relative moduli space $\cN_{\cC/B}$ (resp. $\cN_{\cC/B,\Lambda}$) of vector bundles of rank $n$ and degree $d$ (resp. determinant $\Lambda$) over $\cC/B$, which are smooth projective schemes over $B$ whose generic and special fibres are the corresponding moduli spaces over $K$ and $k$ (see \cite[Theorem 4.3.7]{HL} for a very general construction of fine moduli spaces of sheaves in a relative setting). In particular, we have a relative Jacobian $\Jac(\cC/B)=\cN_{\cC/B}(1,0)$. The construction of the previous paragraph works in this generality and produces a $B$-morphism
\[
\phi_{\cC} : \cN_{\cC/B,\Lambda}\times_B \Jac(\cC/B)\to \cN_{\cC/B}.
\]
Assume now that we know that, over the generic fibre, this induces a isomorphism of Chow motives $\fh(\phi_{\cC_K}): \fh(\cN_{\cC_K,\Lambda_K}\times \Jac(\cC_{K}))\to \fh(\cN_{\cC_K})$. By applying the specialisation morphisms for Chow groups in smooth projective families over $B$ \cite[\S 10.1]{FultonBook}, which are compatible with composition of correspondences, we see that this implies the same claim over $k$.

We can thus assume that $k$ is of characteristic $0$. Let $L/k$ be a finite field extension. The base change functor $\CHM(k,\QQ)\to \CHM(L,\QQ)$ is conservative by a trace argument. Since the characteristic of $k$ is $0$, by replacing $k$ by a finite extension, we can assume that the group scheme $\Jac(C)[n]$ of $n$-torsion points in the Jacobian is the constant finite étale group scheme associated to $\Gamma_{n}:=\Jac(C)[n](k)$, and we identify the two.

The group $\Gamma_{n}$ acts on $\cN_{\cL}$ via $M \cdot E := E\otimes M^{-1}$ and on $\Jac(C)$ by translation; we let $\cN_{\cL}\times^{\Gamma_{n}} \Jac(C)$ denote the quotient by the diagonal $\Gamma_n$-action. The morphism $\phi$ factors through this quotient and induces a morphism
\[
\overline{\phi}: \cN_{\cL}\times^{\Gamma_{n}} \Jac(C) \to \cN.
\]
It is well known that this morphism is an isomorphism, but we could not find a reference in this generality; let us sketch the argument. It is enough to prove that this is an isomorphism after extending the field $k$, hence we can assume that $k$ is algebraically closed. We construct an inverse in the other direction as follows. Let $S$ be a $k$-scheme and $[\cF\to C\times S]\in \cN(S)$ be the class of a family of rank $n$ degree $d$ stable vector bundles parametrised by $S$ (recall that $\cF$ and $\cF\otimes \pi_{S}^{*}\Lambda$ where $\Lambda$ is a line bundle on $S$ determine the same $S$-point of $\cN$). Then $\det(\cF)\otimes \pi_C^*\cL^{-1}\to C\times S$ is a family of line bundles of degree $0$ and in particular defines an $S$-point of $\Jac(C)$. This family of line bundles does not necessarily admit an $n$-th root, but it does after passing to the finite étale cover $S'\to S$ defined by the pullback square
\[
\xymatrix@C=4em{
  S'\ar[r]^{M} \ar[d]_{p} & \Jac(C) \ar[d]^{[n]} \\
  S \ar[r]^{\det(\cF)\otimes \cL^{-1}} & \Jac(C).
  }
\]
By construction, $S'$ is a $\Gamma_{n}$-torsor over $S$ and the induced map $S' \ra \Jac(C)$ gives us a family $M$ of degree $0$ line bundles on $C$ parametrised by $S'$ (this is the point where it is useful to assume $k$ is algebraically closed), uniquely determined up to the pullback of a line bundle on $S'$. We have $\det(p_C^*\cF\otimes M^{-1})\simeq p_C^*\det(\cF)\otimes M^{\otimes{-n}}\simeq \pi_C^*\cL\otimes \pi_{S'}^{*}\Lambda$ for $p_C := \mathrm{Id}_C \times p$ and $\pi_C : C \times S' \ra C$, and $\Lambda$ a line bundle on $S'$. Hence, the pair $(p_{C}^{*}\cF\otimes M^{-1},M)$ defines a morphism $S^{'}\to\cN_{L}\times \Jac(C)$, which is $\Gamma_{n}$-equivariant and descends to a morphism $S\to \cN_{L}\times^{\Gamma_{n}}\Jac(C)$ between the $\Gamma_n$-quotients. One checks that the resulting morphism does not depend on the choice of $\cF$ and $M$ in their equivalence class. The whole construction is functorial in $S$, and defines a morphism $\cN\to \cN_{L}\times^{\Gamma_{n}}\Jac(C)$, which one can show is an inverse to $\overline{\phi}$.

Since we are working with Chow motives with rational coefficients, we deduce that
\[
\fh(\cN)\simeq  \fh(\cN_{\cL}\times^{\Gamma_n} \Jac(C))\simeq \fh(\cN_{\cL}\times \Jac(C))^{\Gamma_{n}}\simeq (\fh(\cN_{\cL})\otimes \fh(\Jac(C)))^{\Gamma_{n}}.
\]
Let us write $\fh(\Jac(C))=\fh(\Jac(C))^{\Gamma_{n}}\oplus R$. Since the morphism $[n]:\Jac(C)\to \Jac(C)$ is a finite étale $\Gamma_{n}$-torsor, it induces an isomorphism of motives $[n]^{*}:\fh(\Jac(C))\simeq \fh(\Jac(C))^{\Gamma_{n}}$. Moreover, $\fh(\Jac(C))^{\Gamma_{n}}$ is Kimura finite-dimensional, as it is a direct factor of $\fh(\Jac(C))$; thus, by Proposition \ref{prop kimura canc}, we can cancel on both sides and deduce that $R\simeq 0$. We deduce that the $\Gamma_{n}$-action on $\fh(\Jac(C))$ is trivial (see \cite[Lemma 2.1]{JiangYin} for a different argument using results of Beauville on Chow groups of abelian varieties \cite{Beauville_AV}). We deduce that
\[
\fh(\cN)\simeq  (\fh(\cN_{\cL}))^{\Gamma_{n}}\otimes \fh(\Jac(C)).
\]
Finally, it remains to show that the action of $\Gamma_{n}$ on $\fh(\cN_{\cL})$ is trivial. A classical theorem of Harder--Narasimhan \cite{HN} shows that this is the case for the $\Gamma_{n}$-action on the $\ell$-adic cohomology $H^{*}(\cN_{\cL},\QQ_{\ell})$ for any $\ell$ prime to the characteristic of $k$. In other words, the $\ell$-adic realisation of the morphism $\fh(\cN_{\cL})^{\Gamma_{n}}\to \fh(\cN_{\cL})$ is an isomorphism. By Proposition \ref{prop:N-abelian}, this is a morphism between abelian Chow motives. Since $k$ is now a field of characteristic $0$, the $\ell$-adic realisation functor $R_{\ell}:\DM(k,\QQ)\to D(\QQ_{\ell})$ is conservative when restricted to abelian geometric motives by \cite[Theorem 1.12]{wildeshaus}. We deduce that  $\fh(\cN_{\cL})^{\Gamma_{n}}\to \fh(\cN_{\cL})$ is an isomorphism, which concludes the proof.
\end{proof}

\subsection{A closed formula in rank 2 using wall-crossing for stable pairs}

In this section, we obtain a formula for the rational Chow motive of the moduli space of stable vector bundles of rank two and odd degree $d$ from a result of del B\~{a}no \cite{dB_rk2}. Del B\~{a}no's proof relies in turn on work of Thaddeus \cite{Thaddeus_pairs} involving variation of stability for so-called (Bradlow) stable pairs, consisting of a vector bundle with a non-zero section. The notion of (semi)stability for such pairs depends on a stability parameter $\sigma \in \QQ_{>0}$ and involves checking an inequality for all subbundles. For each $\sigma$, there is an associated moduli space of $\sigma$-semistable pairs $\cP^\sigma =\cP^\sigma_C(2,d)$, which is a projective variety constructed via GIT \cite{Thaddeus_pairs}. The space of stability parameters $\QQ_{>0}$ admits a wall and chamber decomposition by considering how the notion of (semi)stability changes as $\sigma$ varies. The walls correspond to critical values of $\sigma$ for which semistability and stability do not coincide (i.e., there is a subbundle violating stability) and in the chambers, semistability and stability coincide and the corresponding moduli space is smooth. For moduli spaces $\cP^\sigma_{\cL} =\cP^\sigma_{C,\cL}(2,d)$ of pairs with fixed determinant, Thaddeus showed (i) for $\sigma$ in the minimal chamber, there is a forgetful map $\cP^{\sigma}_{\cL} \ra \cN_{\cL}$ (which is a projective bundle if $d \geq 4g -3$); (ii) each wall-crossing corresponds to standard flip (or flop) between the $\cP^{\sigma}_{\cL}$'s on both sides with centre a symmetric power of $C$ and iii) for $\sigma$ in the maximal non-empty chamber, the moduli space of stable pairs is a projective space. 

In \cite{Thaddeus_pairs}, the field $k$ is assumed to be algebraically closed; however we claim that the results hold over a general field. Indeed, when $k=\bar{k}$, for a given pair $(E,s)$ in the exceptional locus of a wall-crossing, there is a unique line subbundle $L\subset E$ which does not violate stability on one side of the wall but does on the other \cite[(1.4)]{Thaddeus_pairs}. By descent, this implies that $L$ is defined over $k$ when $E$ is, and the description of the moduli spaces on the walls is valid over $k$.

\begin{thm}\label{thm motive N rank 2}
Let $\cL$ be a line bundle on $C$ of odd degree $d$. Then the rational Chow motive of the moduli space $\cN_{\cL}=\cN_{C,\cL}(2,d)$ of stable rank $2$ bundles with fixed determinant $\cL$ is given by\footnote{The formula holds for all $g\geq 0$, with the convention that $\Sym^{-1}(C)=\emptyset$.}
\[\fh(\cN_{\cL}) \simeq \fh(\Sym^{g-1}(C))(g-1) \oplus \bigoplus_{i=0}^{g-2} \fh(\Sym^i(C))\otimes \left( \QQ(i) \oplus \QQ(3g-3-2i) \right) \]
and the rational Chow motive of the moduli space $\cN=\cN_C(2,d)$ of stable rank $2$ degree $d$ vector bundles is given by 
\[\fh(\cN)\simeq \fh(\Jac(C)) \otimes \left(\fh(\Sym^{g-1}(C))(g-1) \oplus \bigoplus_{i=0}^{g-2} \fh(\Sym^i(C))\otimes \left( \QQ(i) \oplus \QQ(3g-3-2i)\right)\right). \]
\end{thm}
\begin{proof}
  For $g=0$, both moduli spaces are empty and both formulas hold. For $g=1$, the determinant induces a morphism $\cN\to \Pic^{d}(C)\simeq \Jac(C)$ (where we use $\Pic^{1}(C)(k)\neq \emptyset$). We claim that the induced morphism of Chow motives is an isomorphism. It suffices to check this after passing to the algebraic closure $k=\bar{k}$. In that case, the determinant morphism is actually an isomorphism and $\cN_{\cL}$ is a point by \cite[Theorem 7]{Atiyah_Elliptic}. Since $\Pic^{1}(C)(k)\neq \emptyset$, we also have $\fh(C)\simeq \fh(\Jac(C))$ in $\CHM(k,\QQ)$ (even if $C(k)=\emptyset$). 

By Theorem \ref{thm motive N and N fixed det},  it suffices to prove the formula for $\fh(\cN_{\cL})$. By assumption, $C$ has a degree $d$ line bundle $\cL$ and tensoring with $\cL^{\otimes e}$ induces an isomorphism $\cN_C(2,d) \cong \cN_C(2,(2e+1)d)$; thus we can assume without loss of generality that $d \geq 4g-3$ and $g \geq 2$. Let us write $d=4g-3+2\delta$ with $\delta\geq 0$.

Since $d \geq  4g-3$, Thaddeus' work \cite{Thaddeus_pairs} shows there are $\frac{d-1}{2}=2g-2+\delta$ walls and $2g-1 + \delta$ chambers for pairs with corresponding smooth pairs moduli spaces $\cP^{\sigma_i}_{\cL}$ for $0 \leq i \leq 2g -2+\delta$ with $\sigma_0 > \cdots > \sigma_{2g-2+\delta} >0$ which fit into a diagram
	\begin{equation*}
		 \xymatrix{	\cP^{\sigma_0}_{\cL} \ar[d]_{\simeq} \ar@{-->}[rr]^{(0,5g -7+2\delta)}_{\text{centre } C} &  & \cP^{\sigma_1}_{\cL}  \ar@{-->}[rr]^{(1, 5g -9+2\delta)}_{\begin{smallmatrix}\text{centre } \\ \Sym^2(C) \end{smallmatrix}} & & \cP^{\sigma_2}_{\cL} \ar@{-->}[r] & \: \cdots \: \ar@{-->}[r]   & \cP^{{\sigma_{2g-3+\delta}}}_{\cL} \: \: \ar@{-->}[rr]^{(2g-3+\delta,g -1+2\delta)}_{\begin{smallmatrix} \text{centre } \\ \Sym^{2g-2+\delta}(C) \end{smallmatrix}} & & \: \: \cP^{{\sigma_{2g-2+\delta}}}_{\cL} \ar[d]_{\pi} \\
			 \PP^{5g-5+2\delta} & & &&& & & & \cN_{\cL} 
		}
	\end{equation*}
where the horizontal maps are standard flips (or flops) with the given type and centre and $\pi$ is the forgetful map, which is a $\PP^{2g-2+2\delta}$-bundle.

Let us first assume that $C$ admits a line bundle of degree $1$ and that $d=4g-3$ (which in this case can be achieved by tensoring by that line bundle). Under this assumption, using the above sequence of flips with $\delta=0$, del Ba\~{n}o computes the class of the motive of $\cN_{\cL}(2,d)$ in the completion of the ring $\cK$ of $K_{0}(\CHM_{k}^{\eff})$ along the ideal generated by the Lefschetz motive $\LL:=\QQ(1)$ \cite[Corollary 2.6]{dB_rk2} in terms of $\fh^1(C)$.

In general, suppose that $C$ does not admit a line bundle of degree $1$ and that $\delta$ is not necessarily $0$. We claim that the proof of \cite[Corollary 2.6]{dB_rk2} still applies in this case with some minor modifications. The assumption that $C$ has a line bundle of degree $1$ is not used in the proof, besides the reduction to degree $d=4g-3$ (in particular, as del Ba\~{n}o observes in \cite[\S 1.2.4]{dB_rk2}, the formula for motives with rational coefficients of high enough symmetric powers in terms of the motive of the Jacobian holds without any assumption on $C$, despite the fact that those symmetric powers are not projective bundles over the Jacobian). The computation for $\delta>0$ is then completely parallel to the computation for $\delta=0$. Let us just explain why the end result does not depend on $\delta$. The computation shows that the class of the motive of $\cN_{\cL}$ in the ring $\cK$ is given by applying the homomorphism $\cK[[T]]\to\cK,T\mapsto \LL$ to the expression
\[
\frac{1}{1-T^{2g-1+2\delta}}\left[ \frac{(1-T^{2g-1+2\delta})(1+T)^{\fh^{1}(C)}}{(1-T)(1-T^{2})}-\frac{(1+1)^{\fh^{1}(C)}}{(1-\LL)}\left(\frac{T^{2g-1+\delta}-T^{2g-1+\delta}}{1-T}-\frac{T^{3g-1+2\delta}-T^{g}}{1-T^{2}}\right) \right]
\]
and we see that the contributions of $\delta$ cancel out to give the same expression as in the case $\delta=0$.

At this point, we have a formula in the ring $\cK$. Del Ba\~{n}o then claims in \cite[Corollary 2.7]{dB_rk2} that this implies a formula in the category $M_{\num}^{\eff}(k,\QQ)$ of effective numerical motives with rational coefficients. Since he does not give a complete proof and this is a crucial step in our argument, let us fill in the details. Write $\cK_{\num}$ for the completion of $K_{0}(M_{\num}^{\eff}(k,\QQ))$ along the ideal generated by $\LL$. The formula in $\cK$ implies the same formula in $\cK_{\num}$. Let us show that the map $K_{0}(M_{\num}^{\eff}(k,\QQ))\to \cK_{\num}$ is injective. For this, it is enough to show that the filtration $(\LL^{n})_{n\geq 0}$ is separated. Jannsen's theorem \cite{Jannsen} states that $M_{\num}^{\eff}(k,\QQ)$ is an abelian semisimple category, so $K_{0}(M_{\num}^{\eff}(k,\QQ))$ is the free abelian group generated by isomorphism classes of objects of $M_{\num}^{\eff}(k,\QQ)$. The ring $K_{0}(M_{\num}^{\eff}(k,\QQ))$ is idempotent-complete, so the ideal $(\LL^{n})$ is idempotent-complete for every $n\geq 0$. It is thus enough to show that a simple object $M$ in $\cap_{n\geq 0}(\LL^{n})$ is zero, or equivalently that for any $N\in M_{\num}(k,\QQ)$, we have $\Hom(N,M)=0$. Writing $M=M'\otimes\LL^{n}$, we see that this group vanishes for $n$ large enough for dimensional reasons. Since both sides of \cite[Corollary 2.7]{dB_rk2} lie in $K_{0}(M_{\num}^{\eff}(k,\QQ))$, we get a formula in $K_{0}(M_{\num}^{\eff}(k,\QQ))$. Finally, since $M_{\num}^{\eff}(k,\QQ)$ is semisimple, the formula of \cite[Corollary 2.7]{dB_rk2} holds for isomorphism classes of objects in $M_{\num}^{\eff}(k,\QQ)$.

Recall that the motive $\fh(C)$ decomposes as $\fh(C)=\QQ(0)\oplus \fh^{1}(C)\oplus \QQ(1)$. By using this decomposition and expanding the formula in our Theorem, we see that it is equivalent to del Ba\~{n}o's formula. We have thus proven our theorem in $M_{\num}(k,\QQ)$.

By Theorem \ref{prop:N-abelian}, both sides of the formula are abelian Chow motives. By Propositions \ref{prop kimura ab} and \ref{prop kimura canc}, we deduce that the formula also holds in $\CHM(k,\QQ)$. This concludes the proof. \end{proof}

We note that as we used Kimura finite-dimensionality, the above isomorphism is not explicit.

\subsubsection{Corollaries and comparisons with previous results}\label{sec cor Chow gps N}

From Theorem \ref{thm motive N rank 2}, one obtains the following description of the Chow groups of $\cN_{\cL}$ for $n=2$ and $d$ odd, when $\Pic^1(C)(k)\neq \emptyset$. We assume that $g\geq 2$ for simplicity and to avoid some case distinctions.

\begin{cor}\label{cor:general formula Chow}
Let $a\in \NN$. There is an isomorphism
\[
\CH^{a}(\cN_{\cL})_\QQ\simeq \CH^{a+1-g}(\Sym^{g-1}(C))_\QQ\oplus\bigoplus_{i=0}^{g-2}\left(\CH^{a-i}(\Sym^{i}(C))_\QQ\oplus \CH^{a-3g+3+2i}(\Sym^{i}(C))_\QQ\right).
\]
\end{cor}  

For $a$ small or close to $3g-3$, many of these terms vanish for dimensional reasons (Lemma \ref{lem chow vanish}). Consequently, we can recover several descriptions of rational Chow groups in the literature (although our isomorphism is non-explicit). 
\begin{cor}\
\begin{enumerate}[label=\emph{(\roman*)}, leftmargin=0.7cm]
\item \cite{Ramanan} $\CH^{1}(\cN_{\cL})\simeq \ZZ$ .
\item \cite{BKN} $\CH^{2}(\cN_{\cL})_\QQ\simeq \left\{ \begin{array}{ll} \CH_{0}(C)_{\QQ} & \text{ if } g=2, \\ \CH_{0}(C)_{\QQ}\oplus \QQ &   \text{ if }  g>2. \end{array}\right.$
\item \cite{ChoeHwang} $\CH_{1}(\cN_{\cL})_\QQ\simeq \CH_{0}(C)_{\QQ}$.
\end{enumerate}  
\end{cor}  
\begin{proof}
These are special cases of Corollary \ref{cor:general formula Chow} and Lemma \ref{lem chow vanish}, with the additional remark that $\CH^1(\cN_\cL)$ is torsion free because $\cN_{\cL}$ is a smooth projective Fano variety.
\end{proof}  

In fact, by examining the flip sequence in the proof of Theorem \ref{thm motive N rank 2} and using Theorem \ref{thm Jiang Chow flip}, we can be more precise and compute some integral Chow groups of $\cN_{\cL}$; for example, one can deduce that $\CH_{1}(\cN_{\cL})_{\hom} \simeq \Jac(C) \simeq\CH^2(\cN_{\cL})_{\hom}$ (see \cite{BKN,ChoeHwang, LiLinPan}).

More interestingly, our formula enables us to deduce simple descriptions of Chow groups which were previously unknown. Let us give the first examples.

\begin{cor}
  We have
  \[
\CH_{2}(\cN_{\cL})_{\QQ} \simeq  \left\{ \begin{array}{ll} \QQ & \text{ if } g=2, \\ \QQ\oplus \CH_{0}(\Sym^{2}(C))_\QQ&   \text{ if }  g\geq 3 \end{array}\right.
\]
and
\[
\CH^{3}(\cN_{\cL})_{\QQ}\simeq \left\{ \begin{array}{ll} \QQ & \text{ if } g=2, \\ \QQ\oplus \Pic(\Jac(C))_{\QQ}&   \text{ if }  g=3 \\ \QQ^{\oplus 2}\oplus \Pic(\Jac(C))_{\QQ}& \text{ if }g \geq 4 \end{array}\right.
\]
\end{cor}  
\begin{proof}
  This follows from Corollary \ref{cor:general formula Chow} and Lemma \ref{lem chow vanish} together with the fact that
  \[
\Pic(\Sym^{2}(C))_{\QQ}\simeq \QQ\oplus \Pic(\Jac(C))_\QQ.
\]
which follows as $\fh(C)=\QQ(0)\oplus \fh^{1}(C)\oplus \QQ(1)$.
\end{proof}  

For each value of $a$, the descriptions of $\CH^a(\cN_{\cL})_\QQ$ and $\CH_a(\cN_{\cL})_\QQ$ vary in low genus but stabilise in higher genus to the following uniform formulas.

\begin{cor}
	\label{cor:StabilisationCH}
  Let $a\in \NN$. For $g\geq a+1$, we have
  \[
\CH^{a}(\cN_{\cL})_{\QQ} \simeq \bigoplus_{i=\left \lceil \frac{a}{2} \right\rceil}^{a}\CH^{a-i}(\Sym^{i}(C))_\QQ;
  \]
  \[
\CH_{a}(\cN_{\cL})_{\QQ} \simeq \bigoplus_{i=\left \lceil \frac{a}{2} \right\rceil}^{a}\CH_{a-i}(\Sym^{i}(C))_\QQ;
  \]
\end{cor}
\begin{proof}
These follow from Corollary \ref{cor:general formula Chow} and Lemma \ref{lem chow vanish}.
\end{proof}  

\section{Motives of moduli spaces of parabolic vector bundles}
\label{sec:ParabolicBundles}
Moduli spaces of parabolic vector bundles were introduced by Mehta and Seshadri \cite{MS}, where one of the key differences with moduli of vector bundles is that in their GIT construction there is a choice of linearisations of the action giving rise to various notions of stability, encoded by a set of parabolic weights. Several authors have studied the geometry of the birational transformations between these moduli spaces for different weights \cite{BH,BY_rationality,Thaddeus_VGIT}. The Betti numbers and Poincar\'{e} polynomials of these moduli spaces over the complex numbers have been computed by Holla \cite{Holla} using a gauge theoretic approach \`{a} la Atiyah--Bott. For $C = \PP^1$ and rank $n =2$, the Poincar\'{e} polynomials have been studied using variation of parabolic stability \cite{Bauer}. We will use the explicit wall-crossing descriptions to give formulae for the Chow motives of these moduli spaces in rank $n = 2$.

\subsection{Moduli spaces of parabolic vector bundles}

Throughout this section, we fix a set $D= \{ p_1, \dots , p_N \}$ of distinct rational $k$-points on $C$, which we refer to as the parabolic points. We limit ourselves to the case of rational points, mostly for simplicity; it is likely that very similar formulas hold for closed points.

\subsubsection{(Quasi)-Parabolic vector bundles}

\begin{defn}
	\label{def:ParaBun}
A \emph{quasi-parabolic vector bundle} $E_*=(E,E_{i,j})$ on $(C,D)$ is vector bundle $E$ on $C$ with flags $E_{i,j}$ in the fibres at each point $p_i \in D$
\[ E_{p_i} = E_{i,1} \supsetneqq E_{i,2} \supsetneqq \cdots \supsetneqq E_{i,l_i} \supsetneqq E_{i,l_i +1}=0.\]
A \emph{parabolic vector bundle} is a quasi-parabolic vector bundle $E_*$ with weights $\alpha = (\alpha_{i,j})$ satisfying 
\[ 0 \leq \alpha_{i,1} < \dots < \alpha_{i,l_i} < 1 \: \quad \text{for each } \: 1 \leq i \leq N.\]
The discrete invariants of $E_*$ are given by the rank and degree of $E$ and for each $p_i \in D$, the \emph{length} $l_i:=l(E_{i,j})$ and \emph{flag type} $n_{i,j} := \dim E_{i,j}$ (or equivalently, the \emph{multiplicity} $m(E_*)=(m_{i,j})$ defined by $m_{i,j} := n_{i,j} - n_{i,j+1}$ for $1 \leq i \leq N, 1 \leq j \leq l_i$). We write these invariants as a tuple $\eta(E_*) = (\rk(E),\deg(E),m(E_*))$. The flags are \emph{full} if $m_{i,j} = 1$ for all $i,j$. 
\end{defn}

For us, we will be interested in varying the weights for quasi-parabolic bundles and we will later think of these weights as defining a notion of stability for quasi-parabolic vector bundles and the variation of weights gives different notions of stability and moduli spaces. 

\begin{rmk}\label{rmk filt sheaf descr}
One can equivalently think of a quasi-parabolic vector bundle $E_*=(E,E_{i,j})$ on $(C,D)$ as a vector bundle with 
a sheaf filtration (by locally free sheaves) for each $p_i \in D$ 
\[ E = E^i_1 \supsetneqq E_{2}^i \supsetneqq \cdots \supsetneqq E_{l_i}^i \supsetneqq E_{l_i +1}^i=E(-p_i) \]
where $E_j^i$ is the kernel of the sheaf homomorphism $E \twoheadrightarrow (E_{p_i}/E_{i,j}) \otimes \cO_{p_i}$. If $(E_*,\alpha)$ is a parabolic vector bundle, then the weights $\alpha$ determine, for each $p_i \in D$, a filtered sheaf $\cE^i_x$ indexed by $x \in \RR$ (\textit{cf.}\  \cite{Simpson_harmonic} and \cite[$\S$5]{BY_rationality}): for $\alpha_{i,j-1} <x \leq \alpha_{i,j}$ set $\cE^i_{x}:= E^i_j$ (where we define $\alpha_{i,0} = 0$ and $\alpha_{i,l_{i+1}}=1$), which defines $\cE^i_x$ for $x \in (0,1]$ and then set $\cE^i_{x+m} := \cE_x^i(-mp_i)$ for $m \in \ZZ$. 
\end{rmk}

\begin{defn}
Let $(E_*,\alpha)$ and $(F_*,\beta)$ be parabolic vector bundles on on $(C,D)$. A homomorphism $\phi : E \ra F$ is said to be parabolic (resp. strongly parabolic) if $\alpha_{i,j} > \beta_{i,k}$ (resp. $\alpha_{i,j} \geq \beta_{i,k}$) implies $\phi(E_{i,k}) \subset F_{i,k+1}$ for all $i,j,k$.
\end{defn}

If $E_*$ and $F_*$ have the same flag lengths and $\alpha = \beta$, then a homomorphism $\phi : E \ra F$ is parabolic if $\phi(E_{i,j}) \subset F_{i,j}$ (resp. strongly parabolic if $\phi(E_{i,j}) \subset F_{i,j+1}$) for all $i,j$.

There is a subsheaf of parabolic homomorphisms $\cP ar\cH om((E_*,\alpha),(F_*,\beta)) \subset \cH om(E,F)$ with torsion quotient supported on $D$. Consequently, the Euler characteristic of this sheaf is 
\[ \chi(\cP ar\cH om((E_*,\alpha),(F_*,\beta)) = \chi(\cH om(E,F)) - \sum_{p_i \in D} \left( \rk (E)  \rk(F) -P_i((E_*,\alpha),(F_*,\beta)) \right), \]
where $P_i((E_*,\alpha),(F_*,\beta)):=\dim \ParHom((E_{i,*},\alpha_{i,*}),(F_{i,*},\beta_{i,*}))$; see \cite[$\S$4]{Yokogawa} and \cite[Lemma 2.4]{BH}. As this Euler characteristic only depends on the underlying discrete invariants and weights, we write
\[ \chi_{\mathrm{par}}((\eta(E_*),\alpha)^\vee \otimes (\eta(F_*),\beta)) :=    \chi(\cP ar\cH om((E_*,\alpha),(F_*,\beta)). \]

\begin{rmk}\label{rmk euler char parhom}
If $E_*$ and $F_*$ have the same flag lengths and $\alpha = \beta$, we drop the weights from the notation and we have 
\[\chi_{\mathrm{par}}(\eta(E_*)^\vee \otimes \eta(F_*)) =  -n_Fd_E +n_Ed_F+ n_En_F(1-g)  - \sum_{i=1}^N \sum_{j > k}  m_{i,j}(E_*) m_{i,k}(F_*). \]
\end{rmk}

\subsubsection{Parabolic weights and stability}

For a fixed rank $n$, the space of weights for rank $n$ parabolic vector bundles on $(C,D)$ is
\[ \cA:= \{ \widetilde{\alpha}=(\widetilde{\alpha}_{i,j}) \in \RR^{Nn} : 0 \leq \widetilde{\alpha}_{i,1} \leq \widetilde{\alpha}_{i,2} \leq \cdots \leq \widetilde{\alpha}_{i,n} <1 \text{ for } 1 \leq i \leq N \}. \]
Note that in the space of weights $\cA$, we allow $\widetilde{\alpha}_{i,j} = \widetilde{\alpha}_{i,j+1}$.

\begin{defn}
For a weight $\widetilde{\alpha} \in \cA$ and $1 \leq i \leq N$, let $l_i(\widetilde{\alpha})$ denote the number of distinct weights in $\widetilde{\alpha}_{i,1} \leq \widetilde{\alpha}_{i,2} \leq \cdots \leq \widetilde{\alpha}_{i,n}$ and for $1 \leq j \leq l_i(\widetilde{\alpha})$, we let $m_{i,j}(\widetilde{\alpha})$ denote the multiplicites of these distinct weights in increasing order. We refer to $l(\widetilde{\alpha}) = (l_i(\widetilde{\alpha}))_i$ and $m(\widetilde{\alpha})=(m_{i,j}(\widetilde{\alpha}))_{i,j}$ as the \emph{length} and \emph{multiplicity} of $\widetilde{\alpha}$. There is a decomposition
\[ \cA = \bigsqcup_m \cA_{m} \]
by multiplicities and the ordering on multiplicities by successive refinement defines an ordering on the set of weights of fixed multiplicity: $\cA_m > \cA_{m'}$ if $\cA_{m'}$ is a proper face contained in the closure of $\cA_m$.
For $\widetilde{\alpha}$ of length $l$ and multiplicity $m$, the \emph{collapsed weight} $\alpha$ is obtained by deleting for each $i$ repeated instances of weights:
\[ \alpha:= ({\alpha}_{i,1} < {\alpha}_{i,2} < \dots < {\alpha}_{i,l_i})_{1 \leq i \leq N} \:\text{ where }{\alpha}_{i,j}:=\widetilde{\alpha}_{i, 1+\sum_{k < j} m_{i,k}}.\]
\end{defn}

When the rank $n$ and multiplicity $m$ are fixed, we can freely go back and forth between extended weights $\widetilde{\alpha}$ and collapsed weights $\alpha$ by deleting or inserting repeated weights.

\begin{defn}
Let $E_*$ be a quasi-parabolic vector bundle on $(C,D)$ of rank $n$ and with multiplicity $m = (m_{i,j})$. For $\widetilde{\alpha} \in \cA_m$, we define the $\alpha$-slope of $E_*$ by
\[ \mu_{\alpha}(E_*) = \frac{\deg_{\alpha}(E_*)}{\rk(E)}, \: \text{ where } \deg_{\alpha}(E_*) =  \deg(E) + \sum_{p_i \in D} \sum_{j=1}^{n}\widetilde{\alpha}_{i,j} = \deg(E) + \sum_{p_i \in D} \sum_{j=1}^{l_i}{\alpha}_{i,j} m_{i,j}.\]

For a subbundle $E' \subset E$, we can intersect $E'_{p_i}$ with the flags $E_{i,j}$ in $E_{p_i}$. If we don't delete repeated subspaces in this flag in $E'_{p_i}$, we obtain a length $l_i$ flag in $E'_{p_i}$ with \emph{multiplicity} defined (slightly unconventionally, but as in \cite{BY_rationality}) by setting for $1 \leq j \leq l_i$
\[ m'_{i,j}:= \dim (E'_{p_i} \cap E_{i,j}) - \dim (E'_{p_i} \cap E_{i,j+1}) \]
which may now also be zero. Then the $\alpha$-slope of this subbundle is define by
\[ \mu_{\alpha}(E'_*) := \frac{\deg_{\alpha}(E'_*)}{\rk(E')}, \: \text{ where } \deg_{\alpha}(E'_*) =  \deg(E') + \sum_{p_i \in D} \sum_{j=1}^{l_i}{\alpha}_{i,j} m'_{i,j}.\]
\end{defn}

If instead we delete repeated subspaces in the flag in $E'_{p_i}$, we obtain a flag of length $l_i' \leq l_i$
\[ E'_{p_i} = E'_{i,1} \supsetneqq E'_{i,2} \supsetneqq \cdots \supsetneqq E'_{i,l'_i} \supsetneqq E'_{i,l'_i +1}=0,\]
giving a quasi-parabolic structure\footnote{Usually the multiplicity of $E'_*$ is defined as the differences $\dim E'_{i,k} - \dim E'_{i,k+1} >0$ in the collapsed flag.} $E'_*$. We endow this with a subset of the weights 
\begin{equation}\label{weights for subbundle}
 (\alpha_{i,k}' )= ({\alpha}_{i,j_1} < {\alpha}_{i,j_2} < \dots < {\alpha}_{i,j_{l'_i}})_{1\leq i \leq N}, \text{ where }  j_k:= \max\{j : E_{i,k}' \subseteq E_{i,j} \}.
\end{equation}
Then $ \deg_{\alpha}(E'_*) = \deg(E') + \sum_{p_i \in D} \sum_{k=1}^{l'_i}{\alpha}_{i,j_k} (\dim E'_{i,k} - \dim E'_{i,k+1})$. 

\begin{rmk}\label{rmk compl inv}
Given a short exact sequence $E' \hookrightarrow E \twoheadrightarrow E''$ with quasi-parabolic structure $E_*$ on $E$, one can similarly give $E''$ a quasi-parabolic structure and define its multiplicity $m''$ analogously. The advantage of our unconventional definition of multiplicities $m'$ and $m''$ of sub- and quotient bundles is that then $m = m' + m''$. If $E$ has weights $\alpha$, then we also have subsets $\alpha'$ and $\alpha''$ of the weights as described above which satisfy $\widetilde{\alpha} = \widetilde{\alpha}' + \widetilde{\alpha}''$ for the associated extended weights. Moreover $(E_*',\alpha') \hookrightarrow (E_*,\alpha)\twoheadrightarrow (E_*'',\alpha'')$ are parabolic homomorphisms.
\end{rmk}

\begin{defn}
	\label{def:StabilityParaBun}
A quasi-parabolic bundle $E_*$ with $m(E_*) = m(\alpha)$ is \emph{$\alpha$-semistable} (resp. \emph{$\alpha$-stable}) if for all subbundles $E' \subset E$, we have $\mu_\alpha(E'_*) \leq \mu_\alpha(E_*)$ (resp. $\mu_\alpha(E'_*) < \mu_\alpha(E_*)$).
\end{defn}

\begin{rmk}\label{rmk shifting weights}
We note that the notion of stability is invariant under the shift
\[ (\alpha_{i,j}) \mapsto (\alpha_{i,j} + C_i) \]
for constants $C_i$ for $1 \leq i \leq N$. Hence, we may assume $\alpha_{i,1} = 0$ without loss of generality.
\end{rmk}

Fix discrete invariants $\eta = (n,d,m)$. For each $\alpha \in \cA_m$, there is a moduli space $\cN^\alpha_{C,D}(\eta)$ of $\alpha$-semistable parabolic bundles with invariants $\eta$, which is a normal projective variety \cite{MS}. If it is non-empty, by the deformation theory of parabolic bundles \cite[Theorem 5.1]{Yokogawa}, it has dimension
\[ \dim \cN^\alpha_{C,D}(\eta) = \dim \cN_C(n,d) + \sum_{i=1}^N \dim \cF(m_i) = n^2(g-1) + 1 +  \sum_{i=1}^N \sum_{j >k} m_{i,j}m_{i,k} \]
where $\cF(m_i)$ is the variety of flags of type $n = \sum_{j=1}^{l_i} m_{i,j} > \sum_{j=2}^{l_i} m_{i,j} > \cdots > m_{i,l_i} > 0$. The moduli space of $\alpha$-stable parabolic bundles $\cN^{\alpha-s}_{C,D}(\eta)$ is a smooth open subvariety of $\cN^\alpha_{C,D}(\eta)$.

\subsubsection{Geometric description of variation of stability}\label{sec geom var stab par}

Fix the discrete invariants $\eta =(n,d,m)$ given by the rank $n$, degree $d$ and multiplicity $m$. For simplicity, we write $\cN^\alpha:=\cN^\alpha_{C,D}(\eta)$ and similarly $\cN^{\alpha-s}$ for the open subset of $\alpha$-stable parabolic bundles.

We can divide the space of weights $\cA$ into walls and chambers such that stability is constant on the chambers and changes over the walls. More precisely, on the walls there are strictly semistable bundles, i.e., we have $\mu_\alpha(E_*') = \mu_\alpha(E_*)$ for a subbundle $E' \subset E$. To describe the walls it therefore suffices to consider the possible discrete invariants $\eta'$ of subbundles.


\begin{defn}\label{defn par wall}
For $\eta':=(n',d', m')$ such that $0 < n' < n$, $d' \in \ZZ$ and  $m'=(m'_{i,j})$ are non-negative integers with $m_{i,j}' \leq m_{i,j}$ and $\sum_{j=1}^{l_i} m_{i,j}' = n'$, we define the corresponding \emph{wall} 
\[ W_{m,\eta'} = \left\{ \alpha \in \cA_m : \mu_\alpha(\eta') = \mu_\alpha(\eta) \right\} = \cA_{m} \cap W_{\eta'}\]
where for $\alpha \in \cA_m$:
\[ \mu_\alpha(\eta') = \mu_\alpha(\eta) \iff 
\frac{d' + \sum_{i=1}^N \sum_{j=1}^{l_i} {\alpha}_{i,j} m_{i,j}'}{n'} = \frac{d + \sum_{i=1}^N \sum_{j=1}^{l_i} {\alpha}_{i,j} m_{i,j}}{n}.\]
The complement of $\RR^{Nn} \setminus W_{\eta'}$ of each wall is two half-spaces $H_{\eta'}^{\pm}$ such that for $\alpha \in H_{\eta'}^{\pm}$ we have $\pm (\mu_\alpha(\eta') -\mu_\alpha(\eta)) >0$. The connected components of $\cA_m \setminus \cup_{\eta} W_{m,\eta}$ are called \emph{chambers} and we refer to weights $\alpha \in \cA_m \setminus \cup_{\eta} W_{m,\eta}$ as being \emph{generic}.
\end{defn}

Note that there are only finitely many walls in $\cA_m$. The invariants $\eta':=(n',d', m')$ have complementary invariants $\eta''=(n'',d'',m'')$ satisfying $n'+n'' = n$, $d' + d'' =d$, $m'+ m'' = m$ and determine the same wall\footnote{If $\eta'=\eta(E_*')$ for a subbundle $E_*' \subset E_*$, then $\eta''$ is the discrete invariants of the quotient bundle $E''_*$.}.  We say a wall $W_{m,\eta'}$ is \emph{good} if there are no discrete invariants other than $\eta'$ and $\eta''$ which define this wall. All walls are good, if we take multiplicities corresponding to full flags (see \cite[(2.4)]{Thaddeus_var_PHiggs}).

By \cite[Proposition 3.2]{BY_rationality}, $\cA_m$ contains a generic weight if and only if the degree $d$ and multiplicities $m_{i,j}$ have greatest common divisor $1$. In this case, for any generic weight, (semi)stability coincides with stability and corresponding the moduli space of stable parabolic vector bundles is a fine moduli space and a smooth projective variety.

The geometric description of the birational tranformation between two generic weights separated by a good wall is described by Boden--Hu \cite{BH}, Boden--Yokogawa \cite{BY_rationality} and Thaddeus \cite{Thaddeus_VGIT}. These papers assume that $k$ is algebraically closed. However, when crossing a good wall, their proofs show that, for a parabolic bundle in the exceptional locus, there is a unique subbundle which does not violate stability on one side of the wall but does on the other side. A descent argument then implies that the description of the wall-crossing holds over a general field $k$.

\begin{thm}[Boden--Hu, Boden--Yokogawa and Thaddeus]\label{thm geom flips par}
Consider a line segment which joins two adjacent chambers and passes through a weight $\alpha$ on a single good wall $W_{m,\eta'}$. Let $\alpha(\pm) \in H_{\eta'}^{\pm}$ be weights in adjacent chambers which satisfy
\[\pm \left(\mu_{\alpha(\pm)}(\eta') - \mu_{\alpha(\pm)}(\eta) \right) >0\]
Then there is a standard flip
 \[   \xymatrix{ 
   \cN^{\alpha(-)} \ar[rd] \ar@{-->}[rr]^{(n_-,n_+)}_{\text{centre }
    \cN^{\alpha-sss} } &  & \cN^{\alpha(+)} \ar[ld] \\ & \cN^{\alpha}
    }
  \]
with centre isomorphic to a product of smooth projective moduli spaces 
\[ \cN^{\alpha-sss} := \cN^{\alpha} \setminus \cN^{\alpha-s} \cong \cN^{\alpha'}_{C,D}(\eta') \times \cN^{\alpha''}_{C,D}(\eta'')\]
with smaller invariants $\eta'$ and complementary invariants $\eta''$ (the weights $\alpha'$ are defined at \eqref{weights for subbundle} and similarly for $\alpha''$). Over the centre the fibres are projective spaces of dimensions
\begin{align*}
n_-:=n_-(\eta') = & -\chi_{\mathrm{par}}((\eta'',\alpha'')^\vee \otimes (\eta',\alpha'))-1\\
n_+:= n_+(\eta')= &  -\chi_{\mathrm{par}}((\eta',\alpha')^\vee \otimes ( \eta'',\alpha''))-1.
\end{align*}
\end{thm}

The closest formulation in the literature to the above statement is \cite[Theorem 4.1]{BY_rationality}, whose proof computes the flip type. Our (slightly unconventional) approach to multiplicities of sub- and quotient bundles allows us to easily compute the dimensions of the fibres using Remark \ref{rmk euler char parhom}:
\begin{align}\label{equation for n plus}
\begin{split}
n_-(\eta')= & n'd'' - n''d' + n'n''(g-1) -1 + \sum_{i=1}^N\sum_{j>k} m_{i,j}'' m_{i,k}' \\
n_+(\eta')= &  n''d' - n'd'' + n''n'(g-1) -1+ \sum_{i=1}^N\sum_{j>k} m_{i,j}' m_{i,k}'',
\end{split}
\end{align}
which satisfy $n_- + n_+ + 1 = \codim(\cN^{\alpha-sss},\cN^{\alpha})$ as stated in \cite[Theorem 4.1]{BY_rationality}.

\subsubsection{Geometric description of flag degeneration}

For weights with different multiplicities that are not separated by walls $W_{\eta'}$, one can consider the flag degeneration giving by forgetting part of the flags. The geometric description of this is due to Boden--Hu \cite{BH} and Boden--Yokogawa \cite{BY_rationality}.

\begin{thm} \cite[Theorem 4.2]{BY_rationality} \label{thm geom flag degen}
Let $\cA_{m'} > \cA_m$ and suppose $\alpha \in \cA_{m'}, \beta \in \cA_m$ are generic weights which are not separated by any walls $W_{\eta'} \subset \cA$. Then there is a forgetful map
\[ \cN^\alpha \ra \cN^\beta \]
which is an iterated Zariski locally trivial flag bundle.
\end{thm}

\begin{rmk}\label{rmk flag degen grass}
In fact, one can factor this forgetful map into a tower Grassmannian bundles as follows. Consider a chain $\cA_{m'} =\cA_{m(0)} > \cA_{m(1)}> \dots  > \cA_{m(t)} =\cA_m$ such that for each $0 \leq k <t$ there exists a unique $i_k$ such that the length of $l_{i_k}$ of $m(k)$ and $m(k+1)$ differ by $1$. Take generic  weights $\alpha(k) \in \cA_{m(k)}$ not separated by walls with $\alpha(0)=\alpha$ and $\alpha(t) = \beta$, then there is a unique $j_k$ such that the morphism $\cN^{\alpha(k)} \ra \cN^{\alpha(k+1)}$ is given by forgetting the subspace $E_{i_k,j_k}$ in the flag of $E_{p_{i_k}}$. This map is a Grassmannian bundle with fibre $\text{Gr}(m(k)_{i_k,j_k},m(k)_{i_k,j_k} + m(k)_{i_k,j_k-1})$. Hence, we obtain a factorisation into a tower of Grassmannian bundles
\[ \cN^\alpha = \cN^{\alpha(0)} \ra \cN^{\alpha(1)} \ra \cdots \ra \cN^{\alpha(t)}= \cN^\beta. \]
\end{rmk}

Similarly using flag degenerations, one can relate moduli spaces of parabolic vector bundles to moduli spaces of vector bundles (viewed as parabolic vector bundles with trivial flags) provided one takes sufficiently small weights; see \cite[Proposition 5.3]{BY_rationality}. 

\subsubsection{Hecke modifications of parabolic vector bundles}\label{sec Hecke mod}

One can compare moduli spaces of vector bundles of rank $n$ for different values of $d$ using Hecke modifications of parabolic vector bundles. This useful observation was noted at the end of \cite{MS} in rank $n = 2$ and then expanded upon in \cite[$\S$5]{BY_rationality}. Let us introduce Hecke modifications as isomorphisms between moduli spaces of parabolic vector bundles for the same rank but different degrees, multiplicities and weights. The name Hecke modification reflects the fact that on the underlying vector bundle, we perform a Hecke modification at one of the parabolic points $p_i \in D$ using part of the flag at that point.

\begin{defn}
Fix invariants $\eta = (n,d,m)$ and a weight $\alpha \in \cA_m$. For  $p_{i} \in D$ and $1 \leq j < l_i$ and $\alpha_{i,j} <\beta \leq \alpha_{i,j+1}$, we define new invariants $\cH_{i,j}(\eta)$ and weights $\cH_{i,j,\beta}(\alpha)$ by
\begin{enumerate}
\item $\cH_{i,j}(\eta)=(n,d - \sum_{k=1}^jm_{i,k},\hat{m})$ with cyclically permuted multiplicities
\[(\hat{m}_{i,1},\dots \hat{m}_{i,l_i})=(m_{i,j+1},\dots, m_{i,l_{i}},m_{i,1},\dots, m_{i,j})\]
and $\hat{m}_{i',j} = m_{i',j}$ for all $i' \neq i$ and $1 \leq j \leq l_{i'}$,
\item $\hat{\alpha}:=\cH_{i,j,\beta}(\alpha)$ is given by 
 \[ (\hat{\alpha}_{i,1},\dots \hat{\alpha}_{i,l_i})=(\alpha_{i,j+1}-\beta,\alpha_{i,j+2}-\beta,\dots,\alpha_{i,l_i} - \beta, 1+ \alpha_{i,1} - \beta, \dots, 1 + \alpha_{i,j} - \beta)\]
 and $\hat{\alpha}_{i',j} = \alpha_{i',j}$ for all $i' \neq i$ and $1 \leq j \leq l_{i'}$.
\end{enumerate}
Then the \emph{Hecke modification} at $p_i$ with respect to $j$ and $\alpha_{i,j} <\beta \leq \alpha_{i,j+1}$ is the isomorphism
\[ \cH_{i,j,\beta} : \cN^\alpha(\eta) \ra \cN^{\cH_{i,j,\beta}(\alpha)}(\cH_{i,j}(\eta)) \]
sending $E_*$ to $E'_*$ with $E' := \ker(E \twoheadrightarrow E_{p_i}/E_{i,j+1} \otimes \cO_{p_i})$ which inherits a quasi-parabolic structure from $E$ with multiplicities as specified above (see \cite[\S 5]{BY_rationality}).
\end{defn}

Hecke modifications are isomorphisms with inverses given by Hecke modifications. Moreover, Hecke modifications at different points commute in the obvious sense. 

The Hecke modification $\cH_{i,j,\beta}$ is most simply understood in terms of the $D$-tuple of $\RR$-indexed filtered sheaves $\cE^{i'}_x$ associated to the parabolic vector bundle $(E_*,\alpha)$ (see Remark \ref{rmk filt sheaf descr}) as performing a shift by $\beta$ in the $\RR$-indexed filtration at the parabolic point $p_i$ (see \cite[$\S$5]{BY_rationality}):
\[\cH_{i,j,\beta}(\cE^{i'}_x) = \hat{\cE}^{i'}_x \quad \text{where } \: \hat{\cE}^{i'}_x := \left\{ \begin{array}{ll} \cE^i_{x + \beta} & \text{if } i' = i \\ \cE^{i'}_x & \text{if } i' \neq i\end{array} \right. \]

\begin{ex}\label{ex Hecke mod rank 2}
Let us consider the case for $n = 2$. For a parabolic point $p_i$ at which the multiplicity $m$ specifies a full flag (i.e. $l_i = 2$), we get a Hecke modification for $j=1$. This decreases the degree $d$ by $1$ and as $(m_{i,1},m_{i,2}) = (1,1)$ permuting these multiplicities does not change $m$, thus $\cH_{i,1}(\eta) = (n,d-1,m)$. For $\alpha_{i,1} <\beta \leq  \alpha_{i,2}$, the new weight $\hat{\alpha}=\cH_{i,1,\beta}(\alpha)$ is given by $(\hat{\alpha}_{i,1},\hat{\alpha}_{i,2}) = (\alpha_{i,2} - \beta, 1 + \alpha_{i,1} - \beta)$ and is unchanged for $i'\neq i$. When we take the maximal value $\beta = \alpha_{i,2}$, we write the Hecke modification $\cH_{i,1,\alpha_{i,2}}$ simply as
\[ \cH_{p_i} : \cN^\alpha(2,d,m) \ra \cN^{\alpha(p_i)}(2,d-1,m) \]
with $\alpha(p_i)_{i,j} = (0, 1 + \alpha_{i,1} - \alpha_{i,2})$ and $\alpha(p_i)_{i',j} = \alpha_{i',j}$ for all $i' \neq i$. More generally, given a subset $D'\subset D$ with $l_i = 2$ for all $p_i \in D'$, we let
\[ \cH_{D'} : \cN^\alpha(2,d,m) \ra \cN^{\alpha(D')}(2,d-|D'|,m)\]
denote the composition (in any order) of the Hecke modifications $\cH_{p_i}$ for $p_i \in D'$.
\end{ex}

For degrees $d$ and $d'$ which are coprime to $n$, one can relate the vector bundle moduli spaces $\cN(n,d)$ and $\cN(n,d')$ using moduli spaces of parabolic vector bundles, Hecke modifications and a sequence of wall-crossing flips, together with degeneration of flag structures (see \cite{MS,BY_rationality}).

\subsection{Motivic consequences}

\subsubsection{Motivic variation of stability for moduli of parabolic vector bundles}

By combining Theorem \ref{thm geom flips par} and Theorem \ref{thm Jiang Chow flip}, we obtain the following result describing the Chow motives of moduli spaces of parabolic vector bundles for different weights.

\begin{cor}\label{cor motivic par WC}
For a line segment which joins two adjacent chambers and passes through a weight $\alpha$ on a single good wall $W_{m,\eta'}$, let $\alpha(\pm) \in \cH_{\eta'}^{\pm}$ be weights in adjacent chambers satisfying
\[\pm \left(\mu_{\alpha(\pm)}(\eta') - \mu_{\alpha(\pm)}(\eta) \right) >0.\]
Then there is an explicit isomorphism of integral Chow motives
\begin{equation}
	\label{eqn:WallcrossingPara}
	\fh(\cN^{\alpha(+)}) \oplus \bigoplus_{n_-(\eta) <j \leq n_+(\eta)} \fh(\cN^{\alpha-sss})(j) \simeq \fh(\cN^{\alpha(-)}) \oplus \bigoplus_{n_+(\eta) < j \leq n_-(\eta)} \fh(\cN^{\alpha-sss})(j),
\end{equation}
where $n_{\pm}(\eta')$ are specified in \eqref{equation for n plus}.
\end{cor}

Note that in the above isomorphism at least one of the big direct sums on either side is empty. In the case when $n_-(\eta') = n_+(\eta')$, this wall-crossing is a flop and both big direct sums are empty, so the Chow motives of these moduli spaces are isomorphic.\\

\begin{rmk}\label{rmk par n plus grows with g}
For fixed invariants $\eta = (n,d,m)$, the space $\cA_m$ of corresponding weights is cut into chambers by finitely many walls $W_{m,\eta'}$, which do not depend on the genus $g$ of $C$ (see Definition \ref{defn par wall}). However, the types of the flips at these walls, given by $n_{\pm}(\eta')$ in Equation \eqref{equation for n plus}, grow linearly with $g$. For low genus, there may be some chambers which give rise to empty moduli spaces and there may be some wall-crossings which are simply blow-ups. However, for fixed invariants $\eta = (n,d,m)$, if $g$ is sufficiently large, then we have $n_{\pm}(\eta') >0$ for all invariants $\eta'$ specifying walls in $\cA_m$. For example, for rank $n =2$ with full flags, it suffices to take $g \geq 2$ (see the bounds on the flip type given in Proposition \ref{prop par WC}).
\end{rmk}

As an application, we slightly correct and strengthen the result of Chakraborty \cite[Theorem 1.1]{Chakraborty1} to integral coefficients.

\begin{cor}
	\label{cor:CH1constant}
Fix invariants $\eta = (n,d,m)$ with $m$ corresponding to full flags. 
There exists an integer $g_0$, depending only on $\eta$, such that as long as $g\geq g_0$,
 we have isomorphisms between the Chow groups of 1-cycles and the third homology $H_3$ for two generic weights $\alpha$ and $\beta$:
	\begin{align*}
	\CH_1(\cN^\alpha)&\simeq \CH_1(\cN^\beta),\\
	H_3(\cN^\alpha)&\simeq  H_3(\cN^\beta),
	\end{align*}
where $H_3$ denotes either singular homology equipped with its Hodge structure if $k = \CC$,  or the $\ell$-adic homology over $\bar{k}$ equipped with its Galois action if $\ell$ is coprime to the characteristic of $k$. In particular, the intermediate Jacobians for 1-cycles are isomorphic.	
\end{cor}
\begin{proof}
As the flags are full, every wall is good (i.e. corresponds uniquely to a pair of complementary invariants $(\eta',\eta'')$ with $\eta = \eta' + \eta''$). By wall-crossing, it suffices to show the isomorphism for $\alpha$ and $\beta$ in adjacent chambers. As in Remark \ref{rmk par n plus grows with g}, we take $g$ sufficiently large so that $n_{\pm}(\eta') >0$ for all invariants $\eta'$ specifying walls in $\cA_m$. We see that in \eqref{eqn:WallcrossingPara}, the index $j$, when appearing, is at least 2. Hence the big direct sums do not contribute to $\CH_1$, nor to $H_3$.
\end{proof}

\subsubsection{Motivic descriptions of flag degenerations}

We have the following motivic consequence of Theorem \ref{thm geom flag degen}.

\begin{cor}\label{cor motives flag degen}
Let $\cA_{m'} > \cA_m$ and suppose $\alpha \in \cA_{m'}, \beta \in \cA_m$ are generic weights which are not separated by any walls $W_{\eta'} \subset \cA$. Then there is an explicit isomorphism of integral Chow motives
\[ \fh(\cN^\alpha) \simeq  \fh(\cN^\beta) \otimes T \]
where $T \in \CHM(k,\ZZ)$ is a pure Tate motive which can be explicitly determined.
\end{cor}
\begin{proof}
By Theorem \ref{thm geom flag degen}, the forgetful map $\cN^\alpha \ra \cN^\beta$ is an iterated flag bundle. More precisely, as in Remark \ref{rmk flag degen grass}, we can factor this forgetful map into a tower Grassmannian bundles
\[ \cN^\alpha = \cN^{\alpha(0)} \ra \cN^{\alpha(1)} \ra \cdots \ra  \cN^{\alpha(t)}= \cN^\beta \]
with $\cN^{\alpha(k)} \ra \cN^{\alpha(k+1)}$  having fibre a Grassmannian $\text{Gr}(k)$. Then we have explicit isomorphisms
\[ \fh(\cN^{\alpha(k)}) \simeq \fh(\cN^{\alpha(k+1)}) \otimes \fh(\text{Gr}(k)) \]
which together give
\[ \fh(\cN^{\alpha}) \simeq \fh(\cN^\beta) \otimes T \]
where $T$ is a tensor product of motives $\fh(\text{Gr}(k))$ of Grassmannians for $0 \leq k < t$.
\end{proof}

In the case when we forget all of the flags at all of the parabolic points, we obtain simply a vector bundle. Suppose that $n$ and $d$ are coprime, so that semistability and stability coincide for rank $n$ degree $d$ vector bundles and $\cN = \cN(n,d)$ is smooth and projective. In this case the weight given by the origin $\underline{0} \in \cA$ is a generic weight and corresponds to parabolic vector bundles with trivial flags (i.e. vector bundles) and we have $\cN = \cN^{\underline{0}}$. By  \cite[Proposition 5.3]{BY_rationality}, if $\alpha \in \cA_m$ is a generic weight which is sufficiently small, then $\alpha$ and $\underline{0}$ are not separated by any wall and consequently we obtain the following special case of the above result.

\begin{cor}\label{cor full deg flag motive}
Let $n$ and $d$ be coprime. Then for a sufficiently small generic weight $\alpha \in \cA_m$, we have an explicit isomorphism of integral Chow motives
\[ \fh(\cN^{\alpha}) \simeq \fh(\cN) \otimes T, \quad \text{ where } \:\: T = \bigotimes_{i=1}^N \fh( \cF(m_i))\]
and $\cF(m_i)$ is the variety of flags $n = \sum_{j=1}^{l_i} m_{i,j} > \sum_{j=2}^{l_i} m_{i,j} > \cdots > m_{i,l_i} > 0$.
\end{cor}

As an application, we compute the intermediate Jacobian of 1-cycles for the moduli spaces $\cN^\alpha_\cL$ of parabolic vector bundles with fixed determinant\footnote{Fixing the determinant does not change the type of the wall-crossing flips, only their centres.} $\cL$ for a generic weight $\alpha$; this strengthens the results of Chakraborty \cite[Theorem 1.2]{Chakraborty1} and \cite[Proposition 3.4]{Chakraborty2} to integral coefficients.

\begin{cor}\label{cor int Jac 1cycles par}
Fix invariants $\eta = (n,d,m)$ with $m$ corresponding to full flags. Fix a degree $d$ line bundle $\cL$. There exists an integer $g_0>2$, depending only on $\eta$, such that as along as $g\geq g_0$, 
for a generic weight $\alpha$, we have an isomorphism of Chow groups of 1-cycles,  
\begin{equation*}
	\CH_1(\cN^\alpha_\cL)_{\hom} \simeq \CH_1(\cN_\cL)_{\hom}
	\end{equation*}
and if $k = \CC$, we also have an isomorphism of abelian varieties:
	\begin{equation*}
		J\big(H_3(\cN^\alpha_\cL)\big)\simeq \Jac(C),
	\end{equation*}
	where $J$ denotes the functor that associates to a level 1 Hodge structure its Jacobian. 
	
For rank $n=2$ with full flags at $N$ parabolic points, and always $g\geq g_0$ and $\alpha$ generic, we have $\CH_1(\cN^\alpha_\cL)\simeq \CH_0(C)\oplus \ZZ^{\oplus N}$ and, if $k = \CC$, we have 
	\begin{equation*}
			\operatorname{AJ}\colon \CH_1(\cN^\alpha_\cL)_{\hom}\xrightarrow{\simeq}  J\big(H_3(\cN^\alpha_\cL)\big)\simeq \Jac(C),
	\end{equation*}
where $\operatorname{AJ}$ is the Abel--Jacobi map.
\end{cor}
\begin{proof}
As in Remark \ref{rmk par n plus grows with g}, we take $g$ sufficiently large so that $n_{\pm}(\eta') >0$ for all invariants $\eta'$ specifying walls in $\cA_m$. By Corollary \ref{cor:CH1constant}, the $\CH_1$ and the intermediate Jacobian do not vary for generic weights. Therefore, we can assume that $\alpha$ is sufficiently small. In this case,  by Corollary \ref{cor full deg flag motive}, we have isomorphisms of Chow groups and Hodge structure (for $k = \CC$)  
	\begin{align*}
		\CH_1(\cN^\alpha_\cL)_{\hom}&\simeq \CH_1(\cN_\cL)_{\hom}\oplus \CH_0(\cN_\cL)_{\hom}^{\oplus m},\\
		H_3(\cN^\alpha_\cL, \ZZ)&\simeq H_3(\cN_\cL, \ZZ)\oplus H_1(\cN_\cL, \ZZ)(1)^{\oplus m},
	\end{align*}
	where $m$ is the number of summands $\ZZ(1)$ in the Tate motive $T$. However,  $H_1(\cN_\cL, \ZZ)=0$ and $\CH_0(\cN_\cL)_{\hom}=0$, as $\cN_\cL$ is rational \cite{KS}. Therefore, 
\[	\CH_1(\cN^\alpha_\cL)_{\hom}\simeq \CH_1(\cN_\cL)_{\hom} \quad \text{ and } \quad H_3(\cN^\alpha_\cL, \ZZ)\simeq H_3(\cN_\cL, \ZZ),\]
and they are compatible with respect to the Abel--Jacobi maps.
When $g>2$, by \cite[Theorem 2.1]{IyerLewis}, $H_3(\cN_\cL, \ZZ)\simeq H_1(C, \ZZ)(1)$ as Hodge structures; thus we obtain the claimed result for the intermediate Jacobian.

Now suppose $n=2$ and we take full flags at $N$ parabolic points. By the recent work \cite{LiLinPan}, we have $\CH_1(\cN_\cL) \simeq \CH_0(C)$ and as $\cN_{\cL}$ is rational, we also have $\CH_0(\cN_\cL) = \ZZ$. By applying Corollary \ref{cor full deg flag motive}, where the Tate motive $T$ is the motive of $(\PP^1)^N$, we deduce $\CH_1(\cN^\alpha_\cL)\simeq \CH_0(C)\oplus \ZZ^{\oplus N}$. Over $k = \CC$, we have $\CH_1(\cN_\cL)_{\hom}\simeq \Jac(C)$, which implies the claimed Abel--Jacobi isomorphism with the intermediate Jacobian. 
\end{proof}

Similarly to Corollary \ref{cor:StabilisationCH} in $\S$\ref{sec cor Chow gps N}, we have the following stabilisation result. 

\begin{cor}\label{cor stable chow gps parabolic}
Fix invariants $\eta = (n,d,m)$ with $m$ corresponding to full flags. For $i \in \NN$, there is an integer $g_0$, depending only on $\eta$, such that for $C$ of genus $g\geq g_0$ and any generic weight $\alpha$, the Chow groups $\CH^i(\cN^\alpha)$ and $\CH_i(\cN^\alpha)$ are both independent of $\alpha$ and can be computed in terms of Chow groups of $\cN$.
\end{cor}
\begin{proof}
Recall from Remark \ref{rmk par n plus grows with g} that there are finitely many walls $W_{m,\eta'}$, which are all good and do not depend on the genus $g$ of $C$, but the flip types $n_{\pm}(\eta')$ of these walls grow linearly with $g$ as in Equation \eqref{equation for n plus}. In particular, for fixed $i$, if we take $g$ sufficiently large, then $n_{\pm}(\eta') \geq i$ for all walls. In Corollary \ref{cor motivic par WC}, the index $j$ in \eqref{eqn:WallcrossingPara} then satisfies $j > i$ and thus the big direct sums do not contribute to $\CH^i$ nor $\CH_i$, as $i-j <0$. Consequently $\CH^i(\cN^\alpha)$ and $\CH_i(\cN^\alpha)$ are both independent of $\alpha$. Thus we can take $\alpha$ sufficiently small to apply Corollary \ref{cor full deg flag motive}, where $T=\fh(\cF)^{\otimes N}$ and $\cF$ is the full flag variety, and express $\CH^i(\cN^\alpha)$ and $\CH_i(\cN^\alpha)$ in terms of Chow groups of $\cN$.
\end{proof}

\subsection{Closed formulas for the motive in rank 2}

Let us consider rank $n=2$ vector bundles of degree $d$ with $N$ parabolic points. By the shifting trick described in Remark \ref{rmk shifting weights}, we can without loss of generality set the first parabolic weight at each parabolic point to be zero. Therefore, we assume $\alpha_{i,1} = 0$ for $ i =1, \dots , N$ and let us simply write $\alpha_i:=\alpha_{i,2}$; then we consider weights lying in the half-open hypercube
\[ \alpha = (\alpha_1, \dots , \alpha_N) \in \cA_N:=[0,1)^N.\]
Weights in the open hypercube $(0,1)^N$ correspond to the multiplicity $m = (m_{i,j})$ with $m_{i,j} = 1$ for all $i,j$ (i.e. full flags at each parabolic point). We fix invariants $\eta = (2,d,m)$ for rank $2$ degree $d$ parabolic vector bundles with full flags at each of the $N$ parabolic points.

Unlike for the moduli space $\cN=\cN(2,d)$ of rank $2$ vector bundles, where we had to restrict to $d$ odd to describe the Chow motive (as for $g > 2$ and $d$ even, $\cN$ is singular), for parabolic vector bundles we can consider even degrees provided we choose a generic weight. In rank $n = 2$, we will compute the Chow motives of the associated parabolic moduli spaces for all $d$ and all choices of generic weights. In fact, the moduli spaces of rank $2$ parabolic vector bundles for odd and even degrees are related by Hecke modifications and so it suffices to compute these Chow motives for odd degrees (see Remark \ref{rmk par vb rk 2 even deg}, for how to compute them in even degrees using a Hecke modification).

\subsubsection{Description of the walls}

We recall that the walls correspond to subbundles of the same slope. In rank $2$, the invariants of a subbundle have the form $\eta'=(1,d',m')$, where for each $1 \leq i \leq N$, we have $m_{i,1}' + m_{i,2}' = 1$. In particular, $m_{i,1}'$ is determined by $m_{i,2}' \in \{ 0,1\}$ and so we simply write $m_i':=m'_{i,2}$. The associated wall in $
\cA_N$ is given by the equation
\[ 2d' - d = \sum_{i=1}^N  \alpha_i(1-2m'_i).\]
Note that the integer $2d' - d$ has the same parity as $d$. Since the complementary invariants $\eta''=(n''=1,d'',m'')$ define the same wall, we can assume without loss of generality that $2d'-d \geq 0$. Furthermore, the multiplicity $m'$ can be encoded in a subset $\cI \subset \{ 1, \dots , N\}$ given by $\cI = \{ 1 \leq i \leq N : m_{i}' = 0 \}$. Then the corresponding wall has the equation
\[ 0 \leq 2d' -d = \sum_{i \in \cI} \alpha_i - \sum_{i \in \cI^c} \alpha_i \]
where $\cI^c := \{1,\dots, N \} \setminus \cI$. This wall has non-empty intersection with $\cA_N$ precisely when $|\cI| > 2d' -d$. This proves the following description of the walls.

\begin{lemma}\label{lemma walls}
The walls in $\cA_N = [0,1)^N$ can be enumerated as follows.
\begin{enumerate}
\item For odd $d$, for each $s \in \NN$ with $2s +1 <N$ and for each $\cI \subset \{1,\dots , N\}$ with cardinality $|\cI| > 2s+1$, we have the wall $W_{s,\cI}$ defined by
\[ 2s+1 = \sum_{i \in \cI} \alpha_i - \sum_{i \in \cI^c} \alpha_i. \]
This wall is determined by the discrete invariants $\eta'(s,\cI)=(1,d'(s),m'(\cI))$, where $d'(s) = (2s+1+d)/2$ and $m'(\cI)_{i,1} = 1$ (and $m'(\cI)_{i,2} = 0$) if and only if $i \in \cI$.
\item For even $d$, for each $s \in \NN$ with $2s <N$ and $\cI \subset \{1,\dots , N\}$ with cardinality $|\cI| > 2s$, we have the wall $W_{k,\cI}$ defined by
\[ 2s = \sum_{i \in \cI} \alpha_i - \sum_{i \in \cI^c} \alpha_i. \]
This wall is determined by the discrete invariants $\eta'(s,\cI)=(1,d'(s),m'(\cI))$, where $d'(s) = (2s+d)/2$ and $m'(\cI)_{i,1} = 1$ (and $m'(\cI)_{i,2} = 0$) if and only if $i \in \cI$.
\end{enumerate}
\end{lemma}

Let us refer to a vertex of the closed hypercube $[0,1]^N$ as \emph{even} (resp. \emph{odd}) if it has an even (resp. odd) number of $1$s. The only vertex which lies in $\cA_N = [0,1)^N$ is the origin, which is even. For $d$ odd (resp. even), we see that precisely the odd (resp. even) vertices lie on the walls. In particular, for $d$ odd, the origin $\underline{0} \in \cA_N$ does not lie on any wall, whereas for $d$ even, the origin lies on the intersection of all the walls indexed by $s = 0$. 

Let us focus on the odd degree case for now and later describe the even degree case by performing a Hecke modification (see Example \ref{ex Hecke mod rank 2} and Remark \ref{rmk par vb rk 2 even deg}).

\begin{prop}\label{prop par WC}
For odd $d$, consider the walls $W_{s,\cI} \subset \cA_N = [0,1)^N$ in Lemma \ref{lemma walls}.
\begin{enumerate}
\item The centre of the hypercube $\underline{\frac{1}{2}} \in \cA_N$ lies on the wall $W_{s,\cI}$ if and only if $|\cI| = 2s+1 + \frac{N}{2}$. In particular, for odd $N$, the centre does not lie on any wall. 
\item The centre of the hypercube $\underline{\frac{1}{2}} \in \cA_N$ lies on $W_{s,\cI}$ if and only if this is a flopping wall.
\item Crossing the wall $W_{s,\cI}$ is a flip with centre $\Jac(C)^2$ and type $(n_{-}(s,\cI),n_{+}(s,\cI))$ where
\begin{equation}\label{compute n in terms of k and I}
 n_{-}(s,\cI)  = g + |\cI| -2s - 3 \quad \text{and} \quad n_{+}(s,\cI)  = g + N -|\cI| +2s - 1.
\end{equation}
In particular, the types $n_{\pm}(s,\cI)$ (i.e. the dimensions of the projective space appearing over the centre) satisfy $g-1 \leq n_{\pm}(s,\cI) \leq g + N - 3$. 
\item For any wall $W_{s,\cI}$ not containing $\underline{\frac{1}{2}}$, crossing this wall in the direction from the centre of the hypercube is a flip which decreases the canonical divisor. More precisely, if  $\beta, \gamma \in \cA_N$ are weights in chambers on either side of $W_{k,\cI}$  and we suppose that the line joining $\beta$ and $ \gamma$ meets $W_{k,\cI}$ in a point $\alpha$ which lies on no other wall, then if $\beta$ denotes the weight nearest to $\underline{\frac{1}{2}}$ we have a wall-crossing flip $\cN^\beta >_K \cN^\gamma$.
\end{enumerate}
\end{prop}
\begin{proof}
The first statement is immediate from the explicit equations of the walls given in Lemma \ref{lemma walls}. For the third statement, we use Theorem \ref{thm geom flips par} to describe the flip for crossing the wall $W_{s,\cI}$: the centre of this flip is a product of two moduli spaces for rank $1$ parabolic vector bundles, thus is a product of Jacobians, and the type of the flip is given by $n_{\pm}(s,\cI):=n_{\pm}(\eta'(s,\cI))$, which we can compute using \eqref{equation for n plus} to give the formula in \eqref{compute n in terms of k and I}. For the final part of the third statement, since $2 \leq |\cI| - 2s \leq N$, we obtain inequalities
\[ g-1 \leq n_{\pm}(s,\cI) \leq g + N -3.\]
Note that $n_{-}(s,\cI) = n_{+}(s,\cI)$ if and only if $|\cI| = 2s+1 + \frac{N}{2}$, which proves the second statement. 

The final statement will follow from the claim that for each wall $W_{s,\cI}$, we have
\[ 4\left(\mu_{\underline{\frac{1}{2}}}(\eta'(s,\cI)) - \mu_{\underline{\frac{1}{2}}}(\eta) \right)= n_{+}(s,\cI) - n_{-}(s,\cI) =  2 (2s+1) + N - 2 | \cI|.\]
Indeed we have 
\[ \mu_{\underline{\frac{1}{2}}}(\eta'(s,\cI)) - \mu_{\underline{\frac{1}{2}}}(\eta)= \frac{2s+1 +d}{2} + \sum_{i=1}^N \frac{1}{2} m'(\cI)_{i,2} - \frac{d + \sum_{i=1}^N \frac{1}{2} m_{i,2}}{2} = \frac{1}{4} \left( 2 (2s+1) + N - 2 | \cI|\right).\]
To conclude the final statement, we then use Theorem \ref{thm geom flips par}: if $n_{+}(s,\cI) - n_{-}(s,\cI) >0$, then $\beta = \alpha(+)$ and $\gamma = \alpha(-)$ and if $n_{+}(s,\cI) - n_{-}(s,\cI) <0$, then $\beta = \alpha(-)$ and $\gamma = \alpha(+)$. In both cases, we obtain $\cN^\beta >_K \cN^\gamma$ and so crossing a wall away from the centre decreases $K$.
\end{proof}

\subsubsection{Group actions on the space of weights}

In this section we exploit certain symmetries of the weight space $\cA_N$ by introducing two group actions on $\cA_N$ which enable us to significantly reduce the number of wall-crossings we must study. The first group action arises from Hecke modifications and the second group action arises from permutations of the weights. 

By $\S$\ref{sec Hecke mod} (see Example \ref{ex Hecke mod rank 2}), performing a Hecke modification at a subset $D' \subset D$ gives an isomorphism 
\[ \cH_D: \cN^\alpha(\eta) \cong \cN^{\alpha(D')}(\eta(D')) \]
where $\eta(D')= (2,d - |D'|, m)$ and (recalling we set $\alpha_{i,1} = 0$ and $\alpha_i :=\alpha_{i,2}$ for all $i$)
\[ \alpha(D')_i =\left\{ \begin{array}{ll} \alpha_i  & \text{if } p_i \notin D'  \\ 1 - \alpha_i & \text{if } p_i \in D' \end{array} \right. \]
If we perform a Hecke modification at a subset $D'$ of even cardinality $|D'| = 2e$, then after tensoring by a line bundle of degree $e$, we obtain an isomorphism
\[ \cN^\alpha(\eta) \cong \cN^{\alpha(D')}(\eta)\]
between moduli spaces for the same invariants $\eta$ but different weights. This gives a natural action on the weight space $\cA_N$ by Hecke modifications at an even number of points: this group is generated by Hecke modifications at the points $(p_i,p_{i+1})$ for $1 \leq i \leq N-1$. Therefore, we consider the group $H \cong (\ZZ/2\ZZ)^{N-1}$ acting on $\cA_N$ by, for $1 \leq i \leq N-1$, the involution
\[ (\alpha_1, \dots, \alpha_N) \mapsto (\alpha_1, \dots ,\alpha_{i-1},1-\alpha_i,1-\alpha_{i+1}, \alpha_{i+2}, \dots , \alpha_N). \]
Note that any weights in the same $H$-orbit have isomorphic moduli spaces and the $H$-action sends walls to walls and chambers to chambers.

We also have the natural permutation action of the symmetric group $S_N$ on $\cA_N = [0,1)^N$. Although, the corresponding moduli spaces for weights in the same $S_N$-orbit are not necessarily isomorphic, this action sends wall to walls and chambers to chambers, and preserves the numerical data of the wall-crossings in the following sense.

\begin{lemma} 
The action of $\sigma\in S_N$ on $\cA_N$ sends $W_{s,\cI}$ to $W_{s,\sigma(\cI)}$ and preserves the type of the wall-crossing, in the sense that $n_{\pm}(\eta'(s,\cI)) = n_{\pm}(\eta'(s, \sigma(\cI)))$.
\end{lemma}
\begin{proof}
This follows as $n_{\pm}(\eta'(s,\cI))$ in \eqref{compute n in terms of k and I} only depend on $g$, $s$ and the cardinality $|\cI|$.
\end{proof}

We can utilise these actions to reduce the number of wall-crossings we must consider. 

\begin{prop}\label{prop standard par walls}
Let $d$ be an odd degree. A set of representatives for the walls under the
$H \rtimes S_N$-action is given by, for $0 \leq l \leq N/2 -1$, the wall $W(l)$ defined by
\[ 1 = -\sum_{i=1}^l \alpha_i + \sum_{i=l+1}^N \alpha_i \]
corresponding to the invariants $\eta(l) =(1,(d+1)/2,m(l))$ with $m(l)_{i,1} = 0$ for $i \leq l$ and $m(k)_{i,1} = 1$ for $i > l$.
\end{prop}
\begin{proof}
Let us start with a wall of the form 
\[ 2s+1 = \sum_{i \in \cI} \alpha_i - \sum_{i \in \cI^c} \alpha_i \]
of Lemma \ref{lemma walls}. If $s > 0$, then as $|\cI| > 2s + 1$, we pick a subset $\cJ \subset \cI$ of cardinality $2s$ and perform a Hecke modification at the set $E = \{ p_j : j \in \cJ \} \subset D$. This wall transforms to
\[ 1 = \sum_{i \in \cK} \alpha_i - \sum_{i \in \cK^c} \alpha_i \]
with  $\cK = \cI \setminus \cJ$ of cardinality $|\cK| > 1$. If $|\cK^c| \leq \frac{N}{2} -1$, then after a permutation we can put it in the above form. If $|\cK^c| > \frac{N}{2} -1$, then $|\cK| < \frac{N}{2} +1$. After performing a Hecke modification at a set $\cJ \subset \cK$ of cardinality $2$ and multiplying the equation by $-1$, we obtain a wall
\[ 1 =- \sum_{i \in \cK \setminus \cJ} \alpha_i + \sum_{i \in \cK^c \cup \cJ} \alpha_i\] 
with $|\cK \setminus \cJ| < \frac{N}{2} -1$, and after a permutation we can put it in the above form.
\end{proof}

Recall that each wall $W(l)$ separates $\RR^N \setminus W(l)$ into two half-spaces $H(l)^{\pm}$ such that for $\alpha \in H(l)^{\pm}$, we have $\pm(\mu_\alpha(\eta(l)) - \mu_\alpha(\eta)) > 0$. For odd $d$, note that $\underline{0} \in H(l)^+$. 

\begin{prop}\label{prop flip type standard par walls}
For an odd degree $d$ and $0 \leq l \leq N/2 -1$, crossing the wall $W(l)$ at a point on no other wall from $H(l)^+$ to $H(l)^-$ is a flip of type $(g+l-1,g+N-3 -l)$ with centre $\Jac(C)^2$. In particular, for $l< N/2 -1$, this flip increases $K$ and for $N$ even and $l = N/2 -1$, this wall-crossing is a flop and so preserves $K$.
\end{prop}
\begin{proof}
This follows from Theorem \ref{thm geom flips par} and Equation \eqref{compute n in terms of k and I} in Proposition \ref{prop par WC}.
\end{proof}

We recall that in the wall and chamber decomposition of $\cA_N$ for an odd degree $d$, only the odd vertices in $[0,1]^N$ lie on the walls. Moreover the centre of the hypercube lies on a wall if and only if $N$ is even. We make the following definition of minimal and maximal chambers, whose terminology will be justified in Proposition \ref{prop max and min chambers} below.

\begin{defn}
For an odd degree $d$, a chamber in $\cA_N^\circ = (0,1)^N$ is said to be
\begin{enumerate}
\item \emph{minimal} if it contains an even vertex in its closure,
\item \emph{maximal} if it contains the centre of the hypercube in its closure.
\end{enumerate}
\end{defn}

The minimal are the chambers lying on the exterior of the hypercube. Note that if $N$ is odd, then there is a unique maximal chamber containing the centre of the hypercube. 

\begin{prop}\label{prop max and min chambers}
Consider the $K$-ordering on chambers in $\cA_N$ given by
\[ C_1 >_K C_2 :\iff \cN^{\alpha_1} >_K \cN^{\alpha_2} \text{ for } \alpha_i \in C_i.\]
For an odd degree $d$, the $K$-minimal (resp. $K$-maximal) chambers are the minimal (resp. maximal) chambers defined above. Moreover, the moduli spaces $\cN^\alpha$ for all weights $\alpha$ in minimal chambers are isomorphic and the moduli spaces $\cN^\alpha$ for all weights $\alpha$ in maximal chambers are related by a sequence of flops.
\end{prop}
\begin{proof} The first part follows the last statement in Proposition \ref{prop par WC}. Since Hecke modifications can be used to identify all the minimal chambers, their corresponding moduli spaces are isomorphic. The maximal chambers are all related by crossing walls through the centre of the hypercubes, which are precisely the flopping walls by Proposition \ref{prop par WC}.
\end{proof}

\subsubsection{Explicit formulas for the motive}

\begin{prop}\label{prop motive min and max}
Let $\alpha \in \cA_N^\circ = (0,1)^N$ be generic weight and let $\fh(\cN^\alpha)$ denote the integral Chow motive of the moduli space $\cN^\alpha$ of $\alpha$-semistable rank $2$ and odd degree $d$ parabolic bundles.
\begin{enumerate}
\item If $\alpha = \alpha_{\min}$ is in a minimal chamber, then there is an explicit isomorphism
\[ \fh(\cN^{\alpha_{\min}}) \simeq \fh(\cN) \otimes \fh(\PP^1)^{\otimes N}.  \]
\item If $\alpha = \alpha_{\max}$ is in a maximal chamber, then there is an explicit isomorphism
\[ \fh(\cN^{\alpha_{\max}}) \simeq \fh(\cN) \otimes \fh(\PP^1)^{\otimes N} \oplus \bigoplus_{\begin{smallmatrix}0\leq s < M \\ 2s \leq l < 2M \end{smallmatrix}}\: \: \bigoplus_{j=l}^{N-3 - l} \fh(\Jac(C))^{\otimes 2}(g+j)^{\oplus {N \choose  l-2s}} \]
where $M := (N-2)/4$.
\item For arbitrary generic $\alpha$, there is an explicit isomorphism
\[ \fh(\cN^{\alpha})  \simeq \fh(\cN) \otimes \fh(\PP^1)^{\otimes N} \oplus \bigoplus_{j=0}^{N-3} \fh(\Jac(C))^{\otimes 2}(g+j)^{\oplus b_j}\]
such that the exponents $b_j \in \NN$ can be inductively computed (see Theorem \ref{thm par motive rk 2 odd deg}, for a closed formula). In particular, this Chow motive is a direct factor of $\fh(\cN^{\alpha_{\max}})$ and contains $\fh(\cN^{\alpha_{\min}})$ as a direct summand.
\end{enumerate}
\end{prop}
\begin{proof}
Let us first consider the minimal chambers. It suffices to consider a weight $\alpha$ in the minimal chamber containing the origin in its closure, as all other minimal chambers are related by Hecke modifications and so their corresponding moduli spaces are isomorphic (\textit{cf.}\ Proposition \ref{prop max and min chambers}). By successively degenerating the flag as in Theorem \ref{thm geom flag degen}, we obtain a forgetful morphism
\[ \cN^\alpha  \ra \cN\]
which is a $(\PP^1)^N$-fibration, which gives the above description of $\fh(\cN^{\alpha_{\min}})$ (\textit{cf.}\ Corollary \ref{cor full deg flag motive}).

To describe the Chow motives of the moduli spaces for weights in maximal chambers, it suffices to consider one maximal chamber, as the moduli spaces $\cN^\alpha$ for all weights $\alpha$ in maximal chambers are related by a sequence of flops (\textit{cf.}\ Proposition \ref{prop max and min chambers}). We will compute this using a sequence of flips from a minimal chamber to a maximal chamber. The line segment from the origin to the centre of the hypercube crosses the wall $W_{s,\cI}$ (given in Lemma \ref{lemma walls}) if and only if $2s+1 \leq |\cI|  - \frac{N}{2}$, with equality if and only if the wall-crossing happens at the centre $\underline{\frac{1}{2}}$ (i.e this is a flopping wall by Proposition \ref{prop par WC}). By replacing the centre of the hypercube by a small perturbation of the form $(\frac{1}{2} - \epsilon_1, \dots , \frac{1}{2} -\epsilon_N)$, we can ensure this line segment goes from a minimal chamber to a maximal chamber by crossing one wall at a time and does not cross the flopping walls. Hence this line segment crosses the walls $W_{s,\cI}$ with $2s+1 < |\cI|  - \frac{N}{2}$  (or equivalently $|\cI^c| < \frac{N}{2} - (2s+1)$). Since $|\cI| \leq N$, it follows that $2s+1 <\frac{N}{2}$. We can enumerate these walls as follows: for each $0 \leq s < M:=\frac{N-2}{4}$ and each $0 \leq t < \frac{N}{2} - (2k+1)$ and each subsets $\cI^c \subset \{1, \dots , n\}$ of cardinality $t$, we cross the wall $W_{s,\cI}$. Each of these wall-crossings is a flip whose type is explicitly described in Proposition \ref{prop par WC} and so by applying Corollary \ref{cor motivic par WC}, we obtain the above formula. More precisely, any such wall $W_{s,\cI}$ is transformed by Hecke modifications and permutations to a standard wall of the form $W(|\cI^c| + 2s)$ in Proposition \ref{prop standard par walls}. Since $0 \leq |\cI^c| < \frac{N}{2} - (2s+1)$, we have $2s \leq |\cI^c| + 2s < 2M$. Thus for $0 \leq s < M$ and $2s \leq l < 2M$ we have ${N \choose l - 2s}$ wall-crossings of the type $W(l)$, which is a flip of type $(g-1, g+N-3 -l)$ with centre $\Jac(C)^2$ by Proposition \ref{prop flip type standard par walls}.

For an arbitrary generic weight $\alpha$, consider the ray from the centre of the hypercube through $\alpha$ to the edge of the hypercube. By perturbing $\alpha$, we can assume this ray passes each wall one by one. By Proposition \ref{prop par WC}, we know that moving along this ray from the centre to the edge of the hypercube decreases $K$. Therefore, we consider the corresponding sequence of flips running backwards along this ray, starting at a weight $\alpha_{\min}$ in a minimal chamber (at the exterior of the cube) and passing through $\alpha$ and then ending at a weight $\alpha_{\max}$ in a maximal chamber. All of these wall-crossings are flips with centre $\Jac(C)^2$ and the types are specified by Proposition \ref{prop par WC}. Hence, by iterative applications of Corollary \ref{cor motivic par WC}, one obtains the final statement.
\end{proof}

We can more explicitly compute the exponents appearing above using a method inspired by Bauer's computation \cite{Bauer} of the Poincar\'{e} polynomials of moduli spaces of parabolic vector bundles on $\PP^1$. In fact, for parabolic vector bundles on a curve $C$ of positive genus, the wall-crossing picture and associated combinatorics is the same, except for the fact that rather than the inductive description starting with a chamber where the parabolic moduli space in $g = 0$ is empty, we start with a chamber where the parabolic moduli space is related to the moduli space $\cN= \cN_C(2,d)$ of semistable vector bundles on $C$.

\begin{defn}\label{def exponents in terms of alpha}
For a weight $\alpha$ and $0 \leq j \leq N$, define
\[
d_j(\alpha)  =\# \left\{ \cI \subset \{1,\dots,N\} :   |\cI | \equiv j \: \text{ mod } \: 2 \text{, and }   -1+j < |\cI| + \sum_{i \in \cI^c} \alpha_i - \sum_{i \in \cI} \alpha_i < 1 + j \right\} \]
and
\[ b_j(\alpha)  = \sum_{i=0}^j \left\lfloor \frac{i + 2}{2} \right\rfloor c_{j-i}(\alpha), \: \text{ where } \: c_j(\alpha) = {N \choose j} -  d_j(\alpha) .\]
\end{defn}

The relationship between the constants $b_j(\alpha)$ and $d_j(\alpha)$ is given by the following equality of polynomials in $x$
\begin{equation}\label{eq b and d}
 (1 -x)(1-x^2)\sum_{j=0}^{N-3} b_j(\alpha)x^j = (1+x)^N + \sum_{j=0}^N d_j(\alpha) x^j.
\end{equation}

\begin{thm}\label{thm par motive rk 2 odd deg}
For a generic weight $\alpha \in \cA_N^\circ$, the integral Chow motive of the moduli space of rank $2$ and odd degree $\alpha$-semistable parabolic vector bundles is given by
\[ \fh(\cN^{\alpha})  \simeq \fh(\cN) \otimes \fh(\PP^1)^{\otimes N} \oplus \bigoplus_{j=0}^{N-3} \fh(\Jac(C))^{\otimes 2}(g+j)^{\oplus b_j(\alpha)}\]
with the exponents $b_j(\alpha)$ given in Definition \ref{def exponents in terms of alpha}.
\end{thm}
\begin{proof}
By inductively crossing walls, it suffices to check this formula holds for a particular chamber and then show that the formula respects wall-crossings. First consider $\alpha = (\epsilon_1, \dots , \epsilon_N) \in (0,1)^N$ in the minimal chamber very close to the origin. Then
\[ d_j(\alpha) = \# \left\{ \cI \subset \{1,\dots,N\} : |\cI| = j \right\} = { N \choose j } \text{ and } c_j(\alpha) = b_j(\alpha) = 0, \] 
which agrees with the formula for $\fh(\cN^{\alpha_{\min}})$ in Proposition \ref{prop motive min and max}. 

Now suppose we have two generic weights $\alpha_{\pm}$ which are separated by a single wall $W_{s, \cI}$ defined by $\eta':=\eta'(s,\cI)$ as in Lemma \ref{lemma walls}. Suppose that $\alpha_{\pm} \in H_{s,\cI}^\pm$; thus
\[ \mu_{\alpha_-}(\eta') - \mu_{\alpha_-}(\eta) < 0 < \mu_{\alpha_+}(\eta') - \mu_{\alpha_+}(\eta)\]
or equivalently
\[ 2s+1 + \sum_{i \in \cI^c} (\alpha_{-})_i - \sum_{i \in \cI} (\alpha_-)_i< 0 < 2s+1 + \sum_{i \in \cI^c}  (\alpha_{+})_i - \sum_{i \in \cI}  (\alpha_{+})_i. \]
This implies
\[ \sum_{i \in \cI^c}  (\alpha_{-})_i - \sum_{i \in \cI}  (\alpha_{-})_i + |\cI| < |\cI| - 2s -1 <  \sum_{i \in \cI^c}  (\alpha_{+})_i - \sum_{i \in \cI}  (\alpha_{+})_i + |\cI| \]
and also for $\cJ = \cI^c$ we have
\[ \sum_{i \in \cJ^c}  (\alpha_{+})_i - \sum_{i \in \cJ}  (\alpha_{+})_i + |\cJ| < |\cJ| + 2s +1 = N - |\cI| + 2s +1 < \sum_{i \in \cJ^c}  (\alpha_{-})_i - \sum_{i \in \cJ}  (\alpha_{-})_i + |\cJ|. \]
Hence, the only differences between the values of $d_j(\alpha_-)$ and $d_j(\alpha_+)$ are as follows
\begin{align*}
d_{|\cI| - 2s -2}(\alpha_-)  &=  d_{\cI - 2s -2}(\alpha_+) +1   \quad & d_{|\cI| - 2s }(\alpha_+)  & =  d_{\cI - 2s }(\alpha_-) +1 \\
d_{N - |\cI| + 2s }(\alpha_+)  &=  d_{N - |\cI| + 2s }(\alpha_-) +1  \quad &
d_{N - |\cI| + 2s +2  }(\alpha_-)  & =  d_{N - |\cI| + 2s +2 }(\alpha_+) +1.
\end{align*}
From \eqref{eq b and d}, we obtain a corresponding relationship between the $b_j(\alpha_{\pm})$ 
\begin{equation}\label{poly in x for bs}
(1-x)(1-x^2)\sum_{j=0}^{N-3}(b_j(\alpha_+) - b_j(\alpha_-))x^{j} = x^{|\cI| -2s} -x^{|\cI| - 2s -2} + x^{N - |\cI| + 2s} - x^{N - |\cI| + 2s +2}. 
\end{equation}

By Proposition \ref{prop standard par walls}, crossing the wall $W_{s,\cI}$ from $\alpha_-$ to $\alpha_+$ is a flip with centre $\Jac(C)^2$ of type $(n_-,n_+) = (g + | \cI | -2s -3, g + N -| \cI | + 2s -1)$. By taking Poincar\'{e} polynomials of the corresponding equation in Corollary \ref{cor motivic par WC}, we obtain 
\[ P_t(\cN^{\alpha_+}) - P_t(\cN^{\alpha_-}) = P_t(\Jac(C))^2) \left(P_t(\PP^{n_-}) - P_t(\PP^{n_+})\right) = (1 + t)^{4g} \frac{t^{2(n_+ +1)}-t^{2(n_- +1)} }{1-t^2}.\]
Thus to prove the claimed formula respects this wall-crossing, we need to show
\[ (1-t^2)\sum_{j=0}^{N-3}(b_j(\alpha_+) - b_j(\alpha_-)t^{2j} =t^{2(N - | \cI| +2s)} - t^{2(|\cI| -2s -2)} . \]
After multiplying this by $(1-t^4)$, we see this holds by inserting $x = t^2$ in \eqref{poly in x for bs}.
\end{proof}

\begin{cor} 
For a generic weight $\alpha \in \cA_N^\circ$, the rational Chow motive of the moduli space $\cN^{\alpha}$ of $\alpha$-semistable parabolic vector bundles of rank $2$ and odd degree $d$ is given by
\begin{align*}
  \fh(\cN^{\alpha})  \simeq & \left(\bigoplus_{j=0}^{N-3} \fh(\Jac(C))^{\otimes 2}(g+j)^{\oplus b_j(\alpha)} \right) \oplus \fh(\Jac(C)) \otimes \fh(\PP^1)^{\otimes N} \otimes \fh(\Sym^{g-1}(C))(g-1) \\
  & \oplus  \bigoplus_{i=0}^{g-2} \fh(\Jac(C)) \otimes \fh(\PP^1)^{\otimes N} \otimes \fh(\Sym^i(C)) \otimes  \left( \QQ(i) \oplus \QQ(3g-3-2i)\right)
\end{align*}
with the exponents $b_j(\alpha)$ given in Definition \ref{def exponents in terms of alpha}.
\end{cor}

For $\cL \in \Pic^d(C)$, we have the moduli space $\cN^\alpha_{\cL}$ of $\alpha$-semistable rank $2$ parabolic vector bundles with determinant $\cL$. The wall-crossing picture remains the same except that the centre of each flip/flop between fixed determinant parabolic moduli spaces is only one copy of the Jacobian. Consequently, by the same argument we obtain the following formula.

\begin{cor}\label{cor par motive rk 2 odd det}
For a generic weight $\alpha \in \cA_N^\circ$, the integral Chow motive of the moduli space of $\alpha$-semistable parabolic vector bundles of rank $2$ and odd degree determinant $\cL$ is given by
\[ \fh(\cN^{\alpha}_{\cL})  \simeq \fh(\cN_{\cL}) \otimes \fh(\PP^1)^{\otimes N} \oplus \bigoplus_{j=0}^{N-3} \fh(\Jac(C))(g+j)^{\oplus b_j(\alpha)}\]
with the exponents $b_j(\alpha)$ given in Definition \ref{def exponents in terms of alpha}.
\end{cor}

\begin{rmk}\label{rmk par vb rk 2 even deg}
For $n = 2$ and an even degree $d$, one can perform a Hecke modification at a single point $p_i$ to obtain an isomorphism $\cN^\alpha(2,d) \cong \cN^{\alpha'}(2,d-1)$, where $\alpha'=\alpha(p_i)$ (see Example \ref{ex Hecke mod rank 2}). Therefore, for $n=2$, $d$ even and $\alpha$ generic, we also obtain formulas for the Chow motive of $\cN^\alpha(2,d)$.
\end{rmk}

\begin{ex}
Let us explicitly compute some examples for low values of $N$.
\begin{enumerate}
\item For $N = 2$, we have $\cA_2^\circ = (0,1)^2$ and the following picture shows the wall and chamber decomposition for $d$ odd (on the left) and $d$ even (on the right). 
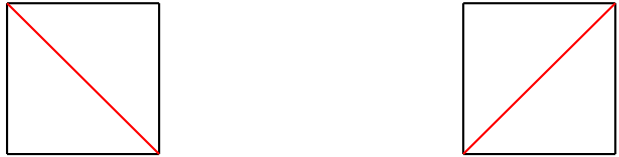
\begin{center}
\begin{tikzpicture}
\draw[black, thick] (-4,-1) -- (-2,-1);
\draw[black, thick] (-4,+1) -- (-2,+1);
\draw[black, thick] (-4,-1) -- (-4,+1);
\draw[black, thick] (-2,-1) -- (-2,+1);
\draw[red, thick] (-4,+1) -- (-2,-1);
\draw[black, thick] (4,-1) -- (2,-1);
\draw[black, thick] (4,+1) -- (2,+1);
\draw[black, thick] (4,-1) -- (4,+1);
\draw[black, thick] (2,-1) -- (2,+1);
\draw[red, thick] (2,-1) -- (4,+1);
\end{tikzpicture}
\end{center}
These pictures are related via the Hecke modification $(\alpha_1,\alpha_2) \mapsto (1-\alpha_1, \alpha_2)$ at $p_1$. There is one wall (shown in red), which is a flopping wall, as it goes through the center of the hypercube. Thus for any generic $\alpha$ and any $d$, we have
\[ \fh(\cN^{\alpha}) \simeq \fh(\cN) \otimes \fh(\PP^1)^{\otimes 2}. \]
In this case the minimal and maximal chambers coincide.
\item For $N=3$, we have $\cA_3^\circ = (0,1)^3$. For $d$ odd, there are 4 walls which cut out a tetrahedron inside this cube through the odd vertices $(1,0,0), (0,1,0), (0,0,1), (1,1,1)$ as in the diagram below.   
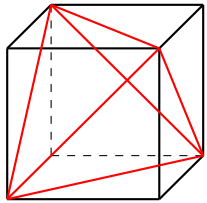
\begin{center}
\begin{tikzpicture}
\draw[black, dashed] (0,0,0) -- (2,0,0);
\draw[black, dashed] (0,0,0) -- (0,2,0);
\draw[black, dashed] (0,0,0) -- (0,0,1.5);
\draw[black, thick] (2,0,1.5) -- (2,0,0);
\draw[black, thick] (2,0,1.5) -- (0,0,1.5);
\draw[black, thick] (2,0,1.5) -- (2,2,1.5);
\draw[black, thick] (2,2,0) -- (0,2,0);
\draw[black, thick] (2,2,0) -- (2,2,1.5);
\draw[black, thick] (2,2,0) -- (2,0,0);
\draw[black, thick] (0,2,1.5) -- (0,2,0);
\draw[black, thick] (0,2,1.5) -- (0,0,1.5);
\draw[black, thick] (0,2,1.5) -- (2,2,1.5);
\draw[red, thick] (2,2,1.5) -- (2,0,0);
\draw[red, thick] (2,2,1.5) -- (0,2,0);
\draw[red, thick] (2,2,1.5) -- (0,0,1.5);
\draw[red, thick] (0,2,0) -- (2,0,0);
\draw[red, thick] (0,0,1.5) -- (2,0,0);
\draw[red, thick] (0,2,0) -- (0,0,1.5);
\end{tikzpicture}
\end{center}
In this case the interior of the tetrahedron is the maximal chamber and there are four exterior chambers, all of which are minimal. The motives of the parabolic moduli spaces for weights in the minimal and maximal chambers are given by
\begin{align*}
\fh(\cN^{\alpha_{\min}}) &\simeq \fh(\cN) \otimes \fh(\PP^1)^{\otimes 3} \\
\fh(\cN^{\alpha_{\max}}) & \simeq \fh(\cN) \otimes \fh(\PP^1)^{\otimes 3} \oplus \fh(\Jac(C))^{\otimes 2}(g),
\end{align*}
where the second formula is obtained from the first formula by crossing the back wall $ 1 = \alpha_1 + \alpha_2 + \alpha_3$ of the tetrahedron, which is of type $(g-1,g)$ with centre $\Jac(C)^2$.
\end{enumerate}
\end{ex}

\section{Motives of moduli spaces of (parabolic) Higgs bundles}

\subsection{Moduli spaces of Higgs bundles}

A \emph{Higgs bundle} on $C$ is a pair $(E,\Phi)$ consisting of a vector bundle $E$ and  an $\mathcal{O}_C$-linear homomorphism $\Phi\colon E\to E\otimes \omega_C$ called the \emph{Higgs field}. A Higgs subbundle of $(E,\Phi)$ is a subbundle of $E$ that is invariant under the Higgs field. There is a notion of semistability for Higgs bundles which involves showing that the slope is increasing on Higgs subbundles. Let $\cM=\cM(n,d)$ denote the moduli space of semistable rank $n$ degree $d$ Higgs bundles on $C$. The moduli space $\cM$ is a quasi-projective variety which contains the cotangent bundle to the moduli space $\cN=\cN(n,d)$ of semistable vector bundles as a dense open subvariety. Over $k = \CC$, $\cM$ is a non-compact hyper-K\"{a}hler manifold \cite{Hitchin}. Moreover, $\cM$ contains the moduli space $\cM^s$ of (geometrically) stable Higgs bundles as a smooth open subset. In this section, we will assume that $n$ and $d$ are coprime, so semistability and stability coincide and $\cM = \cM^s$ is a smooth quasi-projective variety. 

We can also fix $\cL \in \Pic^d(C)(k)$ and consider the moduli space $\cM_{\cL}=\cM_{\cL}(n,d)$ of semistable rank $n$ degree $d$ Higgs bundles on $C$ with fixed determinant $\cL$, which is also quasi-projective and is smooth when $n$ and $d$ are coprime.

\subsubsection{Motives of moduli spaces of Higgs bundles}

Although the smooth quasi-projective moduli spaces $\cM$ and $\cM_{\cL}$ are not proper, Hitchin \cite{Hitchin} and Simpson \cite{Simpson} observed that there is a $\GG_m$-action on both moduli spaces given by scaling the Higgs field such that the fixed locus is proper and the limit as $t \in \GG_m$ tends to zero exists for all points (such a $\GG_m$-action is referred to as being semi-projective, see \cite[Definition A.1]{HPL_Higgs}, and induces an associated Bia{\l}ynicki-Birula decomposition \cite{BB}). In particular, the flow under this action is a deformation retract and consequently the Voevodsky motives of $\cM$ and $\cM_{\cL}$ are pure (see \cite[Corollary 6.9]{HPL_Higgs} for the case of $\cM$) and we can consider their associated Chow motives. By \cite[Therorem 1.1]{HPL_Higgs}, the rational Chow motive $\fh(\cM)$ is contained in the subcategory of $\CHM(k,\QQ)$ generated by $\fh(C)$ and is a direct summand of the motive of a sufficiently large power of $C$.

Unlike in the case of the moduli spaces $\cN$ and $\cN_{\cL}$ of semistable vector bundles without and with fixed determinant whose motives are related by Theorem \ref{thm motive N and N fixed det}, for the Higgs moduli spaces $\cM$ and $\cM_{\cL}$ we have
\[ \fh(\cM) \not\simeq \fh(\cM_{\cL}) \otimes \fh(\Jac(C)). \]
Indeed, already for rank $n =2$, Hitchin shows this over $k = \CC$ on the level of singular cohomology \cite{Hitchin}. We will describe the motives of $\cM$ and $\cM_{\cL}$ for rank $n = 2$ below. In fact, we will show in Proposition \ref{prop:EnlargeCategory} that the motive of $\cM_{\cL}$ is in general \textit{not} contained in the tensor subcategory of $\CHM(k,\QQ)$ generated by $\fh(C)$.

\subsubsection{Formulas for the motive of the rank 2 Higgs moduli spaces}
\label{sec higgs rank 2 fixed det}

In this section, we consider rank $2$ Higgs bundles of odd degree $d$ and let $\cM = \cM(2,d)$ and $\cM_{\cL} = \cM_{\cL}(2,d)$ denote the moduli spaces of Higgs bundles without and with fixed determinant respectively.

\begin{thm}\label{thm motive rk 2 Higgs}
For an odd integer $d$, the integral Chow motive of the moduli space of semistable rank $2$ degree $d$ Higgs bundles on $C$ is given by an explicit isomorphism
\[ \fh(\cM(2,d)) \simeq \fh(\cN(2,d)) \oplus \bigoplus_{j=1}^{g-1} \fh(\Pic^{a_{d,j}}(C)) \otimes \fh(\Sym^{2j-1}(C))(3g - 2j -2 ), \]
where $a_{d,j} = g-j+(d-1)/2$. When working with rational coefficients, we have a non-explicit isomorphism
\begin{align*}
 \fh(\cM(2,d)) \simeq &\fh(\Jac(C)) \otimes \left(\fh(\Sym^{g-1}(C))(g-1) \oplus \bigoplus_{i=0}^{g-2} \fh(\Sym^i(C))\otimes \left( \QQ(i) \oplus \QQ(3g-3-2i)\right)\right) \\
  & \oplus \bigoplus_{j=1}^{g-1} \fh(\Pic^{a_{d,j}}(C)) \otimes \fh(\Sym^{2j-1}(C))(3g - 2j -2 ).
\end{align*}
\end{thm}
\begin{proof}
We prove the above formula using the motivic Bia{\l}ynicki-Birula decomposition associated to the $\GG_m$-action on $\cM(2,d)$ given by $t \cdot (E,\Phi) = (E,t\Phi)$. For this, we first need to describe the fixed locus and the associated geometric  Bia{\l}ynicki-Birula decomposition, which in rank $2$ was studied by Hitchin \cite{Hitchin}. We claim that
\[ \cM(2,d)^{\GG_m} \simeq \cN(2,1) \sqcup \bigsqcup_{j=1}^{g-1} \Sym^{2j-1}(C) \times \Pic^{a_{d,j}}(C). \]
Indeed if $(E,\Phi)$ is a $\GG_m$-fixed point, then either $\Phi = 0$ and $E$ is a semistable vector bundle, or $\Phi \neq 0$ and so we have $\GG_m \subset \Aut(E)$ giving a weight space decomposition $E = L_0 \oplus L_1$ as a sum of two line bundles such that $\Phi$ is given by a homomorphism $L_0 \ra L_1 \otimes \omega_C$. In this latter case, let $(e,d-e)$ denote the degrees of $(L_0,L_1)$; then since the Higgs field gives a non-zero section of $L_0^\vee \otimes L_1 \otimes \omega_C$, we have $d-2e + 2g -2 \geq 0$. Furthermore, stability of the Higgs bundle means that the Higgs subbundle $L_1 \subset (E,\Phi)$ has slope less than that of $E$ and so we must have $d < 2e$. Therefore, we have $(d+1)/2 \leq e \leq  g-1 + (d-1)/2$. For each $e$ in this range, the Higgs bundle is determined by a degree $e$ line bundle $L_0$ and an effective divisor of degree $d- 2e +2g -2$. After setting $j = g-e+(d-1)/2$, we get the fixed locus components $F_j:=\Sym^{2j-1}(C) \times \Pic^{a_{d,j}}(C)$ for  $1 \leq j \leq g-1$. 

For $1 \leq j \leq g-1$, let $\cM(2,d)^{+}_j$ denote the locus of points whose flow as $t \ra 0$ lies in $F_j$. The rank of the fibration $\cM(2,1)^{+}_j \ra F_j$ is $\frac{1}{2} \dim \cM(2,1) = \dim \cN(2,1)$, as the downward flow is Lagrangian; thus the codimension $c^+_j$ of $ \cM(2,1)^{+}_j$ is given by $c_j^+ = \dim \cN(2,1) - (2j-1 +g)$ (see \cite[Proposition 2.2]{HPL_Higgs} and Equation (2) in \textit{loc.\ cit.}\ as well as the references therein). The result then follows from the motivic Bia{\l}ynicki-Birula decomposition (see \cite[Theorem A.4]{HPL_Higgs}, where although the isomorphism is given in Voevodsky's triangulated category, all motives appearing are pure and so we can interpret this in the category of Chow motives) and Theorem \ref{thm motive N rank 2}.
\end{proof}

If $C$ admits a degree $1$ line bundle, then we have $\fh(\Pic^i(C))\simeq\fh(\Jac(C))$ for all $i$, and in Theorem \ref{thm motive rk 2 Higgs}, we obtain the same formula for all $d$. One can replace the Chow motives of $\Sym^{n}(C)$ for $g \leq n \leq 2g-2$ appearing in the formulas in Theorem \ref{thm motive rk 2 Higgs} with lower symmetric powers of $C$ and Tate twists of $\fh(\Jac(C))$ using \cite[Corollary 5.1]{Jiang19}. 

Similarly to above, one can use the motivic Bia{\l}ynicki-Birula decomposition for the moduli space $\cM_{\cL}(2,d)$ of Higgs bundles with fixed determinant. 

\begin{prop}\label{prop motive rk 2 higgs fixed det}
For $d$ odd and $\cL \in \Pic^d(C)$, the integral Chow motive of the moduli space of semistable rank $2$ Higgs bundles with determinant $\cL$ on $C$ is given by an explicit isomorphism
\begin{equation}
	\label{eqn:MotivePHBFixDet}
	\fh(\cM_{\cL}(2,d)) \simeq \fh(\cN_{\cL}(2,d)) \oplus \bigoplus_{j=1}^{g-1} \fh(\widetilde{\Sym}^{2j-1}(C))(3g - 2j -2 )
\end{equation}
where $\widetilde{\Sym}^{2j-1}(C) \ra {\Sym}^{2j-1}(C)$ is the degree $2^{2g}$ \'etale cover given by the base change of the multiplication-by-2 map on $\Jac(C)$.
\end{prop}
\begin{proof}
The proof is almost the same as in the case without fixed determinant, except now the fixed locus for the $\GG_m$-action is
\[ \cM_{\cL}(2,d)^{\GG_m} \simeq \cN_{\cL}(2,d) \sqcup \bigsqcup_{j=1}^{g-1} \widetilde{\Sym}^{2j-1}(C), \]
as the determinant is fixed (see \cite{Hitchin}).
\end{proof}

By combining this with Theorem \ref{thm motive N rank 2}, one obtains a non-explicit isomorphism describing the rational Chow motive of $\cM_{\cL}(2,d)$ with $d$ odd.

The following result contrasts with Proposition \ref{prop:N-abelian}.

\begin{prop}\label{prop:EnlargeCategory}
The rational Chow motive of $\cM_{\cL}(2,d)$ lies in the tensor subcategory of  $\CHM(k, \QQ)$ generated by the motive of the $2^{2g}$-cover $\pi\colon \tilde{C} \ra C$ given as the base change of the multiplication-by-2 map on $\Jac(C)$. Furthermore, over $k=\CC$, we have:
\begin{enumerate}[label=\emph{(\roman*)}, leftmargin=0.7cm]
	\item $\fh(\cM_{\cL}(2,d))$ lies in the tensor subcategory of  $\CHM(k, \QQ)$  generated by the motives of the \'etale double covers of $C$ (or equivalently, $C$ and the Prym varieties associated to the  \'etale double covers of $C$).
	\item For a general curve $C$ of genus $\geq 2$, $\fh(\cM_{\cL}(2,d))$ does \emph{not} belong to the tensor subcategory of $\CHM(k, \QQ)$ generated by the motive of $C$.
\end{enumerate}  	
\end{prop}
\begin{proof}	
For any $k\in \NN$, there is a surjective morphism $\Sym^k(\tilde{C})\to \tilde{\Sym}^k(C)$, thus $\fh( \tilde{\Sym}^k(C))$ is a summand of $\fh(\Sym^k(\tilde{C}))\simeq \Sym^k \fh(\tilde{C})$. Since also $\fh_1(C)$ is a direct factor of $\fh_1(\tilde{C})$, we deduce the first statement from Propositions \ref{prop:N-abelian} and \ref{prop motive rk 2 higgs fixed det}.

Let us now assume that $k = \CC$. We thank Salvatore Floccari and Zhi Jiang for their kind help on the proof. 
	\'Etale double covers (not necessarily connected) of $C$ (or $\Jac(C)$) are in bijection with the following abelian group 
	\[T:=\Hom(H_1(C, \ZZ), \ZZ/2\ZZ)=H^1(C, \ZZ/2\ZZ)\simeq (\ZZ/2\ZZ)^{\oplus 2g}.\]
	For any $t\neq 0\in T$, we denote by $\pi_t\colon C'_t\to C$ the corresponding \'etale double cover and denote by $P_t:=\operatorname{Prym}(C'_t/C)$ the associated $(g-1)$-dimensional Prym variety. By convention, $C_0'=C\coprod C$ and $P_0=\Jac(C)$.
	Note that the rank $2$ variation of Hodge structures $\pi_{t, *}\QQ_{C'_t}$ splits up as $ \QQ_{C}\oplus L_t$. In particular, there is an isomorphism of Hodge structures $H^1(P_t, \QQ)\simeq H^1(C, L_t)$.
	
	Since $\pi\colon \tilde{C} \ra C$ is an abelian cover, the variation of Hodge structures $\pi_*\QQ_{\tilde{C}}$ splits into a direct sum of rank one variations of Hodge structures:
	$$\pi_*\QQ_{\tilde{C}}\simeq \bigoplus_{t\in T}L_t,$$ where  $L_t$ is the local system on $C$ corresponding to the double cover $C'_t/C$.
 	 Therefore, we have isomorphisms of rational Hodge structures:
	\begin{equation*}
		H^1(\tilde{C}, \QQ) \simeq H^1(C, \pi_*\QQ_{\tilde{C}})\simeq \bigoplus_{t\in T} H^1(C, L_t)\simeq \bigoplus_{t\in T}H^1(P_t, \QQ).
	\end{equation*}
	Therefore we have an isogeny 
	\begin{equation}\label{eq isog prym} 
	\Jac(\tilde{C})\simeq_\QQ \prod_{t\in T}P_t
\end{equation}	
	Consequently, the tensor subcategory of $\CHM(k, \QQ)$ generated by $\fh(\tilde{C})$ is the same as the tensor subcategory generated by the motives of $C$ and all Prym varieties $P_t$ for $t\in T$. Let us denote this subcategory by $\langle \fh(P_t), t\in T\rangle$. This is also equivalent to the category generated by the motives of the double covers $C_t'$ of $C$. Since $\fh(\Sym^j(\tilde{C})) \in \langle \fh(P_t), t\in T\rangle$ for all $j \in \NN$ and $\fh(\cN_\cL(2, d))\in \langle\fh(C)\rangle$ (Proposition \ref{prop:N-abelian}), we deduce statement (i) from Equation \eqref{eqn:MotivePHBFixDet}.
	
	To prove (ii), it suffices to show that for some $t\in T\backslash\{0\}$, the Hodge structure  $H^1(P_t, \QQ)$ does not belong to the tensor subcategory generated by $H^1(C, \QQ)$. Assuming the contrary, by Tannakian duality, $H^1(P_t, \QQ)$ is a $(2g-2)$-dimensional  representation of the Hodge group of $C$, which is the symplectic group $\operatorname{Sp}(H^1(C, \QQ))\simeq \operatorname{Sp}_{2g}$ when $C$ is general. But this is absurd since a nontrivial representation of the symplectic group $ \operatorname{Sp}_{2g}$ is of dimension at least $2g$.	
\end{proof}

\begin{rmk}
We think that part (i) of this proposition also holds over field $k \neq \CC$, as we suspect the isogeny \eqref{eq isog prym} holds in greater generality. 
\end{rmk}

\subsection{Motives of moduli spaces of parabolic Higgs bundles}
\label{subsec:ParabolicHiggs}
Let $C$ be a smooth projective curve of genus $g$. Let  $p_1, \dots, p_N$ be $N$ distinct $k$-rational points on $C$ and denote $D=p_1+\cdots+p_N$. Assume that $2g-2+N\geq 0$, i.e.~$\omega_C(D)$ is nef.

\begin{defn}
A (quasi) \textit{parabolic Higgs bundle} on $(C, D)$ is a pair $(E_*,\Phi)$ consisting of
\begin{itemize}
	\item  a (quasi) parabolic vector bundle $E_*$ (see Definition \ref{def:ParaBun}) with full flag-type\footnote{One can also consider non-full flags, but we make this simplifying assumption so that all walls are good (see $\S$\ref{sec geom var stab par}) and so all wall-crossings can be explicitly described.} at each marked point $p_i$, $1\leq i\leq N$;
	\item  an  $\mathcal{O}_C$-linear homomorphism $\Phi\colon E\to E\otimes K_C(D)$, called a \textit{(strongly parabolic) Higgs field}, satisfying that $\Phi(E_{i,j})\subset E_{i, j+1}\otimes \omega_C(D)$, for any $1\leq i\leq N$ and any $1\leq j\leq \operatorname{rk}(E)$.
\end{itemize}
The notion of (semi-)stablity is defined similarly as in Definition \ref{def:StabilityParaBun} using $\alpha$-slopes, except that only parabolic Higgs subbundles need to be considered, i.e. those parabolic subbundles $F$ of $E_*$ that are $\Phi$-invariant: $\Phi(F)\subset F\otimes \omega_C(D)$. 
\end{defn}

For $n\in \NN^*, d\in \ZZ$ and a weight $\alpha$, Yokogawa \cite{Yokogawa, Yokogawa93} constructed the moduli space of $\alpha$-semistable parabolic Higgs bundles rank $n$ degree $d$, which we denote by $\mathcal{M}^\alpha(n, d)$. The stable locus $\mathcal{M}^{\alpha-s}(n, d)$ forms an open subset.

Similarly to the case of parabolic bundles discussed in \S\ref{sec:ParabolicBundles}, there is a wall and chamber structure in the weight space, such that the corresponding moduli space of stable parabolic Higgs bundles stays the same when varying the weight within a chamber, and undergoes a birational transform when crossing a wall. 
However, 
Boden--Yokogawa \cite{BY_ParHiggs} observed that moduli spaces of stable parabolic Higgs bundles have the same Betti numbers when crossing a wall, and they conjectured that the diffeomorphic type should also be preserved, which was proved by Nakajima \cite{NakajimaPHB} shortly after. Subsequently, Thaddeus \cite{Thaddeus_var_PHiggs} gave a more precise geometric picture: when crossing a wall in the weight space, the moduli space of stable parabolic Higgs bundles undergoes a very special birational transform, namely, a \textit{Mukai flop}; see \S \ref{subsec:MukaiFlop}, for the precise definition. 

For a generic weight $\alpha$, the moduli space $\mathcal{M}^\alpha(n, d)=\mathcal{M}^{\alpha-s}(n, d)$ is a smooth quasi-projective (and over $k = \CC$, hyper-K\"ahler) variety \cite{NakajimaPHB}, but in general it is non-proper. Neverthless, its Voevodsky motive is pure.

\begin{lemma}\label{lemma:par Higgs pure}
For $\alpha$ generic (i.e.~not on any wall), the Voevodsky motive of $\mathcal{M}^\alpha(n, d)$ lies in the subcategory of Chow motives. 
\end{lemma}
\begin{proof}
The argument is almost identical to the result for (non-parabolic) Higgs moduli spaces when $n$ and $d$ are coprime \cite[Corollary 6.9]{HPL_Higgs} and so we simply sketch the details. In the parabolic setting, there is also a $\GG_m$-action on parabolic Higgs moduli space $\mathcal{M}^\alpha(n, d)$ given by scaling the parabolic Higgs field. For generic $\alpha$, the parabolic Higgs moduli space $\mathcal{M}^\alpha(n, d)$ is a smooth quasi-projective variety. In this case, the $\GG_m$-action on the smooth variety $\mathcal{M}^\alpha(n, d)$ is semi-projective (see \cite[Definition A.1]{HPL_Higgs}) in the sense that the fixed locus is proper and the limit as $t \in \GG_m$ tends to zero exists for all points. Indeed the $\GG_m$-fixed points and flow are described in \cite[Theorem 8]{Simpson_harmonic}: the $\GG_m$-fixed loci are moduli spaces of chains of parabolic vector bundles for appropriate stability parameters, which are projective varieties by their GIT constructions. In particular, there is an associated Bia{\l}ynicki-Birula decomposition \cite{BB} of $\mathcal{M}^\alpha(n, d)$, and as the flow under this $\GG_m$-action is a deformation retract and the fixed loci are smooth projective varieties, the Voevodsky motive of $\mathcal{M}^\alpha(n, d)$ is pure (see \cite[Appendix A]{HPL_Higgs}).
\end{proof}

As a consequence of Theorem \ref{thm:MotiveMukaiFlop} combined with Thaddeus' aforementioned result \cite[6.2]{Thaddeus_var_PHiggs}, we obtain the following result.

\begin{cor}\label{cor motive par higgs indept of alpha}
Fix $(C, D)$ and $n$ and $d$. Then for a generic weight $\alpha$, the integral Chow motive of the moduli space $\cM^\alpha_{C,D}(n,d)$ of $\alpha$-semistable parabolic Higgs bundles of rank $n$ and degree $d$ is independent of $\alpha$.
\end{cor}

\subsubsection{Closed formula for motives of rank $2$ parabolic Higgs bundles}

In this section we consider moduli spaces parabolic Higgs bundles of rank $n = 2$ and odd degree $d$ with full flags at $N$ points $p_1, \dots p_N$. For a generic weight $\alpha$, we compute the Chow motive of $\cM^{\alpha} = \cM^{\alpha}(2,d)$.

\begin{thm}\label{thm motive rk 2 para higgs}
For a generic weight $\alpha$, we have an explicit isomorphism of integral Chow motives
\[ \fh(\cM^\alpha) \simeq \fh(\cN) \otimes \fh(\PP^1)^{\otimes N} \oplus \bigoplus_{\begin{smallmatrix}  0 \leq l \leq N  \\ \frac{l+1-N}{2} \leq j \leq g -1 \end{smallmatrix}} \fh(\Pic^{a_{d,j}}(C)) \otimes \fh(\Sym^{2j+N - l-1}(C))(3g -2j+l-2)^{\oplus {N \choose l }} \]
where $a_{d,j} := g-j+(d-1)/2$ and $\cN = \cN(2,d)$ is the moduli space of semistable vector bundles.
\end{thm}
\begin{proof}
We use the $\GG_m$-action on $\cM^\alpha$ given by $t \cdot (E_*,\Phi) = (E_*,t\Phi)$ and its associated motivic Bia{\l}ynicki-Birula decomposition to prove the above formula. Since $\fh(\cM^\alpha)$ is independent of $\alpha$ by Corollary \ref{cor motive par higgs indept of alpha}, we will take a particular choice of $\alpha$. 

By performing a linear shift, we can assume that $\alpha$ has the form $(\alpha_{i,1}, \alpha_{i,2}) = (0, \alpha_i)$ with $\alpha_i > 0$. The fixed points of this $\GG_m$-action in rank $2$ for trivial parabolic degree and fixed determinant is described in \cite{BY_ParHiggs} and the fixed locus in our setting is a minor modification of this. If $(E_*,\Phi)$ is a $\GG_m$-fixed point, then either $\Phi = 0$ and $E_*$ is an $\alpha$-semistable parabolic vector bundle or $\Phi \neq 0$ and so we have $\GG_m \subset \Aut(E)$ giving a weight space decomposition $E = L \oplus M$ as a sum of two line bundles such that $\Phi$ is given by a non-zero strongly parabolic homomorphism $\Phi : L \ra M \otimes \omega_C$. The parabolic structure and weights on $L$ and $M$ are induced by that of $E_*$ as follows. Let $d' = \deg(L)$ and $m'$ denote the multiplicity of $L \subset E_*$ given by $m'_i = \dim (L_{p_i} \cap E_{i,2}) \in \{ 0 ,1\}$. Then the weight of the induced flags in $L_{p_i}$ is $\alpha_i$ if $m_i' = 1$, and $0$ if $m_i' = 0$. Hence $\Phi : L \ra M \otimes \omega_C$ being strongly parabolic means $\Phi(L_{p_i}) = 0$ for all $i$ with $m_i' = 1$. Equivalently this means that $\Phi$ factors as
\[ \Phi : L \ra M \otimes \omega_C(D - \sum_{i=1}^N m_i'p_i) \hookrightarrow M \otimes \omega_C(D).\]
In this case $(E_* = L \oplus M,\Phi)$ is specified by $L \in \Pic^{d'}(C)$ and $ L \ra M \otimes \omega_C(D - \sum_{i=1}^N m_i'p_i)$, which corresponds to an effective divisor of degree $d-2d'+2g-2+N-|m'| \geq 0$, where $|m'|:= \sum_{i=1}^N m_i'$. Since  $(E_*,\Phi)$ is $\alpha$-stable, the parabolic Higgs subbundle $M_* \subset E_*$ satisfies $\mu_{\alpha}(M_*) < \mu_{\alpha}(E_*)$, or equivalently $\mu_{\alpha}(L_*) > \mu_{\alpha}(E_*)$, which gives 
\begin{equation}\label{eq M subHiggs}
2d'  > d + \sum_{i=1}^N \alpha_i (1-2m'_i)
\end{equation} 
Therefore the $\GG_m$-fixed locus with non-zero Higgs field is indexed by tuples $(m',d')$ such that $d + \sum_{i=1}^N \alpha_i (1-2m'_i) < 2d' \leq 2g - 2 +N + d - |m'|$ and the corresponding fixed locus is $\Pic^{d'}(C) \times \Sym^{2g+d-2d'+N-l-2}(C)$.

Now let us pick $\alpha_i $ very small, so that the moduli space $\cN^\alpha$ of $\alpha$-stable parabolic vector bundles is a $(\PP^1)^N$-bundle over the moduli space $\cN$ of stable vector bundles (\textit{cf.} Corollary \ref{cor full deg flag motive}) and so \eqref{eq M subHiggs} is equivalent to $2d' \geq d+1$ for all possible $m'$. If we set $j = g-d' +(d-1)/2$, then for this $\alpha$, we have fixed set
\[ (\cM^\alpha)^{\GG_m} = \cN^\alpha \sqcup \bigsqcup_{\begin{smallmatrix} m' \in \{ 0,1\}^N \\ \frac{l+1-N}{2} \leq j \leq g-1  \end{smallmatrix}} \Pic^{a_{d,j}}(C) \times \Sym^{2j+N-l-1}(C). \]
and the codimension of the Bia{\l}ynicki-Birula stratum indexed by $(m',j)$ is 
\[ c^+_{m',j} = \dim \cN^\alpha - (2j+N-l-1+ g) = 3g - 2j +l-2\]
since the downward flow is Lagrangian (as in the proof of Theorem \ref{thm motive rk 2 Higgs}). For $0 \leq l \leq N$, there are $N \choose l$ multiplicities $m'$ with $l = \sum_{i=1}^N m_i'$ and so the motivic Bia{\l}ynicki-Birula decomposition (see \cite[Theorem A.4]{HPL_Higgs}, again interpreted in the category of Chow motives) gives 
\[\fh(\cM^\alpha)\simeq \fh(\cN^\alpha) \oplus \bigoplus_{\begin{smallmatrix}   0 \leq l \leq N  \\ \frac{l+1-N}{2} \leq j \leq g-1 \end{smallmatrix}}\fh(\Pic^{a_{d,j}}(C) ) \otimes \fh(\Sym^{2j+N-l-1}(C))(3g - 2j +l-2)^{\oplus {N \choose l }}. \]
To conclude, we use the fact that $\cN^\alpha \ra \cN$ is a $(\PP^1)^N$-bundle (\textit{cf.}\ Corollary \ref{cor full deg flag motive}). \end{proof}

By combining Theorem \ref{thm motive rk 2 para higgs} with the formula for the rational Chow motive of $\cN(2,d)$ in Theorem \ref{thm motive N rank 2}, we obtain a formula for the rational Chow motive of $\cM^\alpha(2,d)$ in terms of sums and tensor products of Tate twists of motives of $\Pic^i(C)$ and $\Sym^j(C)$.

For even degree $d$ and generic $\alpha$, we can also compute the Chow motive of $\cM^\alpha(2,d)$ using a Bia{\l}ynicki-Birula decomposition, where the fixed locus $\cN^\alpha(2,d)$ is isomorphic to $\cN^{\alpha'}(2,d-1)$ by a Hecke modification at a single parabolic point, where $\alpha':= \alpha(p_i)$ (see Example \ref{ex Hecke mod rank 2} and Remark \ref{rmk par vb rk 2 even deg}). 

We can also obtain a formula for the Chow motives of moduli space $\cM_{\cL}^\alpha(2,d)$ of parabolic Higgs bundles with fixed determinant $\cL$. Indeed these are also invariant of $\alpha$ (for $\alpha$ generic), as the Mukai flops for $\cM^\alpha(2,d)$ restrict to $\cM_{\cL}^\alpha(2,d)$, where the centres are pullbacks of symmetric powers of $C$ under the multiplication-by-2 map on $\Jac(C)$. The proof of the following formula is essentially the same as Theorem \ref{thm motive rk 2 para higgs} using a modification as in Proposition \ref{prop motive rk 2 higgs fixed det}.
 
 \begin{prop}\label{prop par higgs rk 2 fixed det} 
For a generic weight $\alpha$, the Chow motive of moduli space $\cM_{\cL}^\alpha(2,d)$ of parabolic Higgs bundles with fixed determinant $\cL$ of odd degree is given by an explicit isomorphism
\[ \fh(\cM^\alpha_{\cL}) \simeq \fh(\cN_{\cL}) \otimes \fh(\PP^1)^{\otimes N} \oplus \bigoplus_{\begin{smallmatrix}  0 \leq l \leq N  \\ \frac{l+1-N}{2} \leq j \leq g -1 \end{smallmatrix}}  \fh(\widetilde{\Sym}^{2j+N - l-1}(C))(3g -2j+l-2)^{\oplus {N \choose l }} \] 
where $\widetilde{\Sym}^{i}(C) \ra {\Sym}^{i}(C)$ is the degree $2^{2g}$ \'etale cover given as the base change of  the multiplication-by-2 map on $\Jac(C)$.
 \end{prop}

\appendix
\section{A local-to-global trick}\label{appendix}
We present a local-to-global trick employed in \cite{LLW-annals} and \cite{FuWang08}, 
which sometimes allows one to reduce the problem of computing the change of motives or Chow groups under a birational transform to the same problem for a local model. 

Let $i\colon Z\hookrightarrow X$ be a closed immersion between smooth varieties. Let $\tau\colon\tilde{X}\to X$ be the blow-up along $Z$ and $E=\PP(N_{Z/X})$ the exceptional divisor. Define $X_{\loc}:=\PP_Z(N_{Z/X}\oplus \mathcal{O}_Z)$, the compactification of the total space of the vector bundle $N_{Z/X}$, with the infinite part $E$. We think of the inclusion by zero section $Z\hookrightarrow X_{\loc}$ as the local projective model of $i$. We summarise the situation in the following diagram.
\begin{equation}
\label{diag:LocalGlobalTrick}
	\begin{tikzcd}
X_{\loc}  \arrow[dr, swap, "\pi"]&E \arrow[l, hook', swap, "\iota"] \arrow[r, hook, "j"] \arrow[d, "p"]&\tilde{X} \arrow[d, "\tau"]\\
&Z\arrow[r, hook, "i"]& X	
	\end{tikzcd}
\end{equation}
\begin{prop}
	\label{prop:LocalGlobal}
With the above notation, we have the following isomorphism
\begin{equation}
\label{eq:LocalGlobalTrick}
(\tau^*, \pi^*i^*)\colon  \CH^k(X) \xrightarrow{\quad\simeq\quad} \frac{\ker\left(\CH^k(\tilde{X})\oplus \CH^k(X_{\loc})\xtwoheadrightarrow{(j^*, -\iota^{*})} \CH^k(E)\right)}{\im\left(\CH^{k-1}(E)\xhookrightarrow{\; (j_*, -\iota_*) \;}\CH^k(\tilde{X})\oplus \CH^k(X_{\loc})\right)},	
\end{equation}
whose inverse is given by $(\tau_*, i_*\pi_*)$.
\end{prop}
\begin{proof}
The statement and the argument are essentially contained in \cite[\S 4]{LLW-annals}.
The surjectivity of $(j^*, -\iota^*)$ follows from the surjectivity of $\iota^*$, and the injectivity  of $\iota_*$ implies the injectivity of $ (j_*, -\iota_*)$.

Denote by $\xi=c_1(\mathcal{O}_p(1))\in \CH^1(E)$. Since $N_{E/\tilde{X}}\simeq \mathcal{O}_p(-1)$, the composition $j^*j_*$ is multiplication by $-\xi$. 
Since  $N_{E/X_{\loc}}\simeq \mathcal{O}_p(1)$, the composition $\iota^*\iota_*$ is multiplication by $\xi$. Therefore, $(j^*, -\iota^*)\circ (j_*, -\iota_*)=0$, i.e.~
$\im(j_*, -\iota_*)\subset \ker(j^*, -\iota^*)$, and so the right-hand side of \eqref{eq:LocalGlobalTrick} makes sense.

To see that $(\tau^*, \pi^*i^*)$ is well-defined: $(j^*, -\iota^*)\circ (\tau^*, \pi^*i^*)=j^*\tau^*-\iota^*\pi^*i^*=0$, by the commutativity of Diagram \eqref{diag:LocalGlobalTrick}.
Similarly, one checks that $(\tau_*, i_*\pi_*)$ is well-defined.

The composition $(\tau_*, i_*\pi_*)\circ (\tau^*, \pi^*i^*)=\tau_*\tau^*+i_*\pi_*\pi^*i^*=\id$, thanks to the projection formula. It remains to show that $(\tau_*, i_*\pi_*)$ is injective. More explicitly, for any $a\in \CH^l(\tilde{X})$ and $b\in \CH^l(X_{\loc})$ satisfying that
\begin{align}
	j^*(a)&=\iota^*(b),\label{eqn:relation1}\\ 
	\tau_*(a)&=-i_*\pi_*(b),\label{eqn:relation2}
\end{align}
we want to show the existence of $\gamma\in \CH^{l-1}(E)$, such that $a=j_*(\gamma)$ and $b=-\iota_*(\gamma)$.

To this end, by abuse of notation, denote $\xi$ both $c_1(\mathcal{O}_\pi(1))$ and $c_1(\mathcal{O}_p(1))$ (note that the former does restrict to the latter). By the blow-up formula and the projective bundle formula, we can write
\begin{align}
	a&= \tau^*(a_0)+j_*(p^*(a_1)+\cdots+p^*(a_{e-1})\xi^{e-2});\label{eqn:develop-1}\\
	b&= \pi^*(b_0)+\pi^*(b_1)\xi+\cdots+ \pi^*(b_e)\xi^e,\label{eqn:develop-2}
\end{align}
for some $a_0\in \CH^{l}(X)$, $a_r\in \CH^{l-r}(Z)$ and $b_r\in \CH^{l-r}(Z)$, where $e:=\codim(Z\subset X)$.

Using  \eqref{eqn:relation1}, \eqref{eqn:develop-1}, \eqref{eqn:develop-2}, one obtains the following relations:
\begin{align}
i^*(a_0)&=b_0-b_ec_e(N); \label{eqn:relation1-1}\\
a_r+b_r&=b_ec_{e-r}(N),\quad 1\leq r\leq e-1. \label{eqn:relation1-2}
\end{align}
where $N:=N_{Z/X}$ is of rank $e$ and, for $E=\PP(N)$, we used the following identity in $\CH^*(E)$:
\[\xi^e+p^*c_1(N)\xi^{e-1}+\cdots+p^*c_e(N)=0.\]
By \eqref{eqn:relation2},  \eqref{eqn:develop-1}, \eqref{eqn:develop-2}, we have 
\begin{equation}\label{eqn:a0}
	a_0=-i_*(b_e),
\end{equation}
which implies that $i^*(a_0)=-i^*i_*(b_e)=-b_ec_e(N)$. Combining with \eqref{eqn:relation1-1}, we get $b_0=0$.
Using again \eqref{eqn:a0} and the excess intersection formula \cite[\S 6.3]{FultonBook}, 
\begin{equation}\label{eqn:Excess}
\tau^*(a_0)=-j_*(p^*(b_e)c_{e-1}(\mathcal{E})),
\end{equation}
where $\mathcal{E}:=p^*(N)/\mathcal{O}_p(-1)$ is the excess normal bundle of the blow-up square in Diagram \eqref{diag:LocalGlobalTrick}.
Putting \eqref{eqn:Excess} into \eqref{eqn:develop-1}, we obtain that $a= j_*(\gamma)$ with
\[\gamma:=-p^*(b_e)c_{e-1}(\mathcal{E})+p^*(a_1)+p^*(a_2)\xi+\cdots+p^*(a_{e-1})\xi^{e-2}.\]
Therefore, it remains to show that $-\iota_*(\gamma)=b$.
To this end,
\begin{align*}
	-\iota_*(\gamma)&=\iota_*(p^*(b_e)c_{e-1}(\mathcal{E}))-\iota_*(p^*(a_1)+\cdots+p^*(a_{e-1})\xi^{e-2}),\\
	&=\iota_*(p^*(b_e)c_{e-1}(\mathcal{E}))+\iota_*(p^*(b_1-b_ec_{e-1}(N)))+\cdots+\iota_*(p^*(b_{e-1}-b_ec_1(N))\xi^{e-2}),\\
	&=\pi^*(b_e)\iota_*c_{e-1}(\mathcal{E})+\pi^*(b_1-b_ec_{e-1}(N)))\xi+\cdots+\pi^*(b_{e-1}-b_ec_1(N))\xi^{e-1},\\
	&=b+\pi^*(b_e)\left(\iota_*c_{e-1}(\mathcal{E})-\pi^*(c_{e-1}(N))\xi-\cdots-\pi^*(c_1(N))\xi^{e-1}-\xi^e\right),\\
	&=b.
\end{align*}
where the second equality uses \eqref{eqn:relation1-2}, the third equality uses the projection formula (note that $p=\pi\circ\iota$) and the fact that $[E]=\xi \in \CH^1(X_{\loc})$, the fourth equality uses \eqref{eqn:develop-2} and that $b_0=0$, and the last equality uses the equality $$c_{e-1}(\mathcal{E})=p^*(c_{e-1}(N))+\cdots+p^*(c_1(N))\xi^{e-2}+\xi^{e-1},$$ which can be easily deduced from $c_t(N)=c_t(\mathcal{E})c_t(\mathcal{O}_p(-1))=c_t(\mathcal{E})(1-t\xi)$.
\end{proof}

\bibliographystyle{amsplain}
\bibliography{references}

\providecommand{\bysame}{\leavevmode\hbox to3em{\hrulefill}\thinspace}
\providecommand{\MR}{\relax\ifhmode\unskip\space\fi MR }
\providecommand{\MRhref}[2]{%
  \href{http://www.ams.org/mathscinet-getitem?mr=#1}{#2}
}
\providecommand{\href}[2]{#2}
\begin{thebibliography}{10}

\bibitem{SGA1}
\emph{Rev\^etements \'etales et groupe fondamental ({SGA} 1)}, Documents
  Math\'ematiques (Paris) [Mathematical Documents (Paris)], 3, Soci\'et\'e
  Math\'ematique de France, Paris, 2003.

\bibitem{AndreBourbaki}
Y.~Andr\'{e}, \emph{Motifs de dimension finie (d'apr\`es {S}.-{I}. {K}imura,
  {P}. {O}'{S}ullivan{$\dots$})}, no. 299, 2005, S\'{e}minaire Bourbaki. Vol.
  2003/2004, pp.~Exp. No. 929, viii, 115--145.

\bibitem{Atiyah_Elliptic}
M.~F. Atiyah, \emph{Vector bundles over an elliptic curve}, Proc. London Math.
  Soc. (3) \textbf{7} (1957), 414--452.

\bibitem{AB}
M.~F. Atiyah and R.~Bott, \emph{The {Y}ang-{M}ills equations over {R}iemann
  surfaces}, Philos. Trans. Roy. Soc. London Ser. A \textbf{308} (1983),
  no.~1505, 523--615.

\bibitem{BKN}
V.~Balaji, A.~D. King, and P.~E. Newstead, \emph{Algebraic cohomology of the
  moduli space of rank {$2$} vector bundles on a curve}, Topology \textbf{36}
  (1997), no.~2, 567--577.

\bibitem{Bauer}
S.~Bauer, \emph{Parabolic bundles, elliptic surfaces and {${\rm
  SU}(2)$}-representation spaces of genus zero {F}uchsian groups}, Math. Ann.
  \textbf{290} (1991), no.~3, 509--526.

\bibitem{Beauville_diag}
A.~Beauville, \emph{Sur la cohomologie de certains espaces de modules de
  fibr\'{e}s vectoriels}, Geometry and analysis ({B}ombay, 1992), Tata Inst.
  Fund. Res., Bombay, 1995, pp.~37--40.

\bibitem{Beauville_AV}
Arnaud Beauville, \emph{Sur l'anneau de {C}how d'une vari\'{e}t\'{e}
  ab\'{e}lienne}, Math. Ann. \textbf{273} (1986), no.~4, 647--651. \MR{826463}

\bibitem{BD}
K.~Behrend and A.~Dhillon, \emph{On the motivic class of the stack of bundles},
  Adv. Math. \textbf{212} (2007), no.~2, 617--644.

\bibitem{BB}
A.~Bia{\l}ynicki-Birula, \emph{Some theorems on actions of algebraic groups},
  Ann. of Math. (2) \textbf{98} (1973), 480--497.

\bibitem{BGL}
E.~Bifet, F.~Ghione, and M.~Letizia, \emph{On the {A}bel-{J}acobi map for
  divisors of higher rank on a curve}, Math. Ann. \textbf{299} (1994), no.~4,
  641--672.

\bibitem{BH}
H.~U. Boden and Y.~Hu, \emph{Variations of moduli of parabolic bundles}, Math.
  Ann. \textbf{301} (1995), no.~3, 539--559.

\bibitem{BY_ParHiggs}
H.~U. Boden and K.~Yokogawa, \emph{Moduli spaces of parabolic {H}iggs bundles
  and parabolic {$K(D)$} pairs over smooth curves. {I}}, Internat. J. Math.
  \textbf{7} (1996), no.~5, 573--598.

\bibitem{BY_rationality}
\bysame, \emph{Rationality of moduli spaces of parabolic bundles}, J. London
  Math. Soc. (2) \textbf{59} (1999), no.~2, 461--478.

\bibitem{Bulles}
T.-H. B\"{u}lles, \emph{Motives of moduli spaces on {K}3 surfaces and of
  special cubic fourfolds}, Manuscripta Math. \textbf{161} (2020), no.~1-2,
  109--124.

\bibitem{Chakraborty1}
S.~Chakraborty, \emph{Chow group of 1-cycles of the moduli of parabolic bundles
  over a curve}, arxiv: 1907.13431, 2019.

\bibitem{Chakraborty2}
\bysame, \emph{On {A}bel-{J}acobi maps of moduli of parabolic bundles over a
  curve}, arxiv: 2003.00854, 2020.

\bibitem{ChoeHwang}
I.~Choe and J.-M. Hwang, \emph{Chow group of 1-cycles on the moduli space of
  vector bundles of rank 2 over a curve}, Math. Z. \textbf{253} (2006), no.~2,
  281--293.

\bibitem{cisinski-deglise-integral}
D.-C. Cisinski and F.~D\'{e}glise, \emph{Integral mixed motives in equal
  characteristic}, Doc. Math. (2015), no.~Extra vol.: Alexander S. Merkurjev's
  sixtieth birthday, 145--194.

\bibitem{dB_motive_moduli_vb}
S.~del Ba\~{n}o, \emph{On the {C}how motive of some moduli spaces}, J. Reine
  Angew. Math. \textbf{532} (2001), 105--132.

\bibitem{dB_rk2}
\bysame, \emph{On the motive of moduli spaces of rank two vector bundles over a
  curve}, Compositio Math. \textbf{131} (2002), no.~1, 1--30.

\bibitem{DN}
J.-M. Drezet and M.~S. Narasimhan, \emph{Groupe de {P}icard des
  vari\'{e}t\'{e}s de modules de fibr\'{e}s semi-stables sur les courbes
  alg\'{e}briques}, Invent. Math. \textbf{97} (1989), no.~1, 53--94.

\bibitem{FuWang08}
B.~Fu and C.-L. Wang, \emph{Motivic and quantum invariance under stratified
  {M}ukai flops}, J. Differential Geom. \textbf{80} (2008), no.~2, 261--280.

\bibitem{Salvator-Lie-Ziyu}
L.~Fu, S.~Floccari, and Z.~Zhang, \emph{On the motive of {O}'{G}rady's
  ten-dimensional hyper-{K}\"ahler varieties}, To appear in Communications in
  Contemporary Mathematics, 2019.

\bibitem{FultonBook}
W.~Fulton, \emph{Intersection theory}, second ed., Ergebnisse der Mathematik
  und ihrer Grenzgebiete. 3. Folge. A Series of Modern Surveys in Mathematics
  [Results in Mathematics and Related Areas. 3rd Series. A Series of Modern
  Surveys in Mathematics], vol.~2, Springer-Verlag, Berlin, 1998.

\bibitem{HN}
G.~Harder and M.~S. Narasimhan, \emph{On the cohomology groups of moduli spaces
  of vector bundles on curves}, Math. Ann. \textbf{212} (1974/75), 215--248.

\bibitem{Hitchin}
N.~J. Hitchin, \emph{The self-duality equations on a {R}iemann surface}, Proc.
  London Math. Soc. (3) \textbf{55} (1987), no.~1, 59--126.

\bibitem{Holla}
Y.~I. Holla, \emph{Poincar\'{e} polynomial of the moduli spaces of parabolic
  bundles}, Proc. Indian Acad. Sci. Math. Sci. \textbf{110} (2000), no.~3,
  233--261.

\bibitem{HPL_Higgs}
V.~Hoskins and S.~Pepin~Lehalleur, \emph{On the {V}oevodsky motive of the
  moduli space of higgs bundles on a curve}, To appear in Selecta Mathematica,
  2019.

\bibitem{HPL_formula}
\bysame, \emph{A formula for the {V}oevodsky motive of the moduli stack of
  vector bundles on a curve}, (arxiv: 1809.02150) to appear in Geometry and
  Topology, 2020.

\bibitem{HuberKahn}
A.~Huber and B.~Kahn, \emph{The slice filtration and mixed {T}ate motives},
  Compos. Math. \textbf{142} (2006), no.~4, 907--936.

\bibitem{Huybrechts97JDG}
D.~Huybrechts, \emph{Birational symplectic manifolds and their deformations},
  J. Differential Geom. \textbf{45} (1997), no.~3, 488--513.

\bibitem{HL}
D.~Huybrechts and M.~Lehn, \emph{The geometry of moduli spaces of sheaves},
  second ed., Cambridge Mathematical Library, Cambridge University Press,
  Cambridge, 2010.

\bibitem{illusie-fga}
L.~Illusie, \emph{Grothendieck's existence theorem in formal geometry},
  Fundamental algebraic geometry, Math. Surveys Monogr., vol. 123, Amer. Math.
  Soc., Providence, RI, 2005, With a letter (in French) of Jean-Pierre Serre,
  pp.~179--233.

\bibitem{IyerLewis}
J.~N. Iyer, \emph{The {A}bel-{J}acobi isomorphism for one-cycles on {K}irwan's
  log resolution of the moduli space {$SU_C(2,O_C)$}}, J. Reine Angew. Math.
  \textbf{696} (2014), 1--29, With an appendix by James D. Lewis.

\bibitem{Jannsen}
U.~Jannsen, \emph{Motives, numerical equivalence, and semi-simplicity}, Invent.
  Math. \textbf{107} (1992), no.~3, 447--452.

\bibitem{Jiang19}
Q.~Jiang, \emph{On the {C}how theory of projectivization}, arxiv: 1910.06730,
  2019.

\bibitem{JiangYin}
Z.~Jiang and Q.~Yin, \emph{On the {C}how ring of certain rational cohomology
  tori}, C. R. Math. Acad. Sci. Paris \textbf{355} (2017), no.~5, 571--576.

\bibitem{Kanemitsu-K-equivalence}
A.~Kanemitsu, \emph{Mukai pairs and simple k-equivalence}, arxiv: 1812.05392v1,
  2018.

\bibitem{Kimura-FiniteDimension}
S.-I. Kimura, \emph{Chow groups are finite dimensional, in some sense}, Math.
  Ann. \textbf{331} (2005), no.~1, 173--201.

\bibitem{KS}
A.~King and A.~Schofield, \emph{Rationality of moduli of vector bundles on
  curves}, Indag. Math. (N.S.) \textbf{10} (1999), no.~4, 519--535.

\bibitem{KollarMori}
J.~Koll\'{a}r and S.~Mori, \emph{Birational geometry of algebraic varieties},
  Cambridge Tracts in Mathematics, vol. 134, Cambridge University Press,
  Cambridge, 1998, With the collaboration of C. H. Clemens and A. Corti,
  Translated from the 1998 Japanese original.

\bibitem{LLW-annals}
Y.-P. Lee, H.-W. Lin, and C.-L. Wang, \emph{Flops, motives, and invariance of
  quantum rings}, Ann. of Math. (2) \textbf{172} (2010), no.~1, 243--290.

\bibitem{LiLinPan}
D.~Li, Y.~Lin, and X.~Pan, \emph{A note on 1-cycles on the moduli space of
  rank-2 bundles over a curve}, C. R. Math. Acad. Sci. Paris \textbf{357}
  (2019), no.~2, 209--211.

\bibitem{MVW}
C.~Mazza, V.~Voevodsky, and C.~Weibel, \emph{Lecture notes on motivic
  cohomology}, Clay Mathematics Monographs, vol.~2, American Mathematical
  Society, Providence, RI; Clay Mathematics Institute, Cambridge, MA, 2006.

\bibitem{MS}
V.~B. Mehta and C.~S. Seshadri, \emph{Moduli of vector bundles on curves with
  parabolic structures}, Math. Ann. \textbf{248} (1980), no.~3, 205--239.

\bibitem{Mukai84Invent}
S.~Mukai, \emph{Symplectic structure of the moduli space of sheaves on an
  abelian or {$K3$} surface}, Invent. Math. \textbf{77} (1984), no.~1,
  101--116.

\bibitem{NakajimaPHB}
H.~Nakajima, \emph{Hyper-{K}\"{a}hler structures on moduli spaces of parabolic
  {H}iggs bundles on {R}iemann surfaces}, Moduli of vector bundles ({S}anda,
  1994; {K}yoto, 1994), Lecture Notes in Pure and Appl. Math., vol. 179,
  Dekker, New York, 1996, pp.~199--208.

\bibitem{Ramanan}
S.~Ramanan, \emph{The moduli spaces of vector bundles over an algebraic curve},
  Math. Ann. \textbf{200} (1973), 69--84.

\bibitem{SchollClassical}
A.~J. Scholl, \emph{Classical motives}, Motives ({S}eattle, {WA}, 1991), Proc.
  Sympos. Pure Math., vol.~55, Amer. Math. Soc., Providence, RI, 1994,
  pp.~163--187.

\bibitem{Seshadri_VBAC}
C.~S. Seshadri, \emph{Space of unitary vector bundles on a compact {R}iemann
  surface}, Ann. of Math. (2) \textbf{85} (1967), 303--336.

\bibitem{Simpson_harmonic}
C.~Simpson, \emph{Harmonic bundles on noncompact curves}, J. Amer. Math. Soc.
  \textbf{3} (1990), no.~3, 713--770.

\bibitem{Simpson}
\bysame, \emph{The {H}odge filtration on nonabelian cohomology}, Algebraic
  geometry---{S}anta {C}ruz 1995, Proc. Sympos. Pure Math., vol.~62, Amer.
  Math. Soc., Providence, RI, 1997, pp.~217--281.

\bibitem{Thaddeus_pairs}
M.~Thaddeus, \emph{Stable pairs, linear systems and the {V}erlinde formula},
  Invent. Math. \textbf{117} (1994), no.~2, 317--353.

\bibitem{Thaddeus_VGIT}
\bysame, \emph{Geometric invariant theory and flips}, J. Amer. Math. Soc.
  \textbf{9} (1996), no.~3, 691--723.

\bibitem{Thaddeus_var_PHiggs}
\bysame, \emph{Variation of moduli of parabolic {H}iggs bundles}, J. Reine
  Angew. Math. \textbf{547} (2002), 1--14.

\bibitem{VoevodskyBookChapter}
V.~Voevodsky, \emph{Triangulated categories of motives over a field}, Cycles,
  transfers, and motivic homology theories, Ann. of Math. Stud., vol. 143,
  Princeton Univ. Press, Princeton, NJ, 2000, pp.~188--238.

\bibitem{VoevodskyCancellation}
\bysame, \emph{Cancellation theorem}, Doc. Math. (2010), no.~Extra vol.: Andrei
  A. Suslin sixtieth birthday, 671--685.

\bibitem{wildeshaus}
J.~Wildeshaus, \emph{On the interior motive of certain {S}himura varieties: the
  case of {P}icard surfaces}, Manuscripta Math. \textbf{148} (2015), no.~3-4,
  351--377.

\bibitem{Yokogawa93}
K.~Yokogawa, \emph{Compactification of moduli of parabolic sheaves and moduli
  of parabolic {H}iggs sheaves}, J. Math. Kyoto Univ. \textbf{33} (1993),
  no.~2, 451--504.

\bibitem{Yokogawa}
\bysame, \emph{Infinitesimal deformation of parabolic {H}iggs sheaves},
  Internat. J. Math. \textbf{6} (1995), no.~1, 125--148.

\end{thebibliography}

\medskip \medskip

\noindent{Radboud University, IMAPP, PO Box 9010, 6500 GL Nijmegen, The Netherlands} 

\medskip \noindent{\texttt{fu@math.univ-lyon1.fr, v.hoskins@math.ru.nl, simon.pepinlehalleur@ru.nl}}

\end{document}